\documentclass[oneside,french,english]{amsart}
\usepackage[T1]{fontenc}
\usepackage[latin9]{inputenc}
\usepackage{geometry}
\usepackage{babel}
\makeatletter
\addto\extrasfrench{%
   \providecommand{\fg}{\ifdim\lastskip>\z@\unskip\fi~\frqq}%
}

\makeatother
\usepackage{units}
\usepackage{mathrsfs}
\usepackage{amsthm}
\usepackage{amstext}
\usepackage{amssymb}
\usepackage{stmaryrd}
\usepackage{graphicx}
\usepackage{esint}
\usepackage{xargs}[2008/03/08]
\usepackage[unicode=true,pdfusetitle,
 bookmarks=true,bookmarksnumbered=false,bookmarksopen=false,
 breaklinks=false,pdfborder={0 0 1},backref=false,colorlinks=false]
 {hyperref}
\usepackage{breakurl}

\makeatletter
\numberwithin{equation}{section}
\numberwithin{figure}{section}
\theoremstyle{plain}
\newtheorem{thm}{\protect\theoremname}[section]
  \theoremstyle{definition}
  \newtheorem{defn}[thm]{\protect\definitionname}
  \theoremstyle{remark}
  \newtheorem{rem}[thm]{\protect\remarkname}
  \theoremstyle{plain}
  \newtheorem{prop}[thm]{\protect\propositionname}
  \theoremstyle{plain}
  \newtheorem{lem}[thm]{\protect\lemmaname}
  \theoremstyle{plain}
  \newtheorem{cor}[thm]{\protect\corollaryname}
  \theoremstyle{plain}
  \newtheorem{fact}[thm]{\protect\factname}
  \theoremstyle{plain}
  \newtheorem*{fact*}{\protect\factname}

\makeatother

  \addto\captionsenglish{\renewcommand{\corollaryname}{Corollary}}
  \addto\captionsenglish{\renewcommand{\definitionname}{Definition}}
  \addto\captionsenglish{\renewcommand{\factname}{Fact}}
  \addto\captionsenglish{\renewcommand{\lemmaname}{Lemma}}
  \addto\captionsenglish{\renewcommand{\propositionname}{Proposition}}
  \addto\captionsenglish{\renewcommand{\remarkname}{Remark}}
  \addto\captionsenglish{\renewcommand{\theoremname}{Theorem}}
  \addto\captionsfrench{\renewcommand{\corollaryname}{Corollaire}}
  \addto\captionsfrench{\renewcommand{\definitionname}{Définition}}
  \addto\captionsfrench{\renewcommand{\factname}{Fait}}
  \addto\captionsfrench{\renewcommand{\lemmaname}{Lemme}}
  \addto\captionsfrench{\renewcommand{\propositionname}{Proposition}}
  \addto\captionsfrench{\renewcommand{\remarkname}{Remarque}}
  \addto\captionsfrench{\renewcommand{\theoremname}{Théorème}}
  \providecommand{\corollaryname}{Corollary}
  \providecommand{\definitionname}{Definition}
  \providecommand{\factname}{Fact}
  \providecommand{\lemmaname}{Lemma}
  \providecommand{\propositionname}{Proposition}
  \providecommand{\remarkname}{Remark}
\providecommand{\theoremname}{Theorem}

\begin{document}

\global\long\def\eps{\varepsilon}

\global\long\def\pain{{\displaystyle \left(P_{j}\right)}}

\global\long\def\tt#1{\mathtt{#1}}

\global\long\def\tx#1{\mathrm{#1}}

\global\long\def\cal#1{\mathcal{#1}}

\global\long\def\scrib#1{\mathscr{#1}}

\global\long\def\frak#1{\mathfrak{#1}}

\global\long\def\math#1{{\displaystyle #1}}


\global\long\def\pare#1{\Big({\displaystyle #1}\Big)}

\global\long\def\acc#1{\left\{  #1\right\}  }

\global\long\def\cro#1{\left[#1\right]}

\newcommandx\dd[1][usedefault, addprefix=\global, 1=]{\tx d#1}

\global\long\def\DD#1{\tx D#1}

\global\long\def\ppp#1#2{\frac{\partial#1}{\partial#2}}

\global\long\def\ddd#1#2{\frac{\tx d#1}{\tx d#2}}

\global\long\def\norm#1{\left\Vert #1\right\Vert }

\global\long\def\abs#1{\left|#1\right|}

\global\long\def\ps#1{\left\langle #1\right\rangle }

\global\long\def\crocro#1{\left\llbracket #1\right\rrbracket }

\global\long\def\ord#1{\tx{ord}\left(#1\right)}


\global\long\def\wc{\mathbb{C}}

\global\long\def\ww#1{\mathbb{#1}}

\global\long\def\wr{\mathbb{R}}

\global\long\def\wn{\mathbb{N}}

\global\long\def\wn{\mathbb{Z}}

\global\long\def\wq{\mathbb{Q}}

\global\long\def\wcc{\left(\ww C^{2},0\right)}

\global\long\def\wccc{\left(\ww C^{3},0\right)}

\global\long\def\wcn{\left(\ww C^{n},0\right)}

\global\long\def\cn{\ww C^{n},0}


\global\long\def\germ#1{\ww C\left\{  #1\right\}  }

\global\long\def\germinv#1{\ww C\left\{  #1\right\}  ^{\times} }

\global\long\def\form#1{\ww C\left\llbracket #1\right\rrbracket }

\global\long\def\forminv#1{\ww C\left\llbracket #1\right\rrbracket ^{\times}}

\global\long\def\maxid{\mathfrak{m}}

\newcommandx\jet[1][usedefault, addprefix=\global, 1=k]{\mbox{\ensuremath{\tx J}}_{#1}}

\global\long\def\hol#1{\mathcal{O}\left(#1\right)}

\newcommandx\formdiff[1][usedefault, addprefix=\global, 1=]{\widehat{\Omega^{#1}}}


\global\long\def\pp#1{\frac{\partial}{\partial#1}}

\newcommandx\der[2][usedefault, addprefix=\global, 1=\mathbf{x}, 2=]{\tx{Der}_{#2}\left(\form{#1}\right)}

\newcommandx\vf[1][usedefault, addprefix=\global, 1={\ww C^{3},0}]{\chi\left(#1\right)}

\newcommandx\fvf[1][usedefault, addprefix=\global, 1={\ww C^{3},0}]{\widehat{\chi}\left(#1\right)}

\newcommandx\sfvf[1][usedefault, addprefix=\global, 1={\ww C^{3},0}]{\widehat{\chi}_{\omega}\left(#1\right)}

\newcommandx\fisot[2][usedefault, addprefix=\global, 1=\tx{fib}]{\widehat{\tx{Isot}}_{#1}\left(#2\right)}

\global\long\def\isotsect#1#2#3{\tx{Isot}_{\tx{fib}}\left(#1,\cal S_{#2,#3};\tx{Id}\right)}

\global\long\def\ynorm{Y_{\tx{norm}}}

\global\long\def\lie#1{\cal L_{\math{#1}}}


\newcommandx\fdiff[1][usedefault, addprefix=\global, 1={\ww C^{3},0}]{\tx{Diff}_{\tx{fib}}\left(#1\right)}

\newcommandx\diff[1][usedefault, addprefix=\global, 1={\ww C^{3},0}]{\tx{Diff}\left(#1\right)}

\newcommandx\diffid[1][usedefault, addprefix=\global, 1={\ww C^{3},0}]{\tx{Diff}\left(#1;\tx{Id}\right)}

\newcommandx\fdiffid[1][usedefault, addprefix=\global, 1={\ww C^{3},0}]{\tx{Diff}_{\tx{fib}}\left(#1;\tx{Id}\right)}

\newcommandx\diffsect[2][usedefault, addprefix=\global, 1=\theta, 2=\eta]{\tx{Diff}_{\tx{fib}}\left(\cal S_{#1,#2};\tx{Id}\right)}

\newcommandx\diffform[1][usedefault, addprefix=\global, 1={\ww C^{3},0}]{\widehat{\tx{Diff}}\left(#1\right)}

\newcommandx\fdiffform[1][usedefault, addprefix=\global, 1={\ww C^{3},0}]{\widehat{\tx{Diff}}_{\tx{fib}}\left(#1\right)}

\newcommandx\fdiffformid[1][usedefault, addprefix=\global, 1={\ww C^{3},0}]{\widehat{\tx{Diff}}_{\tx{fib}}\left(#1;\tx{Id}\right)}

\newcommandx\gid[1][usedefault, addprefix=\global, 1=k]{\cal G_{\tx{Id}}^{\left(#1\right)}}

\newcommandx\diffformid[1][usedefault, addprefix=\global, 1={\ww C^{3},0}]{\widehat{\tx{Diff}}\left(#1;\tx{Id}\right)}

\newcommandx\sdiff[1][usedefault, addprefix=\global, 1={\ww C^{3},0}]{\tx{Diff}_{\omega}\left(#1\right)}

\newcommandx\sdiffid[1][usedefault, addprefix=\global, 1={\ww C^{3},0}]{\tx{Diff}_{\omega}\left(#1;\tx{Id}\right)}

\newcommandx\sdiffform[1][usedefault, addprefix=\global, 1={\ww C^{3},0}]{\widehat{\tx{Diff}}_{\omega}\left(#1\right)}

\newcommandx\sdiffformid[1][usedefault, addprefix=\global, 1={\ww C^{3},0}]{\widehat{\tx{Diff}}_{\omega}\left(#1;\tx{Id}\right)}


\global\long\def\sect#1#2#3{S\left(#1,#2,#3\right)}

\global\long\def\germsect#1#2{\cal S_{#1,#2}}

\global\long\def\asympsect#1#2{\cal{AS}_{#1,#2}}

\global\long\def\fol#1{\mathcal{F}_{#1}}


\global\long\def\rrel{\mathcal{R}}

\global\long\def\surj{\twoheadrightarrow}

\global\long\def\inj{\hookrightarrow}

\global\long\def\bij{\simeq}

\global\long\def\quotient#1#2{\bigslant{#1}{#2}}


\global\long\def\res#1{\tx{res}\left(#1\right)}

\global\long\def\param{\cal P}

\global\long\def\po{\cal P_{0}}

\global\long\def\pfib{\cal P_{\tx{fib}}}

\global\long\def\pw{\cal P_{\omega}}

\global\long\def\porb{\cal P_{\tx{orb}}}

\global\long\def\ls#1{\Gamma_{#1\lambda}}

\global\long\def\lsp#1{\Gamma'_{#1\lambda}}

\global\long\def\lsr#1#2{\Gamma_{#1\lambda}\left(#2\right)}

\global\long\def\sn{\mathbb{\cal{SN}}}

\global\long\def\sndiag{\mathbb{\cal{SN}}_{\tx{diag}}}

\global\long\def\sndiagnd{\mathbb{\cal{SN}}_{\tx{diag,nd}}}

\global\long\def\snodiag{\cal{SN}_{\tx{diag},0}}

\global\long\def\sns{\mathbb{\cal{SN}}_{\omega}}

\global\long\def\snsdiag{\mathbb{\cal{SN}}_{\mbox{\ensuremath{\tx{diag}}},\omega}}

\global\long\def\snsnd{\mathbb{\cal{SN}}_{\tx{snd}}}

\global\long\def\sno{\cal{SN}_{0}}

\global\long\def\snfib{\cal{SN}_{\tx{fib}}}

\global\long\def\snofib{\cal{SN}_{\tx{fib},0}}

\global\long\def\snoid{\cal{SN}_{\tx{Id},0}}

\global\long\def\fsn{\mathbb{\widehat{\cal{SN}}}}

\global\long\def\fsns{\widehat{\mathbb{\cal{SN}}}_{\omega}}

\global\long\def\fsndiag{\mathbb{\widehat{\cal{SN}}}_{\tx{diag}}}

\global\long\def\fsnnd{\mathbb{\widehat{\cal{SN}}}_{\tx{nd}}}

\global\long\def\fsnfibnd{\mathbb{\widehat{\cal{SN}}}_{\tx{fib,nd}}}

\global\long\def\fsndiagnd{\mathbb{\widehat{\cal{SN}}}_{\tx{diag,nd}}}

\global\long\def\fsno{\cal{\widehat{SN}}_{*}}

\global\long\def\fsnfib{\cal{\widehat{SN}}_{\tx{fib}}}

\global\long\def\mfibid{\mathscr{M}_{\tx{fib,Id}}}

\global\long\def\mfib{\mathscr{M}_{\tx{fib}}}

\global\long\def\mspaceid{\mathscr{M}_{\tx{Id}}}

\global\long\def\ms{\mathscr{M}_{\omega}}

\newcommand{\bigslant}[2]{{\raisebox{.3em}{$#1$}\left/\raisebox{-.3em}{$#2$}\right.}}

\title[Sectorial normalization of doubly-resonant saddle-nodes]{Sectorial analytic normalization for a class of doubly-resonant
saddle-node vector fields in $\left(\ww C^{3},0\right)$}

\author{Amaury Bittmann}

\address{IRMA, Université de Strasbourg, 7 rue René Descartes, 67084 Strasbourg
Cedex, France }

\email{\href{mailto:bittmann@math.unistra.fr}{bittmann@math.unistra.fr}}
\begin{abstract}
In this work, following \cite{bittmann1}, we consider germs of analytic
singular vector fields in $\ww C^{3}$ with an isolated and doubly-resonant
singularity of saddle-node type at the origin. Such vector fields
come from irregular two-dimensional differential systems with two
opposite non-zero eigenvalues, and appear for instance when studying
the irregular singularity at infinity in Painlevé equations $\left(P_{j}\right)_{j=I,\dots,V}$
for generic values of the parameters. Under suitable assumptions,
we prove a theorem of analytic normalization over sectorial domains,
analogous to the classical one due to Hukuhara-Kimura-Matuda \cite{HKM}
for saddle-nodes in $\ww C^{2}$. We also prove that the normalizing
map is essentially unique and \emph{weakly Gevrey-1 summable}.
\end{abstract}

\keywords{Painlevé equations, singular vector field, irregular singularity,
resonant singularity, Gevrey-1 summability}

\maketitle

\section{Introduction}

~

As in \cite{bittmann1}, we consider (germs of) singular vector fields
$Y$ in $\ww C^{3}$ which can be written in appropriate coordinates
$\left(x,\mathbf{y}\right):=\left(x,y_{1},y_{2}\right)$ as 
\begin{eqnarray}
Y & = & x^{2}\pp x+\Big(-\lambda y_{1}+F_{1}\left(x,\mathbf{y}\right)\Big)\pp{y_{1}}+\Big(\lambda y_{2}+F_{2}\left(x,\mathbf{y}\right)\Big)\pp{y_{2}}\,\,\,\,\,,\label{eq: intro}
\end{eqnarray}
where $\lambda\in\ww C^{*}$ and $F_{1},\, F_{2}$ are germs of holomorphic
functions in $\left(\ww C^{3},0\right)$ of homogeneous valuation
(order) at least two. They represent irregular two-dimensional differential
systems having two opposite non-zero eigenvalues:
\[
\begin{cases}
{\displaystyle x^{2}\ddd{y_{1}\left(x\right)}x=-\lambda y_{1}\left(x\right)+F_{1}\left(x,\mathbf{y}\left(x\right)\right)}\\
{\displaystyle x^{2}\ddd{y_{2}\left(x\right)}x=\lambda y_{2}\left(x\right)+F_{2}\left(x,\mathbf{y}\left(x\right)\right)} & .
\end{cases}
\]
 These we call doubly-resonant vector fields of saddle-node type (or
simply \textbf{doubly-resonant saddle-nodes}). We will impose more
(non-generic) conditions in the sequel. The motivation for studying
such vector fields is at least of two types.
\begin{enumerate}
\item There are two independent resonance relations between the eigenvalues
(here $0$, $-\lambda$ and $\lambda$): we generalize then the study
in \cite{MR82,MR83}. More generally, this work is aimed at understanding
singularities of vector fields in $\ww C^{3}$. According to a theorem
of resolution of singularities in dimension less than three in \cite{mcquillan2013almost},
there exists a list of ``final models'' for singularities (\emph{log-canonical})\emph{
}obtained after a finite procedure of \emph{weighted blow-ups} for
three dimensional singular analytic vector fields. In this list, we
find in particular doubly-resonant saddles-nodes, as those we are
interested in. In dimension $2$, these final models have been intensively
studied (for instance by Martinet, Ramis, Ecalle, Ilyashenko, Teyssier,
...) from the view point of both formal and analytic classification
(some important questions remain unsolved, though). In dimension $3$,
the problems of formal and analytic classification are still open
questions, although Stolovitch has performed such a classification
for 1-resonant vector fields in any dimension \cite{Stolo}. The presence
of two kinds of resonance relations brings new difficulties.
\item Our second main motivation is the study of the irregular singularity
at infinity in Painlevé equations$\left(P_{j}\right)_{j=I,\dots,V}$,
for generic values of the parameters (\emph{cf.} \cite{Yoshida85}).
These equations were discovered by Paul Painlevé~\cite{Painleve}
because the only movable singularities of the solutions are poles
(the so-called \emph{Painlevé property}). Their study has become a
rich domain of research since the important work of Okamoto~\cite{OkamotoSpace}.
The fixed singularities of the Painlevé equations, and more particularly
those at infinity, where notably investigated by Boutroux with his
famous \emph{tritronquées} solutions \cite{Boutroux13}. Recently,
several authors provided more complete information about such singularities,
studying ``quasi-linear Stokes phenomena'' and also giving connection
formulas; we refer to the following (non-exhaustive) sources \cite{Joshi,Kapaev,KapaevKitaev,KitaevJosi,ClarksonMcLeod,Costin15}.
Stokes coefficients are invariant under local changes of analytic
coordinates, but do not form a complete invariant of the vector field.
To the best of our knowledge there currently does not exist a general
analytic classification for doubly-resonant saddle-nodes. Such a classification
would provide a new framework allowing to analyze Stokes phenomena
in that class of singularities.
\end{enumerate}
In this paper we provide a theorem of analytic normalization over
sectorial domain \emph{(à la }Hukuhara-Kimura-Matuda \cite{HKM} for
saddle-nodes in $\left(\ww C^{2},0\right)$) for a specific class
(to be defined later on) of doubly-resonant saddle-nodes which contains
the Painlevé case. In~a forthcoming paper we use this theorem in
order to provide a complete analytic classification for this class
of vector fields, based on the ideas in the important works \cite{MR82,MR83,Stolo}.\bigskip{}

In \cite{Yoshida84,Yoshida85} Yoshida shows that doubly-resonant
saddle-nodes arising from the compactification of Painlevé equations
$\left(P_{j}\right)_{j=I,\dots,V}$ (for generic values for the parameters)
are conjugate to vector fields of the form: 
\begin{eqnarray}
 & Z= & x^{2}\pp x+\pare{-\left(1+\gamma y_{1}y_{2}\right)+a_{1}x}y_{1}\pp{y_{1}}\nonumber \\
 &  & +\pare{1+\gamma y_{1}y_{2}+a_{2}x}y_{2}\pp{y_{2}}\,\,\,\,\,,\label{eq: normal form Yoshida-1}
\end{eqnarray}
with $\gamma\in\ww C^{*}$ and $\left(a_{1},a_{2}\right)\in\ww C^{2}$
such that $a_{1}+a_{2}=1$. One should notice straight away that this
``conjugacy'' does not agree with what is traditionally (in particular
in this paper) meant by conjugacy, for Yoshida's transform $\Psi\left(x,\mathbf{y}\right)=\left(x,\psi_{1}\left(x,\mathbf{y}\right),\psi_{2}\left(x,\mathbf{y}\right)\right)$
takes the form
\begin{eqnarray}
\psi_{i}\left(x,\mathbf{y}\right) & = & y_{i}\left(1+\sum_{\substack{\left(k_{0},k_{1},k_{2}\right)\in\ww N^{3}\\
k_{1}+k_{2}\geq1
}
}\frac{q_{i,\mathbf{k}}\left(x\right)}{x^{k_{0}}}y_{1}^{k_{1}+k_{0}}y_{2}^{k_{1}+k_{0}}\right)\,\,\,\,\,,\label{eq: Yoshida-1}
\end{eqnarray}
where each $q_{i,\mathbf{k}}$ is formal power series although $x$
\emph{appears with negative exponents}. This expansion may not even
be a formal Laurent series. It is, though, the asymptotic expansion
along $\left\{ x=0\right\} $ of a function analytic in a domain 
\begin{eqnarray*}
 & \Big\{\left(x,\mathbf{z}\right)\in\text{\ensuremath{S\times\mathbf{D}\left(0,\mathbf{r}\right)\mid\abs{z_{1}z_{2}}<\nu\abs x}}\Big\}
\end{eqnarray*}
for some small $\nu>0$, where $S$ is a sector of opening greater
than $\pi$ with vertex at the origin and $\mathbf{D}\left(0,\mathbf{r}\right)$
is a polydisc of small poly-radius $\mathbf{r}=\left(r_{1},r_{2}\right)\in\left(\ww R_{>0}\right)^{2}$.
Moreover the $\left(q_{i,\mathbf{k}}\left(x\right)\right)_{i,\mathbf{k}}$
are actually Gevrey-1 power series. The drawback here is that the
transforms are convergent on regions so small that taken together
they cannot cover an entire neighborhood of the origin in $\ww C^{3}$
(which seems to be problematic to obtain an analytic classification
\emph{à la }Martinet-Ramis).

\bigskip{}

Several authors studied the problem of convergence of formal transformations
putting vector fields as in $\left(\mbox{\ref{eq: intro}}\right)$
into ``normal forms''. Shimomura, improving on a result of Iwano
\cite{Iwano}, shows in \cite{Shimo} that analytic doubly-resonant
saddle-nodes satisfying more restrictive conditions are conjugate
(formally and over sectors) to vector fields of the form
\begin{eqnarray*}
 &  & x^{2}\pp x+\left(-\lambda+a_{1}x\right)y_{1}\pp{y_{1}}+\left(\lambda+a_{2}x\right)y_{2}\pp{y_{2}}\,\,\,\,
\end{eqnarray*}
\emph{via} a diffeomorphism whose coefficients have asymptotic expansions
as $x\rightarrow0$ in sectors of opening greater than $\pi$. 

Stolovitch then generalized this result to any dimension in \cite{Stolo}.
More precisely, Stolovitch's work offers an analytic classification
of vector fields in $\ww C^{n+1}$ with an irregular singular point,
without further hypothesis on eventual additional resonance relations
between eigenvalues of the linear part. However, as Iwano and Shimomura
did, he needed to impose other assumptions, among which the condition
that the restriction of the vector field to the invariant hypersurface
$\acc{x=0}$ is a linear vector field. In \cite{DeMaesschalck}, the
authors obtain a \emph{Gevrey-1 summable} ``normal form'', though
not as simple as Stolovitch's one and not unique \emph{a priori},
but for more general kind of vector field with one zero eigenvalue.
However, the same assumption on hypersurface $\acc{x=0}$ is required
(the restriction is a linear vector field). Yet from \cite{Yoshida85}
stems the fact that this condition is not met in the case of Painlevé
equations $\left(P_{j}\right)_{j=I,\dots,V}$. 

In comparison, we merely ask here that the restricted vector field
be orbitally linearizable (see Definition \ref{def: asympt hamil}),
\emph{i.e.} the foliation\emph{ }induced by $Y$ on $\acc{x=0}$ (and
not the vector field $Y_{|\acc{x=0}}$ itself) be linearizable. The
fact that this condition is fulfilled by the singularities of Painlevé
equations formerly described is well-known.

\subsection{Scope of the paper}

~

The action of local analytic~/~formal diffeomorphisms $\Psi$ fixing
the origin on local holomorphic vector fields $Y$ of type $\left(\mbox{\ref{eq: intro}}\right)$
by change of coordinates is given by
\begin{eqnarray*}
\Psi_{*}Y & := & \mathrm{D}\Psi\left(Y\right)\circ\Psi^{-1}~.
\end{eqnarray*}
In \cite{bittmann1} we performed the formal classification of such
vector fields by exhibiting an explicit universal family of vector
fields for the action of formal changes of coordinates at $0$ (called
a family of normal forms). Such a result seems currently out of reach
in the analytic category: it is unlikely that an explicit universal
family for the action of local analytic changes of coordinates be
described anytime soon. If we want to describe the space of equivalent
classes (of germs of a doubly-resonant saddle-node under local analytic
changes of coordinates) with same formal normal form, we therefore
need to find a complete set of invariants which is of a different
nature. We call \textbf{moduli space} this quotient space and would
like to give it a (non-trivial) presentation based on functional invariants
\emph{à la} Martinet-Ramis \cite{MR82,MR83}. 

The main ingredient to obtain such analytic invariant is to prove
first the existence of analytic sectorial normalizing maps (over a
pair of opposite ``wide'' sectors of opening greater than $\pi$
whose union covers a full punctured neighborhood of $\left\{ x=0\right\} $).
This is the main result of the present paper. We have not been able
to perform this normalization in such a generality, and only deal
here with $x$-fibered local analytic conjugacies acting on vector
fields of the form (\ref{eq: intro}) with some additional assumptions
detailed further down (see Definitions \ref{def: drsn}, \ref{def: non-deg}
and \ref{def: asympt hamil}). Importantly, these hypothesis are met
in the case of Painlevé equations mentioned above. \medskip{}

Our approach has some geometric flavor, since we avoid the use of
fixed-point methods altogether to establish the existence of the normalizing
maps, and generalize instead the approach of Teyssier~\cite{teyssier2004equation,Teyssier03}
relying on path-integration of well-chosen $1$-forms (following Arnold's
method of characteristics~\cite{Arnold}).

As a by-product of this normalization we deduce that the normalizing
sectorial diffeomorphisms are \emph{weakly} Gevrey-$1$ asymptotic
to the normalizing formal power series of \cite{bittmann1}, retrospectively
proving their \emph{weak $1$-summability} (see subsection \ref{sub:Weak-Gevrey-1-power}
for definition). When the vector field additionally supports a symplectic
transverse structure (which is again the case of Painlevé equations)
we prove that the (essentially unique) sectorial normalizing map is
realized by a transversally symplectic diffeomorphism.

\subsection{Definitions and main results}

~

To state our main results we need to introduce some notations and
nomenclature. 
\begin{itemize}
\item For $n\in\ww N_{>0}$, we denote by $\left(\ww C^{n},0\right)$ an
(arbitrary small) open neighborhood of the origin in $\ww C^{n}$.
\item We denote by $\germ{x,\mathbf{y}}$, with $\mathbf{y}=\left(y_{1},y_{2}\right)$,
the $\ww C$-algebra of germs of holomorphic functions at the origin
of $\ww C^{3}$, and by $\germ{x,\mathbf{y}}^{\times}$ the group
of invertible elements for the multiplication (also called units),
\emph{i.e. }elements $U$ such that $U\left(0\right)\neq0$.
\item $\vf$ is the Lie algebra of germs of singular holomorphic vector
fields at the origin $\ww C^{3}$. Any vector field in $\vf$ can
be written as 
\[
Y={\displaystyle b\left(x,y_{1},y_{2}\right)\pp x+b_{1}\left(x,y_{1},y_{2}\right)\pp{y_{1}}+b_{2}\left(x,y_{1},y_{2}\right)\pp{y_{2}}}
\]
with $b,b_{1},b_{2}\in\germ{x,y_{1},y_{2}}$ vanishing at the origin.
\item $\diff$ is the group of germs of a holomorphic diffeomorphism fixing
the origin of $\ww C^{3}$. It acts on $\vf$ by conjugacy: for all
\[
\left(\Phi,Y\right)\in\diff\times\vf
\]
we define the push-forward of $Y$ by $\Phi$ by
\begin{equation}
\Phi_{*}\left(Y\right):=\left(\mbox{D}\Phi\cdot Y\right)\circ\Phi^{-1}\qquad,\label{eq: push forward intro}
\end{equation}
where $\mbox{D}\Phi$ is the Jacobian matrix of $\Phi$.
\item $\fdiff$ is the subgroup of $\diff$ of fibered diffeomorphisms preserving
the $x$-coordinate, \emph{i.e. }of the form $\left(x,\mathbf{y}\right)\mapsto\left(x,\phi\left(x,\mathbf{y}\right)\right)$.
\item We denote by $\fdiffid$ the subgroup of $\fdiff$ formed by diffeomorphisms
tangent to the identity.
\end{itemize}
All these concepts have \emph{formal} analogues, where we only suppose
that the objects are defined with formal power series, not necessarily
convergent near the origin. 
\begin{defn}
\label{def: drsn}A \textbf{diagonal doubly-resonant saddle-node}
is a vector field $Y\in\vf$ of the form 
\begin{eqnarray*}
Y & = & x^{2}\pp x+\Big(-\lambda y_{1}+F_{1}\left(x,\mathbf{y}\right)\Big)\pp{y_{1}}+\Big(\lambda y_{2}+F_{2}\left(x,\mathbf{y}\right)\Big)\pp{y_{2}}\,\,\,\,\,,
\end{eqnarray*}
with $\lambda\in\ww C^{*}$ and $F_{1},F_{2}\in\germ{x,\mathbf{y}}$
of order at least two. We denote by $\sndiag$ the set of such vector
fields.\end{defn}
\begin{rem}
One can also define the foliation associate to a diagonal doubly-resonant
saddle-node in a geometric way. A vector field $Y\in\vf$ is orbitally
equivalent to a diagonal doubly-resonant saddle-node $\Big($\emph{i.e.}
$Y$ is conjugate to some $VX$, where $V\in\germ{x,\mathbf{y}}^{\times}$
and $X\in\snfib$$\Big)$ if and only if the following conditions
hold:
\begin{enumerate}
\item $\tx{Spec}\left(\tx D_{0}Y\right)=\acc{0,-\lambda,\lambda}$ with
$\lambda\neq0$;
\item there exists a germ of irreducible analytic hypersurface $\mathscr{H}_{0}=\acc{S=0}$
which is transverse to the eigenspace $E_{0}$ (corresponding to the
zero eigenvalue) at the origin, and which is stable under the flow
of $Y$;
\item $\cal L_{Y}\left(S\right)=U.S^{2}$, where $\cal L_{Y}$ is the Lie
derivative of $Y$ and $U\in\germ{x,\mathbf{y}}^{\times}$. 
\end{enumerate}
\end{rem}
By Taylor expansion up to order $1$ with respect to $\mathbf{y}$,
given a vector field $Y\in\sndiag$ written as in (\ref{eq: intro})
we can consider the associate 2-dimensional system:
\begin{equation}
x^{2}\ddd{\mathbf{y}}x=\mathbf{\alpha}\left(x\right)+\mathbf{A}\left(x\right)\mathbf{y}\left(x\right)+\mathbf{F}\left(x,\mathbf{y}\left(x\right)\right)\qquad,\label{eq: system doubly resonant saddle node}
\end{equation}
with ${\displaystyle \mathbf{y}}=\left(y_{1},y_{2}\right)$, such
that the following conditions hold: 
\begin{itemize}
\item ${\displaystyle \alpha\left(x\right)=\left(\begin{array}{c}
\alpha_{1}\left(x\right)\\
\alpha_{2}\left(x\right)
\end{array}\right)},$ with $\alpha_{1},\alpha_{2}\in\germ x$ and ${\displaystyle {\displaystyle \alpha_{1},\alpha_{2}\in\tx O\left(x^{2}\right)}}$
\item ${\displaystyle \mathbf{A}\left(x\right)\in\mbox{Mat}_{2,2}\left(\germ x\right)}$
with ${\displaystyle \mathbf{A}\left(0\right)=\tx{diag}\left(-\lambda,\lambda\right)}$,
$\lambda\in\ww C^{*}$
\item ${\displaystyle \mathbf{F}\left(x,\mathbf{y}\right)=\left(\begin{array}{c}
F_{1}\left(x,\mathbf{y}\right)\\
F_{2}\left(x,\mathbf{y}\right)
\end{array}\right)}$, with ${\displaystyle F_{1},F_{2}\in\germ{x,\mathbf{y}}}$ and ${\displaystyle F_{1},F_{2}\in\tx O\left(\norm{\mathbf{y}}^{2}\right)}$.
\end{itemize}
Based on this expression, we state:
\begin{defn}
\label{def: non-deg} The \textbf{residue} of $Y\in\sndiag$ is the
complex number
\[
{\displaystyle \tx{res}\left(Y\right):=\left(\frac{\mbox{Tr}\left(\mathbf{A}\left(x\right)\right)}{x}\right)_{\mid x=0}}\qquad.
\]
We say that $Y$ is\textbf{ non-degenerate }(\emph{resp. }\textbf{strictly
non-degenerate) }if $\tx{res}\left(Y\right)\notin\ww Q_{\leq0}$ (\emph{resp.
}$\Re\left(\tx{res}\left(Y\right)\right)>0$).\end{defn}
\begin{rem}
It is obvious that there is an action of $\fdiff[\ww C^{3},0,\tx{Id}]$
on $\sndiag$. The residue is an invariant of each orbit of $\snfib$
under the action of $\fdiff[\ww C^{3},0,\tx{Id}]$ by conjugacy (see~\cite{bittmann1}).
\end{rem}
The main result of \cite{bittmann1} can now be stated as follows:
\begin{thm}
\cite{bittmann1}\label{thm: forme normalel formelle} Let $Y\in\sndiag$
be non-degenerate. Then there exists a unique formal fibered diffeomorphism
$\hat{\Phi}$ tangent to the identity such that: 
\begin{eqnarray}
\hat{\Phi}_{*}\left(Y\right) & = & x^{2}\pp x+\left(-\lambda+a_{1}x+c_{1}\left(y_{1}y_{2}\right)\right)y_{1}\pp{y_{1}}\nonumber \\
 &  & +\left(\lambda+a_{2}x+c_{2}\left(y_{1}y_{2}\right)\right)y_{2}\pp{y_{2}}\,\,\,\,,\label{eq: fibered normal form-1-2}
\end{eqnarray}
where $\lambda\in\ww C^{*}$, ${\displaystyle c_{1},c_{2}\in\form v}$
are formal power series in $v=y_{1}y_{2}$ without constant term and
$a_{1},a_{2}\in\ww C$ are such that ${\displaystyle a_{1}+a_{2}=\tx{res}\left(Y\right)\in\ww C\backslash\ww Q_{\leq0}}$.\end{thm}
\begin{defn}
The vector field obtained in (\ref{eq: fibered normal form-1-2})
is called the \textbf{formal normal form }of $Y$. The formal fibered
diffeomorphism $\hat{\Phi}$ is called the \textbf{formal normalizing
map} of $Y$.
\end{defn}
The above result is valid for formal objects, without considering
problems of convergence. The first main result in this work states
that this formal normalizing map is analytic in sectorial domains,
under some additional assumptions that we are now going to precise.
\begin{defn}
\label{def: asympt hamil}~
\begin{itemize}
\item We say that a germ of a vector field $X$ in $\left(\ww C^{2},0\right)$
is \textbf{orbitally linear} if 
\[
X=U\left(\mathbf{y}\right)\left(\lambda_{1}y_{1}\pp{y_{1}}+\lambda_{2}y_{2}\pp{y_{2}}\right)\,\,,
\]
for some ${\displaystyle U\left(\mathbf{y}\right)\in\germ{\mathbf{y}}^{\times}}$
and $\left(\lambda_{1},\lambda_{2}\right)\in\ww C^{2}$.
\item We say that a germ of vector field $X$ in $\left(\ww C^{2},0\right)$
is analytically (\emph{resp. formally}) \textbf{orbitally linearizable}
if $X$ is analytically (\emph{resp.} formally) conjugate to an orbitally
linear vector field.
\item We say that a diagonal doubly-resonant saddle-node $Y\in\sndiag$
is \textbf{div-integrable} if $Y_{\mid\acc{x=0}}\in\vf[\ww C^{2},0]$
is (analytically) orbitally linearizable. 
\end{itemize}
\end{defn}
\begin{rem}
Alternatively we could say that the foliation associated to ${\displaystyle Y_{\mid\acc{x=0}}}$
is linearizable. Since ${\displaystyle Y_{\mid\acc{x=0}}}$ is analytic
at the origin of $\ww C^{2}$ and has two opposite eigenvalues, it
follows from a classical result of Brjuno (see \cite{Martinet}),
that $Y_{\mid\acc{x=0}}$ is analytically orbitally linearizable if
and only if it is formally orbitally linearizable.\end{rem}
\begin{defn}
We denote by $\snodiag$ the set of strictly non-degenerate diagonal
doubly-resonant saddle-nodes which are div-integrable.
\end{defn}
The vector field corresponding to the irregular singularity at infinity
in the Painlevé equations $\left(P_{j}\right)_{j=I,\dots,V}$ is orbitally
equivalent to an element of $\snofib$, for generic values of the
parameters (see \cite{Yoshida85}).

We can now state the first main result of our paper (we refer to subsection
\ref{sub:Weak-Gevrey-1-power} for the precise definition of \emph{weak
1-summability}). 
\begin{thm}
\label{Th: Th drsn}Let $Y\in\snodiag$ and let $\hat{\Phi}$ (given
by Theorem \ref{thm: forme normalel formelle}) be the unique formal
fibered diffeomorphism tangent to the identity such that 
\begin{eqnarray*}
\hat{\Phi}_{*}\left(Y\right) & = & x^{2}\pp x+\left(-\lambda+a_{1}x+c_{1}\left(y_{1}y_{2}\right)\right)y_{1}\pp{y_{1}}+\left(\lambda+a_{2}x+c_{2}\left(y_{1}y_{2}\right)\right)y_{2}\pp{y_{2}}\\
 & =: & \ynorm\,\,,
\end{eqnarray*}
where $\lambda\neq0$ and ${\displaystyle c_{1}\left(v\right),c_{2}\left(v\right)\in v\form v}$
are formal power series without constant term. Then:
\begin{enumerate}
\item the normal form $\ynorm$ is analytic \emph{(i.e. ${\displaystyle c_{1},c_{2}\in\germ v}$),
}and it also is div-integrable,\emph{ }i.e. $c_{1}+c_{2}=0$;
\item \label{enu:the-formal-normalizing}the formal normalizing map $\hat{\Phi}$
is weakly 1-summable in every direction, except $\arg\left(\pm\lambda\right)$;
\item \label{enu:there-exist-analytic}there exist analytic sectorial fibered
diffeomorphisms $\Phi_{+}$ and $\Phi_{-}$, (asymptotically) tangent
to the identity, defined in sectorial domains of the form ${\displaystyle S_{+}\times\left(\ww C^{2},0\right)}$
and ${\displaystyle S_{-}\times\left(\ww C^{2},0\right)}$ respectively,
where 
\begin{eqnarray*}
S_{+} & := & \acc{x\in\ww C\mid0<\abs x<r\mbox{ and }\abs{\arg\left(\frac{x}{i\lambda}\right)}<\frac{\pi}{2}+\epsilon}\\
S_{-} & := & \acc{x\in\ww C\mid0<\abs x<r\mbox{ and }\abs{\arg\left(\frac{-x}{i\lambda}\right)}<\frac{\pi}{2}+\epsilon}
\end{eqnarray*}
(for any ${\displaystyle \epsilon\in\left]0,\frac{\pi}{2}\right[}$
and some $r>0$ small enough), which admit $\hat{\Phi}$ as weak Gevrey-1
asymptotic expansion in these respective domains, and which conjugate
$Y$ to $\ynorm$. Moreover $\Phi_{+}$ and $\Phi_{-}$ are the unique
such germs of analytic functions in sectorial domains (see Definition
\ref{def: sectorial diff}).
\end{enumerate}
\end{thm}
\begin{rem}
Although item \ref{enu:there-exist-analytic} above is a straightforward
consequence of the \emph{weak 1-summability} (see subsection \ref{sub:Weak-Gevrey-1-power})
of $\hat{\Phi}$ in item \ref{enu:the-formal-normalizing} above,
we will in fact start by proving item \ref{enu:there-exist-analytic}
in Corollary \ref{cor: existence normalisations sectorielles}, and
show item point \ref{enu:the-formal-normalizing} in Proposition \ref{prop: Weak sectorial normalizations}.\end{rem}
\begin{defn}
\label{def: sectorial normalizations}We call $\Phi_{+}$ and $\Phi_{-}$
the \textbf{sectorial normalizing maps} of $Y\in\snodiag$. 
\end{defn}
They are the weak 1-sums of $\hat{\Phi}$ along the respective directions
$\arg\left(i\lambda\right)$ and $\arg\left(-i\lambda\right)$. Notice
that $\Phi_{+}$ and $\Phi_{-}$ are \emph{germs of analytic sectorial
fibered diffeomorphisms}, \emph{i.e.} they are of the form
\begin{eqnarray*}
\Phi_{+}:S_{+}\times\left(\ww C^{2},0\right) & \longrightarrow & S_{+}\times\left(\ww C^{2},0\right)\\
\left(x,\mathbf{y}\right) & \longmapsto & \left(x,\Phi_{+,1}\left(x,\mathbf{y}\right),\Phi_{+,2}\left(x,\mathbf{y}\right)\right)
\end{eqnarray*}
and 
\begin{eqnarray*}
\Phi_{-}:S_{-}\times\left(\ww C^{2},0\right) & \longrightarrow & S_{-}\times\left(\ww C^{2},0\right)\\
\left(x,\mathbf{y}\right) & \longmapsto & \left(x,\Phi_{-,1}\left(x,\mathbf{y}\right),\Phi_{-,2}\left(x,\mathbf{y}\right)\right)
\end{eqnarray*}
(see section 2. for a precise definition of \emph{germ of analytic
sectorial fibered diffeomorphism}). The fact that they are\emph{ }also
\emph{(asymptotically) tangent to the identity }means that we have:
\[
{\displaystyle \Phi_{\pm}\left(x,\mathbf{y}\right)=\tx{Id}\left(x,\mathbf{y}\right)+\tx O\left(\norm{\left(x,\mathbf{y}\right)}^{2}\right)}\,\,.
\]

In fact, we can prove the uniqueness of the sectorial normalizing
maps under weaker assumptions.
\begin{prop}
\label{prop: unique normalizations}Let $\varphi_{+}$ and $\varphi_{-}$
be two germs of sectorial fibered diffeomorphisms in ${\displaystyle S_{+}\times\left(\ww C^{2},0\right)}$
and ${\displaystyle S_{-}\times\left(\ww C^{2},0\right)}$ respectively,
where $S_{+}$ and $S_{-}$ are as in Theorem \ref{Th: Th drsn},
which are (asymptotically) tangent to the identity and such that 
\[
\left(\varphi_{\pm}\right)_{*}\left(Y\right)=\ynorm\,\,.
\]
Then, they necessarily coincide with the weak 1-sums $\Phi_{+}$ and
$\Phi_{-}$ defined above. 
\end{prop}
It is important to say that we will in fact begin with proving the
existence of germs of sectorial fibered diffeomorphisms $\Phi_{+}$
and $\Phi_{-}$ in ${\displaystyle S_{+}\times\left(\ww C^{2},0\right)}$
and ${\displaystyle S_{-}\times\left(\ww C^{2},0\right)}$ respectively,
which are tangent to the identity and conjugate $Y\in\snofib$ to
its normal form (see Corollary \ref{cor: existence normalisations sectorielles}).
The proposition above guarantees the uniqueness of such sectorial
transforms. It is proved in a second step that $\Phi_{+}$ and $\Phi_{-}$
admits the formal normalizing map $\hat{\Phi}$ as weak Gevrey-1 asymptotic
expansion, which is thus weakly 1-summable.
\begin{rem}
In this paper we prove a theorem of existence of sectorial normalizing
map analogous to the classical one due to Hukuhara-Kimura-Matuda for
saddle-nodes in $\left(\ww C^{2},0\right)$ \cite{HKM}, generalized
later by Stolovitch in any dimension in \cite{Stolo}. Unlike the
method based on a fixed point theorem used by these authors, we use
a more geometric approach (following the works of Teyssier~\cite{Teyssier03,teyssier2004equation})
based on the resolution of an homological equation, by integrating
a well chosen 1-form along asymptotic paths. This latter approach
turned out to be more efficient to deal with the fact that $Y_{\mid\acc{x=0}}$
is not necessarily linearizable. Indeed, if we look at \cite{Stolo}
in details, one of the first problem is that in the irregular systems
that needs to be solved by a fixed point method (for instance equation
$\left(2.7\right)$ in the cited paper), the non-linear terms would
not be divisible by the ``time'' variable $t$ in our situation.
This would complicate the different estimates that are done later
in the cited work. This was the first main new phenomena we have met. 

In~a forthcoming paper we prove that the sectorial normalizing maps
$\Phi_{+},\Phi_{-}$ in Theorem \ref{Th: Th drsn} admit in fact the
unique formal normalizing map $\hat{\Phi}$ given by Theorem \ref{thm: forme normalel formelle}
as ``true'' Gevrey-1 asymptotic expansion in ${\displaystyle S_{+}\times\left(\ww C^{2},0\right)}$
and ${\displaystyle S_{-}\times\left(\ww C^{2},0\right)}$ respectively.
This is done by studying $\Phi_{+}\circ\left(\Phi_{-}\right)^{-1}$
in ${\displaystyle \left(S_{+}\cap S_{-}\right)\times\left(\ww C^{2},0\right)}$
(and more generally any germ of sectorial fibered isotropy of $\ynorm$
in ``narrow'' sectorial neighborhoods $\left(S_{+}\cap S_{-}\right)\times\left(\ww C^{2},0\right)$
which admits the identity as weak Gevrey-1 asymptotic expansion).
In the cited paper we also:
\begin{itemize}
\item prove that the formal normalizing map $\hat{\Phi}$ in Theorem \ref{Th: Th drsn}
is in fact 1-summable (and not only \emph{weakly} 1-summable).
\item provide a theorem of analytic classification, based on the study over
``small'' sectors $S_{+}\cap S_{-}$ of the transition maps $\Phi_{+}\circ\Phi_{-}^{-1}$
(also called\emph{ Stokes diffeomorphisms}): they are sectorial isotropies
of the normal form $\ynorm$ which are exponentially close to the
identity.
\end{itemize}
The main difficulty is to establish that such a sectorial isotropy
of $\ynorm$ over the ``narrow'' sectors described above is necessarily
exponentially close to the identity. This will be done \emph{via}
a detailed analysis of these maps in the space of leaves. In fact,
this is the second main new difficulty we have met, which is due to
the presence of the ``resonant'' term 
\[
\frac{c_{m}\left(y_{1}y_{2}\right)^{m}\log\left(x\right)}{x}
\]
in the exponential expression of the first integrals of the vector
field in normal form. In \cite{Stolo}, similar computations are done
in subsection $3.4.1$. In this part of the paper, infinitely many
irregular differential equations appear when identifying terms of
same homogeneous degree. The fact that $Y_{\mid\acc{x=0}}$ is linear
implies that these differential equations are all linear and independent
of each others (\emph{i.e. }they are not mixed together). In our situation
this is not the case, which yields more complicated computations.
\end{rem}

\subsection{\label{sub: transversally symplectic}Painlevé equations: the transversally
Hamiltonian case }

~

In \cite{Yoshida85} Yoshida shows that a vector field in the class
$\snofib$ naturally appears after a suitable compactification (given
by the so-called Boutroux coordinates \cite{Boutroux13}) of the phase
space of Painlevé equations $\left(P_{j}\right)_{j=I,\dots,V}$, for
generic values of the parameters. In these cases the vector field
presents an additional transverse Hamiltonian structure. Let us illustrate
these computations in the case of the first Painlevé equation: 
\begin{eqnarray*}
\left(P_{I}\right)\,\,\,\,\,\,\,\,\,\,\,\,\,\,\,\,\,\,\,\,\,\ddd{^{2}z_{1}}{t^{2}} & = & 6z_{1}^{2}+t\,\,\,\,\,\,\,\,\,\,\,\,\,\,.
\end{eqnarray*}
As is well known since Okamoto \cite{Okamoto}, $\left(P_{I}\right)$
can be seen as a non-autonomous Hamiltonian system 
\[
\begin{cases}
{\displaystyle \ppp{z_{1}}t=-\ppp H{z_{2}}}\\
{\displaystyle \ppp{z_{2}}t=\ppp H{z_{1}}}
\end{cases}
\]
with Hamiltonian 
\begin{eqnarray*}
H\left(t,z_{1},z_{2}\right) & := & 2z_{1}^{3}+tz_{1}-\frac{z_{2}^{2}}{2}.
\end{eqnarray*}
More precisely, if we consider the standard symplectic form $\omega:=dz_{1}\wedge dz_{2}$
and the vector field 
\begin{eqnarray*}
Z & := & \pp t-\ppp H{z_{2}}\pp{z_{1}}+\ppp H{z_{1}}\pp{z_{2}}
\end{eqnarray*}
induced by $\left(P_{I}\right)$, then the Lie derivative 
\[
\cal L_{Z}\left(\omega\right)=\left(\ppp{^{2}H}{t\partial z_{1}}\mbox{d}z_{1}+\ppp{^{2}H}{t\partial z_{2}}\mbox{d}z_{2}\right)\wedge\mbox{d}t=\mbox{d}z_{1}\wedge\mbox{d}t
\]
 belongs to the ideal $\ps{\tx dt}$ generated by $\tx dt$ in the
exterior algebra $\Omega^{*}\left(\ww C^{3}\right)$ of differential
forms in variables $\left(t,z_{1},z_{2}\right)$. Equivalently, for
any $t_{1},t_{2}\in\ww C$ the flow of $Z$ at time $\left(t_{2}-t_{1}\right)$
acts as a \emph{symplectomorphism} between fibers $\acc{t=t_{1}}$
and $\acc{t=t_{2}}$. 

The weighted compactification given by the Boutroux coordinates~\cite{Boutroux13}
defines a chart near $\acc{t=\infty}$ as follows: 
\[
\begin{cases}
{\displaystyle z_{2}=y_{2}x^{-\frac{3}{5}}}\\
{\displaystyle z_{1}=y_{1}x^{-\frac{2}{5}}}\\
{\displaystyle t=x^{-\frac{4}{5}}} & .
\end{cases}
\]
In the coordinates $\left(x,y_{1},y_{2}\right)$, the vector field
$Z$ is transformed, up to a translation $y_{1}\leftarrow y_{1}+\zeta$
with ${\displaystyle \zeta=\frac{i}{\sqrt{6}}}$, to the vector field
\begin{eqnarray}
\tilde{Z} & = & -\frac{5}{4x^{\frac{1}{5}}}Y\label{eq: P1 ham at infinity}
\end{eqnarray}
where
\begin{eqnarray*}
Y & = & x^{2}\pp x+\left(-\frac{4}{5}y_{2}+\frac{2}{5}xy_{1}+\frac{2\zeta}{5}x\right)\pp{y_{1}}+\left(-\frac{24}{5}y_{1}^{2}-\frac{48\zeta}{5}y_{1}+\frac{3}{5}xy_{2}\right)\pp{y_{2}}\,\,\,\,\,\,\,\,\,.
\end{eqnarray*}
We observe that $Y$ is a\emph{ }strictly non-degenerate doubly-resonant
saddle-node as in Definitions \ref{def: drsn} and \ref{def: non-deg}
with residue $\tx{res}\left(Y\right)=1$. Furthermore we have: 
\[
\begin{cases}
{\displaystyle \tx dt} & {\displaystyle =-\frac{4}{5}5^{\frac{4}{5}}x^{-\frac{9}{5}}\tx dx}\\
{\displaystyle \tx dz_{1}\wedge\tx dz_{2}} & {\displaystyle =\frac{1}{x}\left(\tx dy_{1}\wedge\tx dy_{2}\right)+\frac{1}{5x^{2}}\left(2y_{1}\tx dy_{2}-3y_{2}\tx dy_{1}\right)\wedge\tx dx}\\
 & {\displaystyle \in\,\frac{1}{x}\left(\tx dy_{1}\wedge\tx dy_{2}\right)+\ps{\tx dx}}
\end{cases}\,\,\,\,\,,
\]
where $\ps{\mbox{d}x}$ denotes the ideal generated by $\mbox{d}x$
in the algebra of holomorphic forms in $\ww C^{*}\times\ww C^{2}$.
We finally obtain 
\[
\begin{cases}
{\displaystyle {\displaystyle \cal L_{Y}\left(\frac{\tx dy_{1}\wedge\tx dy_{2}}{x}\right)}=\frac{1}{5x}\left(3y_{2}\mbox{d}y_{1}-\left(2\zeta+2y_{1}\right)\mbox{d}y_{2}\right)\wedge\mbox{d}x}\\
{\displaystyle \cal L_{Y}\left(\mbox{d}x\right)=2x\mbox{d}x}
\end{cases}\quad.
\]
 Therefore, both ${\displaystyle \cal L_{Y}\left(\omega\right)}$
and $\cal L_{Y}\left(\mbox{d}x\right)$ are differential forms who
lie in the ideal $\ps{\tx dx}$, in the algebra of germs of meromorphic
1-forms in $\left(\ww C^{3},0\right)$ with poles only in $\acc{x=0}$.
This motivates the following:
\begin{defn}
\label{def: intro}Consider the rational 1-form 
\begin{eqnarray*}
\omega & := & \frac{\mbox{d}y_{1}\wedge\mbox{d}y_{2}}{x}~.
\end{eqnarray*}
We say that vector field $Y$ is \textbf{transversally Hamiltonian
}(with respect to $\omega$ and $\mbox{dx}$) if 
\begin{eqnarray*}
\mathcal{L}_{Y}\left(\tx dx\right)\in\left\langle \tx dx\right\rangle  & \mbox{ and } & \mathcal{L}_{Y}\left(\omega\right)\in\left\langle \tx dx\right\rangle \qquad.
\end{eqnarray*}
For any open sector $S\subset\ww C^{*}$, we say that a germ of sectorial
fibered diffeomorphism $\Phi$ in $S\times\left(\ww C^{2},0\right)$
is \textbf{transversally symplectic} (with respect to $\omega$ and
$\mbox{d}x$) if
\[
\Phi^{*}\left(\omega\right)\in\,\omega+\ps{\tx dx}\qquad
\]
 (Here $\Phi^{*}\left(\omega\right)$ denotes the pull-back of $\omega$
by $\Phi$). 

We denote by $\sdiffid$ the group of transversally symplectic diffeomorphisms
which are tangent to the identity.\end{defn}
\begin{rem}
~
\begin{enumerate}
\item The flow of a transversally Hamiltonian vector field $X$ defines
a map between fibers $\acc{x=x_{1}}$ and $\acc{x=x_{2}}$ which sends
$\omega_{\mid x=x_{1}}$ onto $\omega_{\mid x=x_{2}}$, since 
\[
{\displaystyle \left(\exp\left(X\right)\right)^{*}\left(\omega\right)\in\,\omega+\ps{\mbox{d}x}}\,\,\,.
\]

\item A fibered sectorial diffeomorphism $\Phi$ is transversally symplectic
if and only if $\det\left(\mbox{D}\Phi\right)=1$.
\end{enumerate}
\end{rem}
\begin{defn}
\label{def: nc dr th}A \textbf{transversally Hamiltonian doubly-resonant
saddle-node} is a transversally Hamiltonian vector field which is
conjugate, \emph{via }a transversally symplectic diffeomorphism, to
one of the form 
\begin{eqnarray*}
Y & = & x^{2}\pp x+\Big(-\lambda y_{1}+F_{1}\left(x,\mathbf{y}\right)\Big)\pp{y_{1}}+\Big(\lambda y_{2}+F_{2}\left(x,\mathbf{y}\right)\Big)\pp{y_{2}}\,\,\,\,\,,
\end{eqnarray*}
with $\lambda\in\ww C^{*}$ and $f_{1},f_{2}$ analytic in $\left(\ww C^{3},0\right)$
and of order at least $2$.
\end{defn}
Notice that a transversally Hamiltonian doubly-resonant saddle-node
is necessarily strictly non-degenerate (since its residue is always
equal to $1$), and also div-integrable (see section 3). It follows
from Yoshida's work \cite{Yoshida85} that the doubly-resonant saddle-nodes
at infinity in Painlevé equations$\left(P_{j}\right)_{j=I,\dots,V}$
(for generic values of the parameters) all are transversally Hamiltonian. 

We recall the second main result from \cite{bittmann1}. 
\begin{thm}
\cite{bittmann1}\label{thm: Th ham formel}~

Let $Y\in\sndiag$ be a diagonal doubly-resonant saddle-node which
is supposed to be transversally Hamiltonian. Then, there exists a
unique formal fibered transversally symplectic diffeomorphism $\hat{\Phi}$,
tangent to identity, such that: 
\begin{eqnarray}
\hat{\Phi}_{*}\left(Y\right) & = & x^{2}\pp x+\left(-\lambda+a_{1}x-c\left(y_{1}y_{2}\right)\right)y_{1}\pp{y_{1}}+\left(\lambda+a_{2}x+c\left(y_{1}y_{2}\right)\right)y_{2}\pp{y_{2}}\nonumber \\
 & =: & \ynorm\,\,,\label{eq: fibered normal form-1-1-1}
\end{eqnarray}
where $\lambda\in\ww C^{*}$, $c\left(v\right)\in v\form v$ a formal
power series in $v=y_{1}y_{2}$ without constant term and $a_{1},a_{2}\in\ww C$
are such that $a_{1}+a_{2}=1$.
\end{thm}
As a consequence of Theorem \ref{thm: Th ham formel}, Theorem \ref{Th: Th drsn}
we have the following:
\begin{thm}
\label{thm: Th ham}Let $Y$ be a transversally Hamiltonian doubly-resonant
saddle-node and let $\hat{\Phi}$ be the unique formal normalizing
map given by Theorem \ref{thm: Th ham formel}. Then the associate
sectorial normalizing maps $\Phi_{+}$ and $\Phi_{-}$ are also transversally
symplectic.\end{thm}
\begin{proof}
Since $\hat{\Phi}$ is weakly 1-summable in $S_{\pm}\times\left(\ww C^{2},0\right)$,
the formal power series $\det\left(\mbox{D}\hat{\Phi}\right)$ is
also weakly 1-summable in $S_{\pm}\times\left(\ww C^{2},0\right)$,
and its asymptotic expansion has to be the constant $1$. By uniqueness
of the weak 1-sum, we thus have $\det\left(\mbox{D}\Phi_{\pm}\right)=1$.
\end{proof}

\subsection{Outline of the paper}

~

In section 2, we introduce the different tools we need concerning
asymptotic expansion, Gevrey-1 series and 1-summability. We will in
particular introduce a notion of ``\textbf{weak}'' 1-summability.

In section 3, we prove Proposition \ref{prop: forme pr=0000E9par=0000E9e ordre N},
which states that we can always formally conjugate a non-degenerate
doubly-resonant saddle-node which is also div-integrable to its normal
form up to remaining terms of order $\tx O\left(x^{N}\right)$, for
all $N\in\ww N_{>0}$, and the conjugacy is actually $1$-summable.

In section 4, we prove that for all $Y\in\snofib$, there exists a
pair of sectorial normalizing maps $\left(\Phi_{+},\Phi_{-}\right)$
tangent to the identity which conjugates $Y$ to its normal form $\ynorm$
over sectors with opening greater than $\pi$ and arbitrarily close
to $2\pi$ (see Corollary \ref{cor: existence normalisations sectorielles}).

In section $5$, the uniqueness of the sectorial normalizing maps,
stated in Proposition \ref{prop: unique normalizations}, is proved
thanks to Proposition \ref{prop: isot sect}. Moreover, we will see
that $\Phi_{+}$ and $\Phi_{-}$ both admit the unique formal normalizing
map $\hat{\Phi}$ given by Theorem \ref{thm: forme normalel formelle}
as weak Gevrey-1 asymptotic expansion (see Proposition \ref{prop: Weak sectorial normalizations}). 

\tableofcontents{}

\section{Background}

We refer the reader to \cite{MR82,malgrange1995sommation,ramis1993divergent,DeMaesschalck}
for a detailed introduction to the theory of asymptotic expansion,
Gevrey series and summability (see also \cite{Stolo} for a useful
discussion of these concepts), where one can find the proofs of the
classical results we recall (but we do not prove here). We call $x\in\ww C$
the \emph{independent} variable and ${\displaystyle \mathbf{y}:=\left(y_{1},\dots,y_{n}\right)\in\ww C^{n}}$,
$n\in\ww N$, the \emph{dependent} variables. As usual we define ${\displaystyle \mathbf{y^{k}}:=y_{1}^{k_{1}}\dots y_{n}^{k_{n}}}$
for ${\displaystyle \mathbf{k}=\left(k_{1},\dots,k_{n}\right)\in\ww N^{n}}$,
and ${\displaystyle \abs{\mathbf{k}}=k_{1}+\dots+k_{n}}$. The notions
of asymptotic expansion, Gevrey series and 1-summability presented
here are always considered with respect to the independent variable
$x$ living in (open) sectors, the dependent variable $\mathbf{y}$
belonging to poly-discs 
\[
\mathbf{D\left(0,r\right)}:=\acc{\mathbf{y}=\left(y_{1},\dots,y_{n}\right)\in\ww C^{n}\mid\abs{y_{1}}<r_{1},\dots\abs{y_{n}}<r_{n}}\,\,\,,
\]
of poly-radius ${\displaystyle \mathbf{r}=\left(r_{1},\dots,r_{n}\right)\in\left(\ww R_{>0}\right)^{n}}$.
Given an open subset 
\[
\cal U\subset\ww C^{n+1}=\acc{\left(x,\mathbf{y}\right)\in\ww C\times\ww C^{n}}
\]
 we denote by $\cal O\left({\cal U}\right)$ the algebra of holomorphic
function in $\cal U$. The algebra of germs of analytic functions
of $m$ variables $\mathbf{x}:=\left(x_{1},\dots,x_{m}\right)$ at
the origin is denoted by $\germ{\mathbf{x}}$.

The results recalled in this section are valid when $n=0$. Some statements
for which we do not give a proof can be proved exactly as in the classical
case $n=0$, uniformly in the dependent multi-variable $\mathbf{y}$.
For convenience and homogeneity reasons we will present some classical
results not in their original (and more general) form, but rather
in more specific cases which we will need. Finally, we will introduce
a notion of \emph{weak }Gevrey-1 summability, which we will compare
to the classical notion of 1-summability.

\subsection{Sectorial germs}

~

Given $r>0$, and $\alpha,\beta\in\ww R$ with $\alpha<\beta$, we
denote by $\sect r{\alpha}{\beta}$ the following open sector:

\[
S\left(r,\alpha,\beta\right):=\acc{x\in\ww C\mid0<\abs x<r\mbox{ and }\alpha<\arg\left(x\right)<\beta}\,\,.
\]
Let $\theta\in\ww R$, $\eta\in\ww R_{\geq0}$ and $n\in\ww N$.
\begin{defn}
\label{def: sectorial germs}
\begin{enumerate}
\item An\emph{ x-sectorial neighborhood }(or simply \emph{sectorial neighborhood})
\emph{of the origin} (in $\ww C^{n+1}$) \emph{in the direction $\theta$
with opening $\eta$} is an open set $\cal S\subset\ww C^{n+1}$ such
that
\[
\cal S\supset S\left(r,\theta-\frac{\eta}{2}-\epsilon,\theta+\frac{\eta}{2}+\epsilon\right)\times\mathbf{D\left(0,r\right)}
\]
for some $r>0$, $\mathbf{r}\in\left(\ww R_{>0}\right)^{n}$ and $\epsilon>0$.
We denote by $\left(\germsect{\theta}{\eta},\leq\right)$ the directed
set formed by all such neighborhoods, equipped with the order relation
\begin{eqnarray*}
S_{1}\leq S_{2} & \Longleftrightarrow & S_{1}\supset S_{2}\,\,.
\end{eqnarray*}

\item The algebra of \emph{germs of holomorphic functions in a sectorial
neighborhood of the origin in the direction $\theta$ with opening
$\eta$} is the direct limit 
\[
\cal O\left(\cal S_{\theta,\eta}\right):=\underrightarrow{\lim}\,\cal O\left(\cal S\right)
\]
with respect to the directed system defined by $\acc{\cal O\left(\cal S\right):\cal S\in\germsect{\theta}{\eta}}$.
\end{enumerate}
\end{defn}
We now give the definition of \emph{a (germ of a) sectorial diffeomorphism}.
\begin{defn}
\label{def: sectorial diff}
\begin{enumerate}
\item Given an element $\cal S\in\cal S_{\theta,\eta}$, we denote by $\fdiff[\cal S,\tx{Id}]$
the set of holomorphic fibered diffeomorphisms of the form 
\begin{eqnarray*}
\Phi:\cal S & \rightarrow & \Phi\left(\cal S\right)\\
\left(x,\mathbf{y}\right) & \mapsto & \left(x,\phi_{1}\left(x,\mathbf{y}\right),\phi_{2}\left(x,\mathbf{y}\right)\right)\,\,,
\end{eqnarray*}
such that ${\displaystyle \Phi\left(x,\mathbf{y}\right)-\tx{Id}\left(x,\mathbf{y}\right)=\tx O\left(\norm{x,\mathbf{y}}^{2}\right),\,\,\mbox{ as }\left(x,\mathbf{y}\right)\rightarrow\left(0,\mathbf{0}\right)\mbox{ in }\cal S.}$
\footnote{This condition implies in particular that $\Phi\left(\cal S\right)\in\germsect{\theta}{\eta}$.%
}
\item The set of \emph{germs of (fibered) sectorial diffeomorphisms in the
direction $\theta$ with opening $\eta$, tangent to the identity},
is the direct limit 
\[
\diffsect[\theta][\eta]:=\underrightarrow{\lim}\,\fdiff[\cal S,\tx{Id}]
\]
with respect to the directed system defined by $\acc{\fdiff[\cal S,\tx{Id}]:\cal S\in\cal S_{\theta,\eta}}$.
We equip $\diffsect[\theta][\eta]$ with a group structure as follows:
given two germs $\Phi,\Psi\in{\displaystyle \diffsect}$ we chose
corresponding representatives $\Phi_{0}\in\fdiff[\cal S,\tx{Id}]$
and $\Psi_{0}\in\fdiff[\cal T,\tx{Id}]$ with $\cal S,\cal T\in\cal S_{\theta,\eta}$
such that $\cal T\subset\Phi_{0}\left(\cal S\right)$ and let $\Psi\circ\Phi$
be the germ defined by $\Psi_{0}\circ\Phi_{0}$. %
\footnote{One can prove that this definition is independent of the choice of
the representatives%
}
\end{enumerate}
\end{defn}
We will also need the notion of \emph{asymptotic sectors.}
\begin{defn}
\label{def: asymptotitc sector}An\emph{ (open) asymptotic sector
of the origin in the direction $\theta$ and with opening $\eta$}
is an open set $S\subset\ww C$ such that 
\[
S\in\bigcap_{0\leq\eta'<\eta}\cal S_{\theta,\eta'}\,\,.
\]
We denote by $\cal{AS}_{\theta,\eta}$ the set of all such (open)
asymptotic sectors.
\end{defn}

\subsection{\label{sub: Strong-Gevrey-1-power} Gevrey-1 power series and 1-summability}

~

\subsubsection{Gevrey-1 asymptotic expansions}

~

In this subsection we fix a formal power series which we write under
two forms: 
\[
{\displaystyle \hat{f}\left(x,\mathbf{y}\right)=\sum_{k\geq0}f_{k}\left(\mathbf{y}\right)x^{k}=\sum_{\left(j_{0},\mathbf{j}\right)\in\ww N^{n+1}}f_{j_{0},\mathbf{j}}x^{j_{0}}\mathbf{y^{j}}\in\form{x,\mathbf{y}}}\,\,,
\]
using the canonical identification $\form{x,\mathbf{y}}=\form x\left\llbracket \mathbf{y}\right\rrbracket =\form{\mathbf{y}}\left\llbracket x\right\rrbracket $.
We also fix a norm $\norm{\cdot}$ in $\ww C^{n+1}$.
\begin{defn}
\label{def:Gevrey-1_expansion}~
\begin{itemize}
\item A function $f$ analytic in a domain ${\displaystyle \sect r{\alpha}{\beta}\times\mathbf{D\left(0,r\right)}}$
admits $\hat{f}$ as asymptotic expansion \emph{in the sense of Gérard-Sibuya}
in this domain if for all closed sub-sector $S'\subset S\left(r,\alpha,\beta\right)$
and compact $\mathbf{K}\subset\mathbf{D\left(0,r\right)}$, for all
$N\in\ww N$, there exists a constant $C_{S',K,N}>0$ such that: 
\[
\abs{f\left(x,\mathbf{y}\right)-\sum_{j_{0}+j_{1}+\dots j_{n}\leq N}f_{j_{0},\mathbf{j}}x^{j_{0}}\mathbf{y^{j}}}\leq C_{S',K,N}\norm{\left(x,\mathbf{y}\right)}^{N+1}
\]
for all $\left(x,\mathbf{y}\right)\in S'\times K$.
\item A function $f$ analytic in a domain ${\displaystyle \sect r{\alpha}{\beta}\times\mathbf{D\left(0,r\right)}}$
admits $\hat{f}$ as\emph{ asymptotic expansion (with respect to $x$)
}if for all $k\in\ww N$, $f_{k}\left(\mathbf{y}\right)$ is analytic
in $\mathbf{D\left(0,r\right)}$, and if for all closed sub-sector
$S'\subset S\left(r,\alpha,\beta\right)$, compact subset $\mathbf{K}\subset\mathbf{D\left(0,r\right)}$
and $N\in\ww N$, there exists $A_{S',K,N}>0$ such that: 
\[
\abs{f\left(x,\mathbf{y}\right)-\sum_{k\geq0}^{N}f_{k}\left(\mathbf{y}\right)x^{k}}\leq A_{S',K,N}\abs x^{N+1}
\]
for all $\left(x,\mathbf{y}\right)\in S'\times K$.
\item An analytic function $f$ in a sectorial domain ${\displaystyle \sect r{\alpha}{\beta}\times\mathbf{D\left(0,r\right)}}$
admits $\hat{f}$ as \emph{Gevrey-1 asymptotic expansion} in this
domain, if for all $k\in\ww N$, $f_{k}\left(\mathbf{y}\right)$ is
analytic in $\mathbf{D\left(0,r\right)}$, and if for all closed sub-sector
$S'\subset\sect r{\alpha}{\beta}$, there exists $A,C>0$ such that:
\[
\abs{f\left(x,\mathbf{y}\right)-\sum_{k=0}^{N-1}f_{k}\left(\mathbf{y}\right)x^{k}}\leq AC^{N}\left(N!\right)\abs x^{N}
\]
for all $N\in\ww N$ and $\left(x,\mathbf{y}\right)\in S'\times\mathbf{D}\left(\mathbf{0},\mathbf{r}\right)$.
\end{itemize}
\end{defn}
\begin{rem}
~
\begin{enumerate}
\item If a function admits $\hat{f}$ as Gevrey-1 asymptotic expansion in
${\displaystyle \sect r{\alpha}{\beta}\times\mathbf{D\left(0,r\right)}}$,
then it also admits $\hat{f}$ as asymptotic expansion. 
\item If a function admits $\hat{f}$ as asymptotic expansion in ${\displaystyle \sect r{\alpha}{\beta}\times\mathbf{D\left(0,r\right)}}$,
then it also admits $\hat{f}$ as asymptotic expansion in the the
sense of Gérard-Sibuya. 
\item An asymptotic expansion (in any of the different senses described
above) is unique. 
\end{enumerate}
\end{rem}
As a consequence of Stirling formula, we have the following characterization
for functions admitting $0$ as Gevrey-1 asymptotic expansion.
\begin{prop}
\label{prop: dev asympt nul expo plat}The set of analytic functions
admitting $0$ as Gevrey-1 asymptotic expansion at the origin in a
sectorial domain ${\displaystyle \sect r{\alpha}{\beta}\times\mathbf{D\left(0,r\right)}}$
is exactly the set of of analytic functions $f$ in ${\displaystyle \sect r{\alpha}{\beta}\times\mathbf{D\left(0,r\right)}}$
such that for all closed sub-sector $S'\subset\sect r{\alpha}{\beta}$
and all compact $\mathbf{K}\subset\mathbf{D\left(0,r\right)}$, there
exist $A_{S',K},B_{S',K}>0$ such that:
\[
\abs{f\left(x,\mathbf{y}\right)}\leq A_{S',K}\exp\left(-\frac{B_{S',K}}{\abs x}\right)\,\,.
\]
We say that such a function is exponentially flat at the origin in
the corresponding domain.
\end{prop}

\subsubsection{Borel transform and Gevrey-1 power series}

~
\begin{defn}
~\label{def:Borel_transform}
\begin{itemize}
\item We define the Borel transform $\cal B\left(\hat{f}\right)$ of $\hat{f}$
as: 
\[
\cal B\left(\hat{f}\right)\left(t,\mathbf{y}\right):=\sum_{k\geq0}\frac{f_{k}\left(\mathbf{y}\right)}{k!}t^{k}\,\,.
\]

\item We say that $\hat{f}$ is Gevrey-1 if $\cal B\left(\hat{f}\right)$
is convergent in a neighborhood of the origin in $\ww C\times\ww C^{n}$.
Notice that in this case the $f_{k}\left(\mathbf{y}\right),\, k\geq0$,
are all analytic in a same polydisc $\mathbf{D}\left(\mathbf{0},\mathbf{r}\right)$,
of poly-radius ${\displaystyle \mathbf{r}=\left(r_{n}\dots,r_{n}\right)\in\left(\ww R_{>0}\right)^{n}}$,
so that $\cal B\left(\hat{f}\right)$ is analytic in $\mbox{D}\left(0,\rho\right)\times\mathbf{D}\left(\mathbf{0},\mathbf{r}\right)$,
for some $\rho>0$. Possibly by reducing $\rho,r_{1},\dots,r_{n}>0$,
we can assume that $\cal B\left(\hat{f}\right)$ is bounded in $\mbox{D}\left(0,\rho\right)\times\mathbf{D}\left(\mathbf{0},\mathbf{r}\right)$.
\end{itemize}
\end{defn}
\begin{rem}
~
\begin{enumerate}
\item If a sectorial function $f$ admits $\hat{f}$ for Gevrey-$1$ asymptotic
expansion as in Definition~\ref{def:Gevrey-1_expansion} then $\hat{f}$
is a Gevrey-1 formal power series.
\item The set of all Gevrey-$1$ formal power series is an algebra closed
under partial derivatives $\pp x,\pp{y_{1}},\dots,\pp{y_{n}}$.
\end{enumerate}
\end{rem}
\bigskip{}

\begin{rem}
\label{rem: deifinition bis transformee de Borel}For technical reasons
we will sometimes need to use another definition of the Borel transform,
that is:
\[
\widetilde{\cal B}\left(\hat{f}\right)\left(t,\mathbf{y}\right):=\sum_{k\geq0}f_{k+1}\left(\mathbf{y}\right)\frac{t^{k}}{k!}\,\,.
\]
The first definition we gave has the advantage of being ``directly''
invertible (\emph{via }the Laplace transform) for all 1-summable formal
power series (see next subsection), but behaves not so good with respect
to the product. On the contrary, the second definition will be ``directly''
invertible only for 1-summable formal power series with null constant
term (otherwise a translation is needed). However, the advantage of
the second Borel transform is that it changes a product into a convolution
product: 
\[
\widetilde{\cal B}\left(\hat{f}\hat{g}\right)=\left(\widetilde{{\cal B}}\left(\hat{f}\right)\ast\widetilde{{\cal B}}\left(\hat{g}\right)\right)\,\,,
\]
where the convolution product of two analytic functions $h_{1}h_{2}$
is defined by 
\[
\left(h_{1}\ast h_{2}\right)\left(t,\mathbf{y}\right):=\int_{0}^{t}h_{1}\left(s\right)h_{2}\left(s-t\right)\dd[s]\,\,.
\]
The property of being Gevrey-1 or not does not depend on the choice
of the definition we take for the Borel transform.
\end{rem}

\subsubsection{Directional 1-summability and Borel-Laplace summation}

~
\begin{defn}
\label{def: 1_summability}Given $\theta\in\ww R$ and $\delta>0$,
we define the infinite sector in the direction $\theta$ with opening
$\delta$ as the set 
\[
\cal A_{\theta,\delta}^{\infty}:=\acc{t\in\ww C^{*}\mid\abs{\arg\left(t\right)-\theta}<\frac{\delta}{2}}\,\,.
\]
We say that $\hat{f}$ is \emph{1-summable in the direction} $\theta\in\ww R$,
if the following three conditions holds:
\begin{itemize}
\item $\hat{f}$ is a Gevrey-1 formal power series;
\item $\cal B\left(\hat{f}\right)$ can be analytically continued to a domain
of the form $\cal A_{\theta,\delta}^{\infty}\times\mathbf{D}\left(\mathbf{0},\mathbf{r}\right)$;
\item there exists $\lambda>0,M>0$ such that:
\[
\forall\left(t,\mathbf{y}\right)\in\cal A_{\theta,\delta}^{\infty}\times\mathbf{D}\left(\mathbf{0},\mathbf{r}\right),\,\abs{\cal B\left(\hat{f}\right)\left(t,\mathbf{y}\right)}\leq M\exp\left(\lambda\abs t\right)\,\,\,\,.
\]

\end{itemize}
In this case we set ${\displaystyle \Delta_{\theta,\delta,\rho}:=\cal A_{\theta,\delta}^{\infty}\cup\mbox{D}\left(0,\rho\right)}$
and 
\[
\norm{\hat{f}}_{\lambda,\theta,\delta,\rho,\mathbf{r}}:=\underset{\left(t,\mathbf{y}\right)\in\Delta_{\theta,\delta,\rho}\times\mathbf{D}\left(\mathbf{0},\mathbf{r}\right)}{\sup}\abs{\cal B\left(\hat{f}\right)\left(t,\mathbf{y}\right)\exp\left(-\lambda\abs t\right)}\quad.
\]
If the domain is clear from the context we will simply write ${\displaystyle \norm{\hat{f}}_{\lambda}}$.\end{defn}
\begin{rem}
\label{rem: norme bis borel}~
\begin{enumerate}
\item For fixed $\left(\lambda,\theta,\delta,\rho,\mathbf{r}\right)$ as
above, the set $\mathfrak{B}_{\lambda,\theta,\delta,\rho,\mathbf{r}}$
of formal power series $\hat{f}$ 1-summable in the direction $\theta$
and such that ${\displaystyle \norm{\hat{f}}_{\lambda,\theta,\delta,\rho,\mathbf{r}}<+\infty}$
is a Banach vector space for the norm $\norm{\cdot}_{\lambda,\theta,\delta,\rho,\mathbf{r}}$.
We simply write ${\displaystyle \left(\mathfrak{B}_{\lambda},\norm{\cdot}_{\lambda}\right)}$
when there is no ambiguity.
\item We will also need a norm well-adapted to the second Borel transform
$\widetilde{B}$ (\emph{cf. }Remark \ref{rem: deifinition bis transformee de Borel}),
that is:
\[
\norm{\hat{f}}_{\lambda,\theta,\delta,\rho,\mathbf{r}}^{\tx{bis}}:=\underset{\left(t,\mathbf{y}\right)\in\Delta_{\theta,\delta,\rho}\times\mathbf{D}\left(\mathbf{0},\mathbf{r}\right)}{\sup}\abs{\cal B\left(\hat{f}\right)\left(t,\mathbf{y}\right)\left(1+\lambda^{2}\abs t^{2}\right)\exp\left(-\lambda\abs t\right)}\,\,.
\]
We write then $\mathfrak{B}_{\lambda,\theta,\delta,\rho,\mathbf{r}}^{\tx{bis}}$
the set space of formal power series $\hat{f}$ which are 1-summable
in the direction $\theta$ and such that ${\displaystyle \norm{\hat{f}}_{\lambda,\theta,\delta,\rho,\mathbf{r}}^{\tx{bis}}<+\infty}$.
\item If $\lambda'\geq\lambda$, then ${\displaystyle \mathfrak{B}_{\lambda,\theta,\delta,\rho,\mathbf{r}}}\subset\mathfrak{B}_{\lambda',\theta,\delta,\rho,\mathbf{r}}$
and ${\displaystyle \mathfrak{B}_{\lambda,\theta,\delta,\rho,\mathbf{r}}^{\tx{bis}}}\subset\mathfrak{B}_{\lambda',\theta,\delta,\rho,\mathbf{r}}^{\tx{bis}}$.
\end{enumerate}
\end{rem}
\begin{prop}[{\foreignlanguage{french}{\cite[Proposition 4.]{DeMaesschalck}}}]
\label{prop: norme d'algebre} If $\hat{f},\hat{g}\in\mathfrak{B}_{\lambda,\theta,\delta,\rho,\mathbf{r}}^{\tx{bis}}$,
then $\hat{f}\hat{g}\in\mathfrak{B}_{\lambda,\theta,\delta,\rho,\mathbf{r}}^{\tx{bis}}$
and:
\[
\norm{\hat{f}\hat{g}}_{\lambda,\theta,\delta,\rho,\mathbf{r}}^{\tx{bis}}\leq\frac{4\pi}{\lambda}\norm{\hat{f}}_{\lambda,\theta,\delta,\rho,\mathbf{r}}^{\tx{bis}}\norm{\hat{g}}_{\lambda,\theta,\delta,\rho,\mathbf{r}}^{\tx{bis}}\,\,.
\]
\end{prop}
\begin{rem}
If $\lambda\geq4\pi$, then $\norm{\cdot}_{\lambda,\theta,\delta,\rho,\mathbf{r}}^{\tx{bis}}$
is a sub-multiplicative norm, \emph{i.e. 
\[
\norm{\hat{f}\hat{g}}_{\lambda,\theta,\delta,\rho,\mathbf{r}}^{\tx{bis}}\leq\norm{\hat{f}}_{\lambda,\theta,\delta,\rho,\mathbf{r}}^{\tx{bis}}\norm{\hat{g}}_{\lambda,\theta,\delta,\rho,\mathbf{r}}^{\tx{bis}}\,\,.
\]
}\end{rem}
\begin{defn}
\label{def:Laplace}Let $g$ be analytic in a domain and $\mbox{\ensuremath{\cal A_{\theta,\delta}^{\infty}\times\mathbf{D}\left(\mathbf{0},\mathbf{r}\right)} and let }\lambda>0,M>0$
such that
\[
\forall\left(t,\mathbf{y}\right)\in\cal A_{\theta,\delta}^{\infty}\times\mathbf{D}\left(\mathbf{0},\mathbf{r}\right),\,\abs{g\left(t,\mathbf{y}\right)}\leq M\exp\left(\lambda\abs t\right)\,\,\,\,.
\]
We define the \emph{Laplace transform} of $g$ in the direction $\theta$
as: 
\[
\cal L_{\theta}\left(g\right)\left(x,\mathbf{y}\right):=\int_{e^{i\theta}\ww R_{>0}}g\left(t,\mathbf{y}\right)\exp\left(-\frac{t}{x}\right)\frac{\mbox{d}t}{x}\,,
\]
which is absolutely convergent for all $x\in\ww C$ with $\Re\left(\frac{e^{i\theta}}{x}\right)>\lambda$
and $\mathbf{y\in\mathbf{D}\left(\mathbf{0},\mathbf{r}\right)}$,
and analytic with respect to $\left(x,\mathbf{y}\right)$ in this
domain.\end{defn}
\begin{rem}
As for the Borel transform, there also exists another definition of
the Laplace transform, that is:
\[
\widetilde{\cal L}_{\theta}\left(g\right)\left(x,\mathbf{y}\right):=\int_{e^{i\theta}\ww R_{>0}}g\left(t,\mathbf{y}\right)\exp\left(-\frac{t}{x}\right)\dd[t]\,\,.
\]
\end{rem}
\begin{prop}
\label{prop: Borel_Laplace_summation}A formal power series $\hat{f}\in\form{x,\mathbf{y}}$
is 1-summable in the direction $\theta$ if and only if there exists
a germ of a sectorial holomorphic function $f_{\theta}\in\cal O\left(\cal S_{\theta,\pi}\right)$
which admits $\hat{f}$ as Gevrey-1 asymptotic expansion in some $\cal S\in\cal S_{\theta,\pi}$.
Moreover, $f_{\theta}$ is unique $\big($as a germ in ${\displaystyle \cal O\left(\cal S_{\theta,\pi}\right)}$$\big)$
and in particular 
\[
f_{\theta}=\cal L_{\theta}\left(\cal B\left(\hat{f}\right)\right)\,\,.
\]
The function (germ) $f_{\theta}$ is called the \emph{1-sum of $\hat{f}$
in the direction $\theta$}.\end{prop}
\begin{rem}
With the second definitions of Borel and Laplace transforms given
above, we have a similar result for formal power series of the form
${\displaystyle \hat{f}\left(x,\mathbf{y}\right)=\sum_{k}f_{k}\left(\mathbf{y}\right)x^{k}}$
with: 
\[
f_{\theta}=\widetilde{\cal L}_{\theta}\left(\widetilde{\cal B}\right)\left(\hat{f}\right)+\hat{f}\left(0,\mathbf{y}\right)\,\,.
\]

\end{rem}
We recall the following well-known result.
\begin{lem}
\label{lem:diff_alg_and_summability}The set $\Sigma_{\theta}\subset\form{x,\mathbf{y}}$
of $1$-summable power series in the direction $\theta$ is an algebra
closed under partial derivatives. Moreover the map
\begin{eqnarray*}
\Sigma_{\theta} & \longrightarrow & \cal O\left(\cal S_{\theta,\pi}\right)\\
\hat{f} & \longmapsto & f_{\theta}
\end{eqnarray*}
is an injective morphism of differential algebras.\end{lem}
\begin{defn}
A formal power series $\hat{f}\in\form{x,\mathbf{y}}$ is called \emph{1-summable}
if it is 1-summable in all but a finite number of directions, called
\emph{Stokes directions}. In this case, if ${\displaystyle \theta_{1},\dots,\theta_{k}\in\nicefrac{\ww R}{2\pi\ww Z}}$
are the possible Stokes directions, we say that $\hat{f}$ is 1-summable
except for\textbf{ }$\theta_{1},\dots,\theta_{k}$.

More generally, we say that an $m-$uple $\left(f_{1},\dots,f_{m}\right)\in\form{x,\mathbf{y}}^{m}$
is Gevrey-1 (\emph{resp. }1-summable in direction $\theta$) if this
property holds for each component $f_{j},j=1,\dots,m$. Similarly,
a formal vector field (or\emph{ }diffeomorphism) is said to be Gevrey-1
(\emph{resp. }1-summable in direction $\theta$) if each one of its
components has this property.
\end{defn}
The following classical result deals with composition of 1-summable
power series (an elegant way to prove it is to use an important theorem
of Ramis-Sibuya).
\begin{prop}
\label{prop: compositon summable}Let $\hat{\Phi}\left(x,\mathbf{y}\right)\in\form{x,\mathbf{y}}$
be 1-summable in directions $\theta$ and $\theta-\pi$, and let $\Phi_{+}\left(x,\mathbf{y}\right)$
and $\Phi_{-}\left(x,\mathbf{y}\right)$ be its 1-sums directions
$\theta$ and $\theta-\pi$ respectively. Let also ${\displaystyle \hat{f}_{1}\left(x,\mathbf{z}\right),\dots,\hat{f}_{n}\left(x,\mathbf{z}\right)}$
be 1-summable in directions $\theta$, $\theta-\pi$, and $f_{1,+},\dots,f_{n,+}$,
and $f_{1,-},\dots,f_{n,-}$ be their 1-sums in directions $\theta$
and $\theta-\pi$ respectively. Assume that 
\begin{equation}
\hat{f}_{j}\left(0,\mathbf{0}\right)=0\mbox{, for all }j=1,\dots,n\,\,.\label{eq: condition composition sommable}
\end{equation}
 Then 
\[
\hat{\Psi}\left(x,\mathbf{z}\right):=\hat{\Phi}\left(x,\hat{f}_{1}\left(x,\mathbf{z}\right),\dots,\hat{f}_{n}\left(x,\mathbf{z}\right)\right)
\]
 is 1-summable in directions $\theta,\theta-\pi$, and its 1-sum in
the corresponding direction is 
\[
\Psi_{\pm}\left(x,\mathbf{z}\right):=\Phi_{\pm}\left(x,f_{1,\pm}\left(x,\mathbf{z}\right),\dots,f_{n,\pm}\left(x,\mathbf{z}\right)\right)\,\,\,,
\]
which is a germ of a sectorial holomorphic function in this direction.
\end{prop}
Consider $\hat{Y}$ a formal singular vector field at the origin and
a formal fibered diffeomorphism $\hat{\varphi}:\left(x,\mathbf{y}\right)\mapsto\left(x,\hat{\phi}\left(x,\mathbf{y}\right)\right)$.
Assume that both $\hat{Y}$ and $\hat{\varphi}$ are 1-summable in
directions $\theta$ and $\theta-\pi$, for some $\theta\in\ww R$,
and denote by $Y_{+},Y_{-}$ (\emph{resp. $\varphi_{+},\varphi_{-}$})
their 1-sums in directions $\theta$ and $\theta-\pi$ respectively.
As a consequence of Proposition \ref{prop: compositon summable} and
Lemma \ref{lem:diff_alg_and_summability}, we can state the following:
\begin{cor}
\label{cor: summability push-forward}Under the assumptions above,
$\hat{\varphi}_{*}\left(\hat{Y}\right)$ is 1-summable in both directions
$\theta$ and $\theta-\pi$, and its 1-sums in these directions are
$\varphi_{+}\left(Y_{+}\right)$ and $\varphi_{-}\left(Y_{-}\right)$
respectively.
\end{cor}

\subsection{\label{sub:Weak-Gevrey-1-power}Weak Gevrey-1 power series and weak
1-summability}

~

We present here a weaker notion of 1-summability that we will also
need. Any function $f\left(x,\mathbf{y}\right)$ analytic in a domain
$\cal U\times\mathbf{D}\left(\mathbf{0},\mathbf{r}\right)$, where
$\cal U\subset\ww C$ is open, and bounded in any domain $\cal U\times\overline{\mathbf{D}}\left(\mathbf{0},\mathbf{r'}\right)$
with $r'_{1}<r_{1},\dots,r'_{n}<r_{n}$, can be written 
\begin{equation}
f\left(x,\mathbf{y}\right)=\sum_{\mathbf{j}\in\ww N^{n}}F_{\mathbf{j}}\left(x\right)\mathbf{y}^{\mathbf{j}}\qquad,\label{eq: developpement selon y-1}
\end{equation}
where for all $\mathbf{j}\in\ww N^{n}$, $F_{\mathbf{j}}$ is analytic
and bounded on $\cal U$, and defined \emph{via }the Cauchy formula:
\[
F_{\mathbf{j}}\left(x\right)=\frac{1}{\left(2i\pi\right)^{n}}\int_{\abs{z_{1}}=r'_{1}}\dots\int_{\abs{z_{n}}=r'_{n}}\frac{f\left(x,\mathbf{z}\right)}{\left(z_{1}\right)^{j_{1}+1}\dots\left(z_{n}\right)^{j_{n}+1}}\mbox{d}z_{n}\dots\mbox{d}z_{1}\qquad.
\]
Notice that the convergence of the function series above is uniform
in every compact with respect to $x$ and $\mathbf{y}$.

In the same way, any formal power series $\hat{f}\left(x,\mathbf{y}\right)\in\form{x,\mathbf{y}}$
can be written as 
\[
{\displaystyle \hat{f}\left(x,\mathbf{y}\right)=\sum_{\mathbf{j}\in\ww N^{n}}\hat{F}_{\mathbf{j}}}\left(x\right)\mathbf{y}^{\mathbf{j}}\,\,.
\]

\begin{defn}
~
\begin{itemize}
\item The formal power series $\hat{f}$ is said to be \textbf{weakly Gevrey-1}
if for all $\mathbf{j}\in\ww N^{n}$, $\hat{F}_{\mathbf{j}}\left(x\right)\in\form x$
is a Gevrey-1 formal power series.
\item A function 
\[
{\displaystyle f\left(x,\mathbf{y}\right)=\sum_{\mathbf{j}\in\ww N^{n}}F_{\mathbf{j}}\left(x\right)\mathbf{y^{j}}}
\]
analytic and bounded in a domain $\sect r{\alpha}{\beta}\times\mathbf{D}\left(\mathbf{0},\mathbf{r}\right)$,
admits $\hat{f}$ as \textbf{weak Gevrey-1 asymptotic expansion} in
$x\in\sect r{\alpha}{\beta}$, if for all $\mathbf{j}\in\ww N^{n}$,
$F_{\mathbf{j}}$ admits $\hat{F}_{\mathbf{j}}$ as Gevrey-1 asymptotic
expansion in $\sect r{\alpha}{\beta}$.
\item The formal power series $\hat{f}$ is said to be \textbf{weakly 1-summable
in the direction $\theta\in\ww R$}, if the following conditions hold:

\begin{itemize}
\item for all $\mathbf{j}\in\ww N^{n}$, $\hat{F}_{\mathbf{j}}\left(x\right)\in\form x$
is 1-summable in the direction $\theta$, whose 1-sum in the direction
$\theta$ is denoted by $F_{\mathbf{j},\theta}$;
\item the series ${\displaystyle f_{\theta}\left(x,\mathbf{y}\right):=\sum_{\mathbf{j}\in\ww N^{n}}F_{\mathbf{j},\theta}\left(x\right)\mathbf{y^{j}}}$
defines a germ of a sectorial holomorphic function in a sectorial
neighborhood attached to the origin in the direction $\theta$ with
opening greater than $\pi$.
\end{itemize}

In this case, $f_{\theta}\left(x,\mathbf{y}\right)$ is called \textbf{the
weak 1-sum of $\hat{f}$ in the direction $\theta$}.

\end{itemize}
\end{defn}
As a consequence to the classical theory of summability and Gevrey
asymptotic expansions, we immediately have the following:
\begin{lem}
\label{lem: proprietes faibles}
\begin{enumerate}
\item The weak Gevrey-1 asymptotic expansion of an analytic function in
a domain $\sect r{\alpha}{\beta}\times\mathbf{D}\left(\mathbf{0},\mathbf{r}\right)$
is unique.
\item The weak 1-sum of a weak 1-summable formal power series in the direction
$\theta$, is unique as a germ in ${\displaystyle \cal O\left(\cal S_{\theta,\pi}\right)}$.
\item \label{enu: derivation faible}The set $\Sigma_{\theta}^{\left(\tx{weak}\right)}\subset\form{x,\mathbf{y}}$
of weakly $1$-summable power series in the direction $\theta$ is
an algebra closed under partial derivatives. Moreover the map
\begin{eqnarray*}
\Sigma_{\theta}^{\left(\tx{weak}\right)} & \longrightarrow & \cal O\left(\cal S_{\theta,\pi}\right)\\
\hat{f} & \longmapsto & f_{\theta}
\end{eqnarray*}
is an injective morphism of differential algebras.
\end{enumerate}
\end{lem}
The following proposition is an analogue of Proposition \ref{prop: compositon summable}
for weak 1-summable formal power series, with the a stronger condition
instead of (\ref{eq: condition composition sommable}).
\begin{prop}
\label{prop: composition faible}Let 
\[
{\displaystyle \hat{\Phi}\left(x,\mathbf{y}\right)=\sum_{\mathbf{j}\in\ww N^{n}}\hat{\Phi}_{\mathbf{j}}\left(x\right)\mathbf{y^{j}}\in\form{x,\mathbf{y}}}
\]
and 
\[
{\displaystyle \hat{f}^{\left(k\right)}\left(x,\mathbf{z}\right)=\sum_{\mathbf{j}\in\ww N^{n}}\hat{F}_{\mathbf{j}}^{\left(k\right)}\left(x\right)\mathbf{z^{j}}\in\form{x,\mathbf{z}}}\,\,,
\]
for $k=1,\dots,n$, be n+1 formal power series which are weakly 1-summable
in directions $\theta$ and $\theta-\pi$. Let us denote by $\Phi_{+},f_{+}^{\left(1\right)},\dots,f_{+}^{\left(n\right)}$
$\big($\emph{resp. }$\Phi_{-},f_{-}^{\left(1\right)},\dots,f_{-}^{\left(n\right)}$$\big)$
their respective weak 1-sums in the direction $\theta$ (\emph{resp.}
$\theta-\pi$). Assume that $\hat{F}_{\mathbf{0}}^{\left(k\right)}=0$
for all $k=1,\dots,n$. Then, 
\[
\hat{\Psi}\left(x,\mathbf{z}\right):=\hat{\Phi}\left(x,\hat{f}^{\left(1\right)}\left(x,\mathbf{z}\right),\dots,\hat{f}^{\left(n\right)}\left(x,\mathbf{z}\right)\right)
\]
 is weakly 1-summable directions $\theta$ and $\theta-\pi$, and
its 1-sum in the corresponding direction is 
\[
\Psi_{\pm}\left(x,\mathbf{z}\right)=\Phi_{\pm}\left(x,f_{\pm}^{\left(1\right)}\left(x,\mathbf{z}\right),\dots,f_{\pm}^{\left(n\right)}\left(x,\mathbf{z}\right)\right)\,\,\,,
\]
which is a germ of a sectorial holomorphic function in this direction
with opening $\pi$.\end{prop}
\begin{proof}
First of all, 
\[
\hat{\Psi}\left(x,\mathbf{z}\right):=\hat{\Phi}\left(x,\hat{f}^{\left(1\right)}\left(x,\mathbf{z}\right),\dots,\hat{f}^{\left(n\right)}\left(x,\mathbf{z}\right)\right)
\]
 is well defined as formal power series since for all $k=1,\dots,n$,
$\hat{F}_{\mathbf{0}}^{\left(k\right)}=0$. It is also clear that
\[
\Psi_{\pm}\left(x,\mathbf{z}\right):={\displaystyle \Phi_{\pm}\left(x,f_{\pm}^{\left(1\right)}\left(x,\mathbf{z}\right),\dots,f_{\pm}^{\left(n\right)}\left(x,\mathbf{z}\right)\right)}
\]
 is an analytic in a domain $\cal S_{+}\in\cal S_{\theta,\pi}$ (\emph{resp.
}$\cal S_{-}\in\cal S_{\theta-\pi,\pi}$), because $f_{\pm}^{\left(k\right)}\left(x,\mathbf{0}\right)=0$
for all $k=1,\dots,n$. Finally, we check that $\Psi_{\pm}$ admits
$\hat{\Psi}$ as weak Gevrey-1 asymptotic expansion in $\cal S_{\pm}$.
Indeed:
\begin{eqnarray*}
\Psi_{\pm}\left(x,\mathbf{z}\right) & = & \Phi_{\pm}\left(x,f_{\pm}^{\left(1\right)}\left(x,\mathbf{z}\right),\dots,f_{\pm}^{\left(n\right)}\left(x,\mathbf{z}\right)\right)\\
 & = & \sum_{\mathbf{j}\in\ww N^{n}}\left(\Phi_{\mathbf{j}}\right)_{\pm}\left(x\right)\left(f_{\pm}^{\left(1\right)}\left(x,\mathbf{z}\right)\right)^{j_{1}}\dots\left(f_{\pm}^{\left(1\right)}\left(x,\mathbf{z}\right)\right)^{j_{n}}\\
 & = & \sum_{\mathbf{j}\in\ww N^{n}}\left(\Phi_{\mathbf{j}}\right)_{\pm}\left(x\right)\left(\sum_{\abs{\mathbf{l}}\geq1}\left(F_{\mathbf{l}}^{\left(1\right)}\right)_{\pm}\left(x\right)\mathbf{z^{l}}\right)^{j_{1}}\dots\\
 &  & \,\,\,\,\,\,\dots\left(\sum_{\abs{\mathbf{l}}\geq1}\left(F_{\mathbf{l}}^{\left(n\right)}\right)_{\pm}\left(x\right)\mathbf{z^{l}}\right)^{j_{n}}\\
 & = & \sum_{\mathbf{j}\in\ww N^{n}}\left(\Psi_{\mathbf{j}}\right)_{\pm}\left(x\right)\mathbf{y^{j}}
\end{eqnarray*}
where for all $\mathbf{j}\in\ww N^{n}$, $\left(\Psi_{\mathbf{j}}\right)_{\pm}\left(x\right)$
is obtained as a finite number of additions and products of the $\left(\Phi_{\mathbf{k}}\right)_{\pm}$,$\left(F_{\mathbf{k}}^{\left(1\right)}\right)_{\pm}$,$\dots$,$\left(F_{\mathbf{k}}^{\left(n\right)}\right)_{\pm}$,
$\abs{\mathbf{k}}\leq\abs{\mathbf{l}}$. The same computation is valid
for the associated formal power series, and allows us to define the
$\hat{\Psi}_{\mathbf{j}}\left(x\right)$, for all $\mathbf{j}\in\ww N^{n}$.
Then, each $\left(\Psi_{\mathbf{j}}\right)_{\pm}$ has $\hat{\Psi}_{\mathbf{j}}$
as Gevrey-1 asymptotic expansion in $\cal S_{\pm}$.
\end{proof}
As a consequence of Proposition \ref{prop: composition faible} and
Lemma \ref{lem: proprietes faibles}, we have an analogue version
of Corollary (\ref{cor: summability push-forward}) in the weak 1-summable
case. Consider $\hat{Y}$ a formal singular vector field at the origin
and a formal fibered diffeomorphism $\hat{\varphi}:\left(x,\mathbf{y}\right)\mapsto\left(x,\hat{\phi}\left(x,\mathbf{y}\right)\right)$
such that $\hat{\phi}\left(x,\mathbf{0}\right)=\mathbf{0}$. Assume
that both $\hat{Y}$ and $\hat{\varphi}$ are weakly 1-summable in
directions $\theta$ and $\theta-\pi$, for some $\theta\in\ww R$,
and denote by $Y_{+},Y_{-}$ (\emph{resp. $\varphi_{+},\varphi_{-}$})
their weak 1-sums in directions $\theta$ and $\theta-\pi$ respectively.
\begin{cor}
\label{cor: summability push-forward faible} Under the assumptions
above, $\hat{\varphi}_{*}\left(\hat{Y}\right)$ is weakly 1-summable
in both directions $\theta$ and $\theta-\pi$, and its 1-sums in
these directions are $\varphi_{+}\left(Y_{+}\right)$ and $\varphi_{-}\left(Y_{-}\right)$
respectively.
\end{cor}

\subsection{Weak 1-summability \emph{versus} 1-summability}

~

As in the previous subsection, let a formal power series $\hat{f}\left(x,\mathbf{y}\right)\in\form{x,\mathbf{y}}$
which is written as 
\[
{\displaystyle \hat{f}\left(x,\mathbf{y}\right)=\sum_{\mathbf{j}\in\ww N^{n}}\hat{F}_{\mathbf{j}}}\left(x\right)\mathbf{y}^{\mathbf{j}}\,\,,
\]
so that its Borel transform is 
\[
{\displaystyle \cal B\left(\hat{f}\right)\left(t,\mathbf{y}\right)=\sum_{\mathbf{j}\in\ww N^{n}}\cal B\left(\hat{F}_{\mathbf{j}}\right)\left(t\right)\mathbf{y^{j}}\,\,.}
\]
The next lemma is immediate.
\begin{lem}
\label{lem: defition equivalent 1-summable}~
\begin{enumerate}
\item The power series $\cal B\left(\hat{f}\right)\left(t,\mathbf{y}\right)$
is convergent in a neighborhood of the origin in $\ww C^{n+1}$ if
and only if the $\cal B\left(\hat{F}_{\mathbf{j}}\right),\,\mathbf{j}\in\ww N^{n}$,
are all analytic and bounded in a same disc $\mbox{D}\left(0,\rho\right)$,
$\rho>0$, and if there exists $B,L>0$ such that for all $\mathbf{j}\in\ww N^{n}$,
${\displaystyle \underset{t\in\mbox{D}\left(0,\rho\right)}{\sup}\abs{\cal B\left(\hat{F}_{\mathbf{j}}\right)\left(t\right)}\leq L.B^{\abs{\mathbf{j}}}}$. 
\item If $1.$ is satisfied, then $\cal B\left(\hat{f}\right)$ can be analytically
continued to a domain $\cal A_{\theta,\delta}^{\infty}\times\mathbf{D}\left(\mathbf{0},\mathbf{r}\right)$
if and only if for all $\mathbf{j}\in\ww N^{n}$, $\cal B\left(\hat{F}_{\mathbf{j}}\right)$
can be analytically continued to $\cal A_{\theta,\delta}^{\infty}$
and if for all compact $K\subset\cal A_{\theta,\delta}^{\infty}$,
there exists $B,L>0$ such that for all $\mathbf{j}\in\ww N^{n}$,
${\displaystyle \underset{t\in K}{\sup}\abs{\cal B\left(\hat{F}_{\mathbf{j}}\right)\left(t\right)}\leq L.B^{\abs{\mathbf{j}}}}$.
\item If and $1.$ and $2.$ are satisfied, then there exists $\lambda,M>0$
such that:
\[
\forall\left(t,\mathbf{y}\right)\in\cal A_{\theta,\delta}^{\infty}\times\mathbf{D}\left(\mathbf{0},\mathbf{r}\right),\,\abs{\cal B\left(\hat{f}\right)\left(t,\mathbf{y}\right)}\leq M.\exp\left(\lambda\abs t\right)
\]
if and only if there exists $\lambda,L,B>0$ such that for all $\mathbf{j}\in\ww N^{n}$,
\[
\forall t\in\cal A_{\theta,\delta}^{\infty},\,\abs{\cal B\left(\hat{F}_{\mathbf{j}}\right)\left(t\right)}\leq L.B^{\abs{\mathbf{j}}}\exp\left(\lambda\abs t\right)\qquad.
\]

\end{enumerate}
\end{lem}
\begin{rem}
~
\begin{enumerate}
\item Condition $1.$ above states that the formal power series $\hat{f}$
is Gevrey-1.
\item As usual, there exists an equivalent lemma for the second definitions
of the Borel transform (see Remark \ref{rem: deifinition bis transformee de Borel}).
\end{enumerate}
\end{rem}
The following corollary gives a link between 1-summability and weak
1-summability.
\begin{cor}
\label{cor: faible et forte sommabilite}Let 
\begin{eqnarray*}
{\displaystyle \hat{f}\left(x,\mathbf{y}\right)} & = & \sum_{\mathbf{j}\in\ww N^{n}}\hat{F}_{\mathbf{j}}\left(x\right)\mathbf{y^{j}}\in\form{x,\mathbf{y}}
\end{eqnarray*}
 be a formal power series. Then, $\hat{f}$ is 1-summable in the direction
$\theta\in\ww R$, of 1-sum ${\displaystyle f\in\cal O\left(\cal S_{\theta,\pi}\right)}$,
if and only if the following two conditions hold:
\begin{itemize}
\item $\hat{f}$ is weakly 1-summable in the direction $\theta$;
\item there exists $\lambda,\delta,\rho$ such that for all $\mathbf{j}\in\ww N^{n}$,
${\displaystyle \norm{\hat{F}_{\mathbf{j}}}_{\lambda,\theta,\delta,\rho}<\infty}$
and the power series ${\displaystyle \sum_{\mathbf{j}\in\ww N^{n}}\norm{\hat{F}_{\mathbf{j}}}_{\lambda,\theta,\delta,\rho}\mathbf{y}^{\mathbf{j}}}$\emph{
}\textup{\emph{is convergent near the origin of $\ww C^{n}$.}}
\end{itemize}
\end{cor}
\begin{proof}
This is an immediate consequence of Lemma \ref{lem: defition equivalent 1-summable}.\end{proof}
\begin{rem}
We can replace the norm $\norm{\cdot}_{\lambda,\theta,\delta,\rho}$
by $\norm{\cdot}_{\lambda,\theta,\delta,\rho}^{\tx{bis}}$in the second
point of the above corollary.
\end{rem}
Notice that there exists formal power series which are weakly 1-summable
in some direction but which are not Gevrey-1: for instance, the series
\[
{\displaystyle \hat{f}:=\sum_{j}\hat{F}_{j}\left(x\right)y^{j}}\,\,,
\]
where for all $j\in\ww N$, $\hat{F}_{j}\left(x\right)$ is such that
${\displaystyle \cal B\left(\hat{F}_{j}\right)\left(t\right)=\frac{1}{t+\frac{1}{j}}}$,
is weakly 1-summable in the direction $0\in\ww R$, but is \emph{not}
Gevrey-1, since $\cal B\left(\hat{F}_{j}\right)$ has a pole in every
${\displaystyle -\frac{1}{j}\underset{j\rightarrow+\infty}{\longrightarrow}0}$.

\subsection{Some useful tools on 1-summability of solutions of singular linear
differential equations}

~

For future reuse, we give here two results on the 1-summability of
formal solutions to some singular linear differential equations with
1-summable right hand side, which generalize (and precise) a similar
result proved in \cite{MR82} (Proposition p. 126). The authors use
a norm $\norm{\cdot}_{\beta}$, but we will need to use a norm $\norm{\cdot}_{\beta}^{\tx{bis}}$
later (in the proof of Proposition \ref{prop: preparation}).
\begin{prop}
\label{prop: solution borel sommable precise}Let $\hat{b}$ be a
formal power series 1-summable in the direction $\theta$; consider
a domain $\Delta_{\theta,\delta,\rho}$ as in Definition~\ref{def: 1_summability}.
Let use denote by $b_{\theta}$ its 1-sum in this direction $\theta$.
Let us also fix $\alpha,k\in\ww C$.
\begin{enumerate}
\item \label{enu:Assume--and}Assume $\norm{\hat{b}}_{\beta}^{\tx{bis}}<+\infty$
and that $k\in\ww C\backslash\acc 0$ is such that $d_{k}:=\mbox{dist}\left(-k,\Delta_{\theta,\delta,\rho}\right)>0$
and 
\[
\beta d_{k}>C\abs{\alpha k}\qquad,
\]
where $C>0$ is a constant large enough, independent from parameters
$k,\beta,\theta,\delta,\rho$ (for instance, one can take $C=\frac{2\exp\left(2\right)}{5}+5$).
Then, the irregular singular differential equation 
\begin{equation}
x^{2}\ddd ax\left(x\right)+\left(1+\alpha x\right)ka\left(x\right)=\hat{b}\left(x\right)\label{eq: equation borel sommable}
\end{equation}
has a unique formal solution $\hat{a}$ such that $\hat{a}\left(0\right)=\frac{1}{k}\hat{b}\left(0\right)$.
Moreover, $\hat{a}$ is 1-summable in the direction $\theta$, and
\begin{equation}
\norm{\hat{a}}_{\beta}\leq\frac{\beta}{\beta d_{k}-C\abs{\alpha k}}\norm{\hat{b}}_{\beta}\qquad.\label{eq: inegalite norme borel}
\end{equation}
Finally, the 1-sum $a_{\theta}$ of $\hat{a}$ in the direction $\theta$
is the only solution to
\[
x^{2}\ddd{a_{\theta}}x\left(x\right)+\left(1+\alpha x\right)ka_{\theta}\left(x\right)=b_{\theta}\left(x\right)
\]
which is bounded in some $S_{\theta,\pi}\in\germsect{\theta}{\pi}$.
\item \label{enu:Assume--and-1}Assume $\norm{\hat{b}}_{\beta}<+\infty$
and that $\Re\left(k\right)>0$. Then the regular singular differential
equation 
\begin{equation}
x\ddd ax\left(x\right)+ka\left(x\right)=\hat{b}\left(x\right)\label{eq: regular diff equation}
\end{equation}
admits a unique formal solution $\hat{a}$ which is also 1-summable
in the direction $\theta$, of 1-sum $a_{\theta}$. Moreover, $a_{\theta}$
is the only germ of solution to the differential equation 
\[
x\ddd ax\left(x\right)+ka\left(x\right)=b_{\theta}\left(x\right)
\]
which is bounded in some $S_{\theta,\pi}\in\germsect{\theta}{\pi}$. 
\end{enumerate}
\end{prop}
\begin{proof}
~

(\ref{enu:Assume--and}) Since $\hat{b}$ is 1-summable in the direction
$\theta$, we can choose $\rho>0$ and $\delta>0$ such that $\widetilde{\cal B}\left(\hat{b}\right)$
can be analytically continued to (and is bounded in) any domain of
the form $\Delta_{\theta,\delta,\rho}\cap\overline{\mbox{D}}\left(0,R\right)$,
$R>0$.\\
Let us apply the Borel transform $\widetilde{\cal B}$ to equation
(\ref{eq: equation borel sommable}): we obtain 
\begin{equation}
\left(t+k\right)\widetilde{\cal B}\left(\hat{a}\right)\left(t\right)+\alpha k\int_{0}^{t}\widetilde{\cal B}\left(\hat{a}\right)\left(s\right)\mbox{d}s=\widetilde{\cal B}\left(\hat{b}\right)\left(t\right)\qquad.\label{eq: borel equa diff}
\end{equation}
The derivative with respect to $t$ of this equation shows that $\widetilde{\cal B}\left(\hat{a}\right)$
is solution of a linear differential equation, with only one (regular)
singularity at $t=-k$ (but this singularity is not in $\Delta_{\theta,\delta,\rho}$
by assumption):
\[
\left(t+k\right)\ddd{\widetilde{\cal B}\left(\hat{a}\right)}t\left(t\right)+\left(1+\alpha k\right)\widetilde{\cal B}\left(\hat{a}\right)\left(t\right)=\ddd{\widetilde{\cal B}\left(\hat{b}\right)}t\left(t\right)\qquad.
\]
 Since $\widetilde{\cal B}\left(\hat{b}\right)$ can be analytically
continued to $\Delta_{\theta,\delta,\rho}$, the same goes for ${\displaystyle \ddd{\widetilde{\cal B}\left(\hat{b}\right)}t\left(t\right)}$
and then for $\widetilde{\cal B}\left(\hat{a}\right)$. Since ${\displaystyle \widetilde{\cal B}\left(\hat{a}\right)\left(0\right)=\frac{\widetilde{\cal B}\left(\hat{b}\right)\left(0\right)}{k}=\frac{\hat{b}'\left(0\right)}{k}}$,
we can write: {\small{}
\begin{eqnarray*}
\widetilde{\cal B}\left(\hat{a}\right)\left(t\right) & = & \left(t+k\right)^{-1-\alpha k}\left(\hat{b}'\left(0\right).k^{\alpha k}+\int_{0}^{t}\ddd{\widetilde{\cal B}\left(\hat{b}\right)}s\left(s\right).\left(s+k\right)^{\alpha k}\mbox{d}s\right)\\
 & = & \left(t+k\right)^{-1-\alpha k}\Bigg(\hat{b}'\left(0\right).k^{\alpha k}+\widetilde{\cal B}\left(\hat{b}\right)\left(t\right).\left(t+k\right)^{\alpha k}-\widetilde{\cal B}\left(\hat{b}\right)\left(0\right).k^{\alpha k}\\
 &  & \,\,\,\,\,\,-\alpha k\int_{0}^{t}\widetilde{\cal B}\left(\hat{b}\right)\left(s\right).\left(s+k\right)^{\alpha k-1}\mbox{d}s\Bigg)\\
 & = & \left(t+k\right)^{-1-\alpha k}\Bigg(\widetilde{\cal B}\left(\hat{b}\right)\left(t\right).\left(t+k\right)^{\alpha k}\\
 &  & \,\,\,\,\,\,-\alpha k\int_{0}^{t}\widetilde{\cal B}\left(\hat{b}\right)\left(s\right).\left(s+k\right)^{\alpha k-1}\mbox{d}s\Bigg)\\
\widetilde{\cal B}\left(\hat{a}\right) & = & \frac{\widetilde{\cal B}\left(\hat{b}\right)\left(t\right)}{\left(t+k\right)}-\alpha k.\left(t+k\right)^{-1-\alpha k}\int_{0}^{t}\widetilde{\cal B}\left(\hat{b}\right)\left(s\right).\left(s+k\right)^{\alpha k-1}\mbox{d}s\,\,.
\end{eqnarray*}
}The fact that $\widetilde{\cal B}\left(\hat{b}\right)$ is bounded
in any domain of the form $\Delta_{\theta,\delta,\rho}\cap\overline{\mbox{D}}\left(0,R\right)$,
$R>0$, implies that the same goes for $\widetilde{\cal B}\left(\hat{a}\right)$.
Let us prove inequality (\ref{eq: inegalite norme borel}). For all
$R>0$, for all Gevrey-1 series $\hat{f}\in\form{x,\mathbf{y}}$ such
that $\cal B\left(\hat{f}\right)$ can be analytically continued to
$\Delta_{\theta,\delta,r}$, we set: 
\[
\norm{\hat{f}}_{\beta,R}^{\tx{bis}}:=\underset{t\in\Delta_{\theta,\delta,\rho}\cap\mbox{\ensuremath{\overline{\mbox{D}}\left(0,R\right)}}}{\sup}\acc{\abs{\widetilde{\cal B}\left(\hat{f}\right)\left(t\right)\left(1+\beta^{2}\abs t^{2}\right)\exp\left(-\beta\abs t\right)}}\in\ww R\cup\acc{\infty}\,\,.
\]
Notice that ${\displaystyle \norm{\hat{f}}_{\beta}^{\tx{bis}}=\underset{R>0}{\sup}\acc{\norm{\hat{f}}_{\beta,R}^{\tx{bis}}}}$
for all $\hat{f}$ as above, and that for all $R>0$, $\norm{\hat{a}}_{\beta,R}^{\tx{bis}}<+\infty$,
since $\widetilde{\cal B}\left(\hat{a}\right)$ is bounded in any
domain of the form $\Delta_{\theta,\delta,\rho}\cap\overline{\mbox{D}}\left(0,R\right)$.
Fix some $R>0$, and let $t\in\Delta_{\theta,\delta,\rho}\cap\mbox{\ensuremath{\overline{\mbox{D}}\left(0,R\right)}}$.
From equation (\ref{eq: borel equa diff}) we obtain 
\begin{eqnarray*}
\widetilde{\cal B}\left(\hat{a}\right)\left(t\right) & = & \frac{1}{\left(t+k\right)}\left(\widetilde{{\cal B}}\left(\hat{b}\right)\left(t\right)-\alpha k\int_{0}^{t}\widetilde{{\cal B}}\left(\hat{a}\right)\left(s\right)\mbox{d}s\right)
\end{eqnarray*}
an then
\begin{eqnarray*}
\abs{\widetilde{\cal B}\left(\hat{a}\right)\left(t\right)} & \leq & \frac{1}{\abs{t+k}}\cro{\norm{\hat{b}}_{\beta}^{\tx{bis}}\frac{\exp\left(\beta\abs t\right)}{1+\beta^{2}\abs t^{2}}+\abs{\alpha k}.\norm{\hat{a}}_{\beta,R}^{\tx{bis}}\int_{0}^{\abs t}\frac{\exp\left(\beta u\right)}{1+\beta^{2}u^{2}}\dd[u]}\\
 & \leq & \frac{1}{d_{k}}\frac{\exp\left(\beta\abs t\right)}{1+\beta^{2}\abs t^{2}}\cro{\norm{\hat{b}}_{\beta}^{\tx{bis}}+\abs{\alpha k}\norm{\hat{a}}_{\beta,R}^{\tx{bis}}\frac{C}{\beta}}\,\,,
\end{eqnarray*}
with $C=\frac{2\exp\left(2\right)}{5}+5$. Here we use the following:
\begin{fact}
There exists a constant $C>0$ (\emph{e.g.} $C=\frac{2\exp\left(2\right)}{5}+5$),
such that for all $\beta>0$, we have:
\[
\forall t\geq0,\int_{0}^{t}\frac{\exp\left(\beta u\right)}{1+\beta^{2}u^{2}}\tx du\leq\frac{C}{\beta}\frac{\exp\left(\beta t\right)}{1+\beta^{2}t^{2}}\,\,.
\]
\end{fact}
\begin{proof}
Let $F:u\mapsto\frac{\exp\left(\beta u\right)}{1+\beta^{2}u^{2}}$,
for $u\geq0$. For $t\in\cro{0,\frac{2}{\beta}}$, we have:
\[
\int_{0}^{t}F\left(u\right)\tx du\leq\frac{\exp\left(2\right)}{5}.\frac{2}{\beta}\,\,,
\]
since $F$ is an increasing function over $\wr_{+}$: 
\[
F'\left(u\right)=\beta F\left(u\right).\frac{\left(1-\beta u\right)^{2}}{1+\beta^{2}u^{2}}\geq0\,\,.
\]
Moreover for all $t\geq0$, we have $F\left(t\right)\geq F\left(0\right)=1$.
Hence for all $t\in\cro{0,\frac{2}{\beta}}$:
\[
\int_{0}^{t}F\left(u\right)\tx du\leq\frac{\exp\left(2\right)}{5}.\frac{2}{\beta}.F\left(t\right)\,\,.
\]
For $t\geq\frac{2}{\beta}$, the following inequality holds:
\begin{eqnarray*}
\int_{0}^{t}F\left(u\right)\tx du & \leq & \frac{\exp\left(2\right)}{5}.\frac{2}{\beta}F\left(t\right)+\int_{\frac{2}{\beta}}^{t}F\left(u\right)\tx du\,\,.
\end{eqnarray*}
In addition, if $u\geq\frac{\beta}{2}$, then: 
\begin{eqnarray*}
\frac{\left(1-\beta u\right)^{2}}{1+\beta^{2}u^{2}} & \geq & \frac{1}{5}\,\,,
\end{eqnarray*}
Therefore, for all $u\geq\frac{\beta}{2}$:
\[
F'\left(u\right)=\beta F\left(u\right).\frac{\left(1-\beta u\right)^{2}}{1+\beta^{2}u^{2}}\geq\frac{\beta}{5}F\left(u\right)\,\,.
\]
Hence:
\begin{eqnarray*}
\int_{0}^{t}F\left(u\right)\tx du & \leq & \int_{0}^{\frac{2}{\beta}}F\left(u\right)\tx du+\int_{\frac{2}{\beta}}^{t}F\left(u\right)\tx du\,\,.\\
 & \leq & \frac{\exp\left(2\right)}{5}.\frac{2}{\beta}F\left(t\right)+\frac{5}{\beta}\int_{\frac{2}{\beta}}^{t}F'\left(u\right)\tx du\\
 & \leq & \frac{F\left(t\right)}{5\beta}.\left(2\exp\left(2\right)+25\right)\,\,.
\end{eqnarray*}

\end{proof}
Let us go back to the proof of the lemma. Finally, we have: 
\begin{eqnarray*}
\norm{\hat{a}}_{\beta,R}^{\tx{bis}} & \leq & \frac{1}{d_{k}}\cro{\norm{\hat{b}}_{\beta}^{\tx{bis}}+\frac{C.\abs{\alpha k}.\norm{\hat{a}}_{\beta,R}^{\tx{bis}}}{\beta}}\,\,,
\end{eqnarray*}
and consequently:
\[
\norm{\hat{a}}_{\beta,R}^{\tx{bis}}\leq\frac{\beta}{\beta d_{k}-C\abs{\alpha k}}\norm{\hat{b}}_{\beta}^{\tx{bis}}\,\,.
\]
As a conclusion:
\[
\norm{\hat{a}}_{\beta}^{\tx{bis}}\leq\frac{\beta}{\beta d_{k}-C\abs{\alpha k}}\norm{\hat{b}}_{\beta}^{\tx{bis}}\,\,,
\]
and $a_{\theta}$ is the 1-sum of $\hat{a}$ in the direction $\theta$.\bigskip{}
(\ref{enu:Assume--and-1}) Let us write ${\displaystyle \hat{b}\left(x\right)=\sum_{j\geq0}b_{j}x^{j}}$.
A direct computation shows that the only formal solution to equation
$\left(\mbox{\ref{eq: regular diff equation}}\right)$ is ${\displaystyle \hat{a}\left(x\right)=\sum_{j\geq0}a_{j}x^{j}}$
where for all $j\in\ww N$, ${\displaystyle a_{j}=\frac{b_{j}}{j+k}}$:
it exists since $k\notin\ww Z_{\leq0}$, and then $k+j\neq0$. In
particular, we see immediately that $\hat{a}$ is Gevrey-1, because
the same goes for $\hat{b}$. In other words, the Borel transform
$\cal B\left(\hat{a}\right)$ is analytic in some disc $D\left(0,\rho\right)$,
$\rho>0$. In $D\left(0,\rho\right)$, $\cal B\left(\hat{a}\right)$
satisfies:
\begin{eqnarray*}
t\ddd{\cal B\left(\hat{a}\right)}t\left(t\right)+k\cal B\left(\hat{a}\right)\left(t\right) & = & \cal B\left(\hat{b}\right)\left(t\right)\,\,.
\end{eqnarray*}
The general solution near the origin to this equation is 
\[
y\left(t\right)=\frac{c}{t^{k}}+\frac{1}{t^{k}}\int_{0}^{t}\cal B\left(\hat{b}\right)\left(s\right)s^{k-1}\mbox{d}s\,\,,\, c\in\ww C.
\]
In particular, the only solution analytic in $D\left(0,\rho\right)$
is the one with $c=0$, \emph{i.e.}
\[
\cal B\left(\hat{a}\right)\left(t\right)=\frac{1}{t^{k}}\int_{0}^{t}\cal B\left(\hat{b}\right)\left(s\right)s^{k-1}\mbox{d}s\,\,.
\]
 Since $\cal B\left(\hat{b}\right)$ can be analytically continued
to an infinite domain that have denoted by $\Delta_{\theta,\delta,\rho}$
bisected by $\ww R_{+}e^{i\theta}$ (because $\hat{b}$ is 1-summable
in the direction $\theta$), $\cal B\left(\hat{a}\right)$ can also
be analytically continued to the same domain. Moreover, there exists
$\beta>0$ such that $\norm{\hat{b}}_{\beta}<+\infty$, \emph{i.e.
$\forall t\in\Delta_{\theta,\delta,\rho}$}:\emph{
\[
\abs{\cal B\left(\hat{b}\right)\left(t\right)}\leq\norm{\hat{b}}_{\beta}\exp\left(\beta\abs t\right)\,\,.
\]
}Thus, for all $t\in\Delta_{\theta,\delta,\rho}$, we have:
\begin{eqnarray*}
\abs{\exp\left(-\beta\abs t\right)\cal B\left(\hat{a}\right)\left(t\right)} & \leq & \frac{1}{\abs{t^{k}}}\int_{0}^{\abs t}\abs{\exp\left(-\beta\abs t\right)}\abs{\cal B\left(\hat{b}\right)\left(s\mbox{e}^{i\arg\left(t\right)}\right)}\abs{s^{k-1}\mbox{e}^{i\left(k-1\right)\arg\left(t\right)}}\mbox{d}s\\
 & \leq & \frac{1}{\abs t^{\Re\left(k\right)}}\int_{0}^{\abs t}\abs{\exp\left(-\beta s\right)}\abs{\cal B\left(\hat{b}\right)\left(s\mbox{e}^{i\arg\left(t\right)}\right)}s^{\Re\left(k\right)-1}\mbox{d}s\\
 & \leq & \frac{\norm{\hat{b}}_{\beta}}{\abs t^{\Re\left(k\right)}}\int_{0}^{\abs t}s^{\Re\left(k\right)-1}\mbox{d}s\\
 & = & \frac{\norm{\hat{b}}_{\beta}}{\Re\left(k\right)}\,\,.
\end{eqnarray*}
Thus, $\hat{a}$ is 1-summable in the direction $\theta$.
\end{proof}

\section{1-summable preparation up to any order $N$}

The aim of this section is to prove that we can always formally conjugate
a non-degenerate doubly-resonant saddle-node, which is also div-integrable,
to its normal form up to a remainder of order $\tx O\left(x^{N}\right)$
for every $N\in\ww N_{>0}$. Moreover, we prove that this conjugacy
is in fact 1-summable in every direction $\theta\neq\arg\left(\pm\lambda\right)$,
hence analytic over sectorial domains of opening at least $\pi$.
\begin{prop}
\label{prop: forme pr=0000E9par=0000E9e ordre N}Let $Y\in\sndiag$
be a non-degenerate diagonal doubly-resonant saddle-node which is
div-integrable, with $\tx D_{0}Y=\tx{diag}\left(0,-\lambda,\lambda\right)$,
$\lambda\neq0$. Then, for all $N\in\ww N_{>0}$, there exists a formal
fibered diffeomorphism $\Psi^{\left(N\right)}\in\fdiffformid$ tangent
to the identity and 1-summable in every direction $\theta\neq\arg\left(\pm\lambda\right)$
such that:
\begin{eqnarray*}
\left(\Psi^{\left(N\right)}\right)_{*}\left(Y\right) & = & x^{2}\pp x+\left(\left(-\left(\lambda+d^{\left(N\right)}\left(y_{1}y_{2}\right)\right)+a_{1}x\right)+x^{N}F_{1}^{\left(N\right)}\left(x,\mathbf{y}\right)\right)y_{1}\pp{y_{1}}\\
 &  & +\left(\left(\lambda+d^{\left(N\right)}\left(y_{1}y_{2}\right)+a_{2}x\right)+x^{N}F_{2}^{\left(N\right)}\left(x,\mathbf{y}\right)\right)y_{2}\pp{y_{2}}\\
 & =: & Y^{\left(N\right)}\,\,\,,
\end{eqnarray*}
where $\lambda\in\ww C^{*}$, ${\displaystyle \left(a_{1}+a_{2}\right)=\tx{res}\left(Y\right)\in\ww C\backslash\ww Q_{\leq0}}$,
$d^{\left(N\right)}\left(v\right)\in v\germ v$ is an analytic germ
vanishing at the origin, and $F_{1}^{\left(N\right)},F_{2}^{\left(N\right)}\in\form{x,\mathbf{y}}$
are 1-summable in the direction $\theta$, and of order at least one
with respect to $\mathbf{y}$. Moreover, one can choose $d^{\left(2\right)}=\dots=d^{\left(N\right)}$
for all $N\geq2$. \end{prop}
\begin{defn}
A vector field $Y^{\left(N\right)}$ as is the proposition above is
said to be \emph{normalized up to order $N$}.\end{defn}
\begin{rem}
~
\begin{enumerate}
\item Observe that this result does not require the more restrictive assumption
of being ``strictly non-degenerate'' (\emph{i.e.} $\Re\left(a_{1}+a_{2}\right)>0$)
.
\item As a consequence of Corollary \ref{cor: summability push-forward},
the 1-sum $\Psi_{\theta}^{\left(N\right)}$ of $\Psi^{\left(N\right)}$
in the direction $\theta$ is a germ of sectorial fibered diffeomorphism
tangent to the identity, \emph{i.e.} ${\displaystyle \Psi_{\theta}^{\left(N\right)}\in\diffsect[\theta][\pi]}$,
which conjugates $Y$ to the 1-sum $Y_{\theta}^{\left(N\right)}$
of $Y^{\left(N\right)}$ in the direction $\theta$.
\end{enumerate}
\end{rem}
In order to prove this result we will proceed in several steps and
use after each step Proposition \ref{prop: compositon summable} and
Corollary \ref{cor: summability push-forward} in order to prove the
1-summability in every direction $\theta\neq\arg\left(\pm\lambda\right)$
of the different objects. First, we will normalize analytically the
vector field restricted to $\acc{x=0}$. Then, we will straighten
the formal separatrix to $\acc{y_{1}=y_{2}=0}$ in suitable coordinates.
Next, we will simplify the linear terms with respect to $\mathbf{y}$.
After that, we will straighten two invariant hypersurfaces to $\acc{y_{1}=0}$
and $\acc{y_{2}=0}$. Finally, we will conjugate the vector field
to its final normal form up to remaining terms of order $\tx O\left(x^{N}\right)$.

\subsection{Analytic normalization on the hyperplane $\protect\acc{x=0}$}

~

\subsubsection{Transversally Hamiltonian \emph{versus} div-integrable}

~

We start by proving that an element of $\sndiag$ which is transversally
Hamiltonian is necessarily div-integrable.
\begin{prop}
If $Y\in\sndiag$ is transversally Hamiltonian, then $Y$ is div-integrable. \end{prop}
\begin{proof}
Let us consider more generally a diagonal doubly-resonant saddle-node
$Y\in\sndiag$ such that $Y_{\mid\acc{x=0}}$ is a Hamiltonian vector
field with respect to $\mbox{d}y_{1}\wedge\mbox{d}y_{2}$ (this is
the case if $Y$ is transversally Hamiltonian): there exists a Hamiltonian
${\displaystyle H\left(\mathbf{y}\right)=\lambda y_{1}y_{2}+O\left(\norm{\mathbf{y}}^{3}\right)\in\germ{\mathbf{y}}}$,
such that 
\[
Y=x^{2}\pp x+\left(\left(-\ppp H{y_{2}}+xF_{1}\left(x,\mathbf{y}\right)\right)\pp{y_{1}}+\left(\ppp H{y_{1}}+xF_{2}\left(x,\mathbf{y}\right)\right)\pp{y_{2}}\right)\,\,\,,
\]
where $F_{1},F_{2}\in\germ{x,\mathbf{y}}$ are vanishing at the origin.
If we define ${\displaystyle J:=\left(\begin{array}{cc}
0 & -1\\
1 & 0
\end{array}\right)\in M_{2}\left(\ww C\right)}$ and $\nabla H:=\,^{t}\left(\mbox{D}H\right)$, then ${\displaystyle Y_{\mid\acc{x=0}}=J\nabla H}$.
According to the Morse lemma for holomorphic functions, there exists
a germ of an analytic change of coordinates $\varphi\in\diff[\ww C^{2},0]$
given by
\begin{eqnarray*}
\mathbf{y}=\left(y_{1},y_{2}\right) & \mapsto & \varphi\left(\mathbf{y}\right)=\left(y_{1}+O\left(\left\Vert \mathbf{y}\right\Vert ^{2}\right),y_{2}+O\left(\left\Vert \mathbf{y}\right\Vert ^{2}\right)\right)\,\,\,\,\,,
\end{eqnarray*}
such that ${\displaystyle {\displaystyle \widetilde{H}\left(\mathbf{y}\right):=H\left(\varphi^{-1}\left(\mathbf{y}\right)\right)=y_{1}y_{2}}}$.
Let us now recall a trivial result from linear algebra.
\begin{fact*}
Let ${\displaystyle J:=\left(\begin{array}{cc}
0 & -1\\
1 & 0
\end{array}\right)\in M_{2}\left(\ww C\right)}$, and $P\in M_{2}\left(\ww C\right)$. Then, ${\displaystyle PJ\,^{t}P=\mbox{det}\left(P\right)J}$.
\end{fact*}
As a consequence we have: 
\begin{cor}
\label{cor: hamiltonian system}Let $H\in\ww C\acc{\mathbf{y}}$ be
a germ of an analytic function at $0$, $Y_{0}:=J\nabla H$ the associated
Hamiltonian vector field in $\ww C^{2}$ (for the usual symplectic
form $dy_{1}\wedge dy_{2}$), and an analytic diffeomorphism near
the origin denoted by $\varphi$. Then: 
\begin{eqnarray*}
\varphi_{*}\left(Y_{0}\right) & := & \left(\tx D\varphi\circ\varphi^{-1}\right)\cdot\left(Y_{0}\circ\varphi^{-1}\right)=\det\left(\tx D\varphi\circ\varphi^{-1}\right)J\nabla\widetilde{H}\,\,\,\,,
\end{eqnarray*}
where $\widetilde{H}:=H\circ\varphi^{-1}$.
\end{cor}
As a conclusion we have proved that $Y$ is div-integrable.
\end{proof}

\subsubsection{General case}

~

Now we prove how to normalize the restriction to $\acc{x=0}$ of a
div-integrable element of $\sndiag$.
\begin{prop}
\label{prop: preparation sur hypersurface invariante}Let $Y\in\sndiag$
be div-integrable. Then, there exists $\psi\in\fdiffid$ of the form
\[
\psi:\left(x,\mathbf{y}\right)\mapsto\left(x,y_{1}+O\left(\norm{\mathbf{y}}^{2}\right),y_{2}+O\left(\norm{\mathbf{y}}^{2}\right)\right)
\]
such that 
\[
\psi_{*}\left(Y\right)=x^{2}\pp x+\left(-\left(\lambda+d\left(v\right)\right)y_{1}+xT_{1}\left(x,\mathbf{y}\right)\right)\pp{y_{1}}+\left(\left(\lambda+d\left(v\right)\right)y_{2}+xT_{2}\left(x,\mathbf{y}\right)\right)\pp{y_{2}}\,\,\,,
\]
with $v:=y_{1}y_{2}$, $d\left(v\right)\in v\germ v$ and $T_{1},T_{2}\in\germ{x,\mathbf{y}}$
vanishing at the origin.\end{prop}
\begin{proof}
By assumption, and according to a theorem due to Brjuno (\emph{cf
}\cite{Martinet}), up to a first transformation analytic at the origin
in $\ww C^{2}$, we can suppose that 
\[
Y_{\mid\acc{x=0}}=\left(\lambda+h\left(\mathbf{y}\right)\right)\left(-y_{1}\pp{y_{1}}+y_{2}\pp{y_{2}}\right)\,\,\,.
\]
Then, it remains to apply the following lemma to $Y_{\mid\acc{x=0}}$.\end{proof}
\begin{lem}
\label{lem: preparation sur hypersurface invariante}Let $Y_{0}$
be a germ of analytic vector field in $\left(\ww C^{2},0\right)$
of the form
\begin{eqnarray*}
Y_{0} & = & \left(\lambda+h\left(\mathbf{y}\right)\right)\left(-y_{1}\pp{y_{1}}+y_{2}\pp{y_{2}}\right)\,\,\,,
\end{eqnarray*}
with $h\in\germ{\mathbf{y}}$ vanishing at the origin. Then there
exists ${\displaystyle \phi\in\diff[\ww C^{2},0,\tx{Id}]}$ such that
\begin{eqnarray*}
\phi_{*}\left(Y_{0}\right) & = & \left(\lambda+d\left(v\right)\right)\left(-y_{1}\pp{y_{1}}+y_{2}\pp{y_{2}}\right)\,\,\,\,,
\end{eqnarray*}
with $v:=y_{1}y_{2}$ and $d\in v\ww C\acc v$.\end{lem}
\begin{rem}
In other words, we have removed every non-resonant term in $h\left(\mathbf{y}\right)$.
In fact, we re-obtain here a particular case (with one vector field
in dimension 2) of the principal result in \cite{StoloChampsCommutants}
(which is itself inspired of Vey's works).\end{rem}
\begin{proof}
We claim that $\phi$ can be chosen of the form 
\[
{\displaystyle \phi\left(\mathbf{y}\right)=\left(y_{1}e^{-\gamma\left(\mathbf{y}\right)},y_{2}e^{\gamma\left(\mathbf{y}\right)}\right)}\,\,,
\]
for a conveniently chosen $\gamma\in\germ{\mathbf{y}}$. Indeed, let
us study how such a diffeomorphism acts on $Y_{0}$. Let us write
$U:=\left(\lambda+h\left(v\right)\right)$ and $L:=\left(-y_{1}\pp{y_{1}}+y_{2}\pp{y_{2}}\right)$,
such that $Y_{0}=UL$. An easy computation shows:
\begin{eqnarray*}
\phi_{*}\left(Y_{0}\right) & = & \phi_{*}\left(U.L\right)\\
 & = & \left(\cro{U\cdot\left(1-{\cal L_{L}\left(\gamma\right)}\right)}\circ\phi^{-1}\right)L\,\,,
\end{eqnarray*}
where $\cal L_{L}$ is the Lie derivative of associate to $L$. We
want to to find $\gamma$ such that the unit 
\[
D:=\cro{U\left(1-{\cal L_{L}\left(\gamma\right)}\right)}\circ\phi^{-1}
\]
 is free from \emph{non-resonant} terms, \emph{i.e.} is of the form
\[
D=\lambda+d\left(y_{1}y_{2}\right)=\lambda+\sum_{k\geq1}d{}_{k}\left(y_{1}y_{2}\right)^{k}\,\,.
\]
Notice that if a unit ${\displaystyle W=\sum_{k\geq0}W_{k}\left(y_{1}y_{2}\right)^{k}\germ{\mathbf{y}}}^{\times}$
is free from non-resonant terms, then: 
\[
\begin{cases}
W\circ\phi^{-1}=W\\
\cal L_{L}\left(W\right)=0 & .
\end{cases}
\]
Thus, let us split both $U$ and $\gamma$ in a ``resonant'' and
a ``non-resonant'' part:
\[
\begin{cases}
U=U_{\tx{res}}+U_{\tx{n-res}}\\
\gamma=\gamma_{\tx{res}}+\gamma_{\tx{n-res}}
\end{cases}
\]
where 
\[
\begin{cases}
U_{\tx{n-res}}=\underset{\substack{k_{1}\neq k_{2}}
}{\sum}U_{k_{1},k_{2}}y_{1}^{k_{1}}y_{2}^{k_{2}}\\
U_{\tx{res}}=\underset{\substack{k}
}{\sum}U_{k,k}\left(y_{1}y_{2}\right)^{k}\\
\gamma_{\tx{n-res}}=\underset{\substack{k_{1}\neq k_{2}}
}{\sum}\gamma_{k_{1},k_{2}}y_{1}^{k_{1}}y_{2}^{k_{2}}\\
\gamma_{\tx{res}}=\underset{k}{\sum}\gamma_{k,k}\left(y_{1}y_{2}\right)^{k} & .
\end{cases}
\]
Then the non-resonant terms of $U\left(1-{\cal L_{L}\left(\gamma\right)}\right)$
are 
\[
\left(U_{\tx{n-res}}-\left(U_{\tx{n-res}}+U_{\tx{res}}\right){\cal L_{L}\left(\gamma_{\tx{n-res}}\right)}\right)\circ\phi^{-1}\,\,.
\]
Hence, the partial differential equation we want to solve is: 
\[
\cal L_{L}\left(\gamma\right)=\frac{U_{\tx{n-res}}}{U_{\tx{res}}+U_{\tx{n-res}}}\,\,.
\]
One sees immediately that this equation admit an analytic solution
(and even infinitely many solutions) $\gamma\in\germ{\mathbf{y}}$,
since the unit $U\in\germ{\mathbf{y}}$ is analytic. 
\end{proof}

\subsection{1-summable simplification of the ``dependent'' affine part}

~

We are concerned by studying vector fields of the form 
\begin{equation}
Y=x^{2}\pp x+\left(-\lambda y_{1}+f_{1}\left(x,\mathbf{y}\right)\right)\pp{y_{1}}+\left(\lambda y_{2}+f_{2}\left(x,\mathbf{y}\right)\right)\pp{y_{2}}\,\,\,\,,\label{eq: Y asympt hamilt}
\end{equation}
with
\[
\begin{cases}
{\displaystyle f_{1}\left(x,\mathbf{y}\right)=-d\left(y_{1}y_{2}\right)y_{1}+xT_{1}\left(x,\mathbf{y}\right)}\\
{\displaystyle f_{2}\left(x,\mathbf{y}\right)=d\left(y_{1}y_{2}\right)y_{2}+xT_{2}\left(x,\mathbf{y}\right)} & ,
\end{cases}
\]
where ${\displaystyle d\left(v\right)\in v\germ v}$ and ${\displaystyle T_{1},T_{2}\in\germ{x,\mathbf{y}}}$
are of order at least one.
\begin{prop}
\label{prop: diag prep}Let $Y\in\sndiag$ be a doubly-resonant saddle-node
of the form 
\[
Y=x^{2}\pp x+\left(-\lambda y_{1}+f_{1}\left(x,\mathbf{y}\right)\right)\pp{y_{1}}+\left(\lambda y_{2}+f_{2}\left(x,\mathbf{y}\right)\right)\pp{y_{2}}\,\,\,\,,
\]
where $f_{1},f_{2}\in\germ{x,\mathbf{y}}$ are such that ${\displaystyle f_{1}\left(x,\mathbf{y}\right),f_{2}\left(x,\mathbf{y}\right)=\mbox{\ensuremath{\tx O}}\left(\norm{\left(x,\mathbf{y}\right)}^{2}\right)}$.
Then there exist formal power series $\hat{y}_{1}$, $\hat{y}_{2}$,
$\hat{\alpha}_{1}$, $\hat{\alpha}_{2}$, $\hat{\beta}_{1}$, $\hat{\beta}_{2}\in x\form x$
which are 1-summable in every direction $\theta\neq\arg\left(\pm\lambda\right)$,
such that the formal fibered diffeomorphism 
\[
\hat{\Phi}:\left(x,y_{1},y_{2}\right)\mapsto\left(x,\hat{y}_{1}\left(x\right)+\left(1+\hat{\alpha}_{1}\left(x\right)\right)y_{1}+\hat{\beta}_{1}\left(x\right)y_{2},\hat{y}_{2}\left(x\right)+\hat{\alpha}_{2}\left(x\right)y_{1}+\left(1+\hat{\beta}_{1}\left(x\right)\right)y_{2}\right)\,\,\,,
\]
(which is tangent to the identity and 1- summable in every direction
$\theta\neq\arg\left(\pm\lambda\right)$) conjugates $Y$ to 
\[
\hat{\Phi}_{*}\left(Y\right)=x^{2}\pp x+\left(\left(-\lambda+a_{1}x\right)y_{1}+\hat{F}_{1}\left(x,\mathbf{y}\right)\right)\pp{y_{1}}+\left(\left(\lambda+a_{2}x\right)y_{2}+\hat{F}_{2}\left(x,\mathbf{y}\right)\right)\pp{y_{2}}\,\,\,,
\]
where $a_{1},a_{2}\in\ww C$ and $\hat{F}_{1},\hat{F}_{2}\in\form{x,\mathbf{y}}$
are of order at least 2 with respect to $\mathbf{y}$, and 1-summable
in every direction $\theta\neq\arg\left(\pm\lambda\right)$.\end{prop}
\begin{rem}
Notice that $\hat{\Phi}_{\mid\acc{x=0}}=\mbox{Id}$, so that $\hat{F}_{i}\left(0,\mathbf{y}\right)=f_{i}\left(0,\mathbf{y}\right)$
for $i=1,2$. Moreover, the residue of $\hat{\Phi}_{*}\left(Y\right)$
is $a_{1}+a_{2}$.
\end{rem}
The proof of Proposition \ref{prop: diag prep} is postponed to subsection
\ref{sub:Proof-of-Proposition}.

\subsubsection{Technical lemmas on irregular differential systems}

~
\begin{lem}
There exists a pair of formal power series $\left(\hat{y}_{1}\left(x\right),\hat{y}_{2}\left(x\right)\right)\in\left(x\form x\right)^{2}$
which are 1-summable in every direction $\theta\neq\arg\left(\pm\lambda\right)$,
such that the formal diffeomorphism given by 
\[
{\displaystyle \hat{\Phi}_{1}\left(x,y_{1},y_{2}\right)=\left(x,y_{1}-\hat{y}_{1}\left(x\right),y_{2}-\hat{y}_{2}\left(x\right)\right)},
\]
(which is 1-summable in every direction $\theta\neq\arg\left(\pm\lambda\right)$)
conjugates $Y$ in (\ref{eq: Y asympt hamilt}) to
\begin{equation}
\hat{Y}_{1}\left(x,\mathbf{y}\right)=x^{2}\pp x+\left(-\lambda y_{1}+\hat{g}_{1}\left(x,\mathbf{y}\right)\right)\pp{y_{1}}+\left(\lambda y_{2}+\hat{g}_{2}\left(x,\mathbf{y}\right)\right)\pp{y_{2}}\qquad,\label{eq: redressement separatrice}
\end{equation}
where $\hat{g}_{1},\hat{g}_{2}$ are formal power series of order
at least $2$ such that $\hat{g}_{1}\left(x,\mathbf{0}\right)=\hat{g}_{2}\left(x,\mathbf{0}\right)=0$,
and are 1-summable in every direction $\theta\neq\arg\left(\pm\lambda\right)$.
\end{lem}
In other words, in the new coordinates, the curve given by $\left(y_{1},y_{2}\right)=\left(0,0\right)$
is invariant by the flow of the vector field, and contains the origin
in its closure: it is usually called the (formal, 1-summable) \emph{center
manifold.}
\begin{proof}
This is an immediate consequence of an important theorem by Ramis
and Sibuya on the summability of formal solutions to irregular differential
systems \cite{ramis1989hukuhara}. This theorem proves the existence
and the 1-summability in every direction $\theta\neq\arg\left(\pm\lambda\right)$,
of $\hat{y}_{1}$ and $\hat{y}_{2}$: $\left(\hat{y}_{1}\left(x\right),\hat{y}_{2}\left(x\right)\right)$
is defined as the unique formal solution to 
\[
\begin{cases}
{\displaystyle x^{2}\ddd{y_{1}}x=-\lambda y_{1}\left(x\right)+f_{1}\left(x,y_{1}\left(x\right),y_{2}\left(x\right)\right)}\\
{\displaystyle x^{2}\ddd{y_{2}}x=\lambda y_{2}\left(x\right)+f_{2}\left(x,y_{1}\left(x\right),y_{2}\left(x\right)\right)}
\end{cases}\,\,,
\]
such that $\left(\hat{y}_{1}\left(0\right),\hat{y}_{2}\left(0\right)\right)=\left(0,0\right)$.
The 1-summability of $\hat{g}_{1}$ and $\hat{g}_{2}$ comes from
Proposition \ref{prop: compositon summable}.
\end{proof}
The next step is aimed at changing to linear terms with respect to
$\mathbf{y}$ in ``diagonal'' form.
\begin{lem}
There exists a pair of formal power series $\left(\hat{p}_{1},\hat{p}_{2}\right)\in\left(\form x\right)^{2}$
which are 1-summable in every direction $\theta\neq\arg\left(\pm\lambda\right)$,
such that the formal fibered diffeomorphism given by 
\[
\hat{\Phi}_{2}\left(x,y_{1},y_{2}\right)=\left(x,y_{1}+x\hat{p}_{2}\left(x\right)y_{2},y_{2}+x\hat{p}_{1}\left(x\right)y_{1}\right)\,\,,
\]
(which is tangent to the identity and 1-summable in every direction
$\theta\neq\arg\left(\pm\lambda\right)$) conjugates $\hat{Y}_{1}$
in (\ref{eq: redressement separatrice}), to
\begin{eqnarray}
\hat{Y}_{2}\left(x,\mathbf{y}\right) & = & x^{2}\pp x+\left(\left(-\lambda+x\hat{a}_{1}\left(x\right)\right)y_{1}+\hat{H}_{1}\left(x,\mathbf{y}\right)\right)\pp{y_{1}}\nonumber \\
 &  & +\left(\left(\lambda+x\hat{a}_{2}\left(x\right)\right)y_{2}+\hat{H}_{2}\left(x,\mathbf{y}\right)\right)\pp{y_{2}}\qquad,\label{eq: Y semi-diag}
\end{eqnarray}
where $\hat{a}_{1},\hat{a}_{2},\hat{H}_{1},\hat{H}_{2}$ are formal
power series which are 1-summable in every direction $\theta\neq\arg\left(\pm\lambda\right)$
and $\hat{H}_{1},\hat{H}_{2}$ are of order at least $2$ with respect
to $\mathbf{y}$. \end{lem}
\begin{proof}
Let us write
\[
\begin{cases}
{\displaystyle \hat{g}_{1}\left(x,\mathbf{y}\right)=x\hat{b}_{1}\left(x\right)y_{1}+x\hat{c}_{1}\left(x\right)y_{2}+\hat{G}_{1}\left(x,\mathbf{y}\right)}\\
{\displaystyle \hat{g}_{2}\left(x,\mathbf{y}\right)=x\hat{c}_{2}\left(x\right)y_{1}+x\hat{b}_{2}\left(x\right)y_{2}+\hat{G}_{2}\left(x,\mathbf{y}\right)}
\end{cases}\qquad,
\]
where $\hat{b}_{1},\hat{b}_{2},\hat{c}_{1},\hat{c}_{2},\hat{G}_{1},\hat{G}_{2}$
are formal power series 1-summable in the direction $\theta\neq\arg\left(\pm\lambda\right)$,
such that $\hat{G}_{1}$ and $\hat{G}_{2}$ are of order at least
$2$ with respect to $\mathbf{y}$. Let us consider the following
irregular differential system naturally associated to $\hat{Y}_{1}$: 

\begin{equation}
x^{2}\ddd{\mathbf{z}}x\left(x\right)=\mathbf{\hat{B}}\left(x\right)\mathbf{z}\left(x\right)+\mathbf{\hat{G}}\left(x,\mathbf{z}\left(x\right)\right)\qquad,\label{eq: systeme diff irreg}
\end{equation}
where 
\[
\mathbf{\hat{B}}\left(x\right)=\left(\begin{array}{cc}
-\lambda+x\hat{b}_{1}\left(x\right) & x\hat{c}_{1}\left(x\right)\\
x\hat{c}_{2}\left(x\right) & \lambda+x\hat{b}_{2}\left(x\right)
\end{array}\right),\quad\mathbf{\hat{G}}\left(x,\mathbf{z}\left(x\right)\right)=\left(\begin{array}{c}
\hat{G}_{1}\left(x,\mathbf{z}\left(x\right)\right)\\
\hat{G}_{2}\left(x,\mathbf{z}\left(x\right)\right)
\end{array}\right)\qquad.
\]
We are looking for 
\[
{\displaystyle \mathbf{\hat{P}}\left(x\right)=\left(\begin{array}{cc}
1 & x\hat{p}_{2}\left(x\right)\\
x\hat{p}_{1}\left(x\right) & 1
\end{array}\right)\in\mbox{GL}_{2}\left(\form x\right)}\,\,,
\]
where $\hat{p}_{1},\hat{p}_{2}$ are 1-summable formal power series
in $x$ every direction $\theta\neq\arg\left(\pm\lambda\right)$,
such that the linear transformation given by $\mathbf{z}\left(x\right)=\hat{\mathbf{P}}\left(x\right)\mathbf{y}\left(x\right)$
changes equation (\ref{eq: systeme diff irreg}) to 
\[
x^{2}\ddd{\mathbf{y}}x\left(x\right)=\mathbf{\hat{A}}\left(x\right)\mathbf{y}\left(x\right)+\mathbf{\hat{H}}\left(x,\mathbf{y}\left(x\right)\right)\qquad,
\]
with 
\[
\hat{\mathbf{A}}\left(x\right)=\left(\begin{array}{cc}
-\lambda+x\hat{a}_{1}\left(x\right) & 0\\
0 & \lambda+x\hat{a}_{2}\left(x\right)
\end{array}\right),\quad\hat{\mathbf{H}}\left(x,\mathbf{y}\left(x\right)\right)=\left(\begin{array}{c}
\hat{H}_{1}\left(x,\mathbf{y}\left(x\right)\right)\\
\hat{H}_{2}\left(x,\mathbf{y}\left(x\right)\right)
\end{array}\right)\quad,
\]
where $\hat{a}_{1},\hat{a}_{2},\hat{H}_{1},\hat{H}_{2}$ are 1-summable
formal power series in $x$ every direction $\theta\neq\arg\left(\pm\lambda\right)$. 

We have

\[
x^{2}\ddd{\mathbf{y}}x\left(x\right)=\hat{\mathbf{P}}\left(x\right)^{-1}\left(\mathbf{\hat{B}}(x)\hat{\mathbf{P}}\left(x\right)-x^{2}\ddd{\hat{\mathbf{P}}}x\left(x\right)\right)\mathbf{y}\left(x\right)+\hat{\mathbf{P}}\left(x\right)^{-1}\mathbf{\hat{G}}\left(x,\hat{\mathbf{P}}\left(x\right)\mathbf{y}\left(x\right)\right)
\]
and we want to determine $\mathbf{\hat{A}}\left(x\right)$ and $\hat{\mathbf{P}}\left(x\right)$
as above so that 
\[
\mathbf{\hat{B}}(x)\hat{\mathbf{P}}\left(x\right)-x^{2}\ddd{\hat{\mathbf{P}}}x\left(x\right)=\mathbf{\hat{A}}\left(x\right)\qquad.
\]
This gives four equations: 
\begin{equation}
\begin{cases}
{\displaystyle \hat{a}_{1}\left(x\right)=\hat{b}_{1}\left(x\right)+x\hat{c}_{1}\left(x\right)\hat{p}_{1}\left(x\right)}\\
{\displaystyle \hat{a}_{2}\left(x\right)=\hat{b}_{2}\left(x\right)+x\hat{c}_{2}\left(x\right)\hat{p}_{2}\left(x\right)}\\
{\displaystyle x^{2}\ddd{\hat{p}_{1}}x\left(x\right)=\left(2\lambda+x\hat{b}_{2}\left(x\right)-x-x\hat{b}_{1}\left(x\right)\right)\hat{p}_{1}\left(x\right)+\hat{c}_{2}\left(x\right)-x^{2}\hat{c}_{1}\left(x\right)\hat{p}_{1}\left(x\right)^{2}}\\
{\displaystyle x^{2}\ddd{\hat{p}_{2}}x\left(x\right)=\left(-2\lambda+x\hat{b}_{1}\left(x\right)-x-x\hat{b}_{2}\left(x\right)\right)\hat{p}_{2}\left(x\right)+\hat{c}_{1}\left(x\right)-x^{2}\hat{c}_{2}\left(x\right)\hat{p}_{2}\left(x\right)^{2}}
\end{cases}\qquad.\label{eq: system 4 eq}
\end{equation}
Thanks to the theorem by Ramis and Sibuya on the summability of formal
solutions to irregular systems \cite{ramis1989hukuhara}, we have
the existence and the 1-summability in every direction $\theta\neq\arg\left(\pm\lambda\right)$,
of $\hat{p}_{1}$ and $\hat{p}_{2}$: $\left(\hat{p}_{1}\left(x\right),\hat{p}_{2}\left(x\right)\right)$
is defined as the unique formal solution to 
\[
\begin{cases}
{\displaystyle x^{2}\ddd{\hat{p}_{1}}x\left(x\right)=\left(2\lambda+x\hat{b}_{2}\left(x\right)-x-x\hat{b}_{1}\left(x\right)\right)\hat{p}_{1}\left(x\right)+\hat{c}_{2}\left(x\right)-x^{2}\hat{c}_{1}\left(x\right)\hat{p}_{1}\left(x\right)^{2}}\\
{\displaystyle x^{2}\ddd{\hat{p}_{2}}x\left(x\right)=\left(-2\lambda+x\hat{b}_{1}\left(x\right)-x-x\hat{b}_{2}\left(x\right)\right)\hat{p}_{2}\left(x\right)+\hat{c}_{1}\left(x\right)-x^{2}\hat{c}_{2}\left(x\right)\hat{p}_{2}\left(x\right)^{2}}
\end{cases}
\]
such that 
\[
{\displaystyle \left(\hat{p}_{1}\left(0\right),\hat{p}_{2}\left(0\right)\right)=\left(\frac{-\hat{c}_{2}\left(0\right)}{2\lambda},\frac{\hat{c}_{1}\left(0\right)}{2\lambda}\right)}\,\,.
\]
Notice that $\hat{a}_{1}$ and $\hat{a}_{2}$ are defined by the first
two equations in (\ref{eq: system 4 eq}). Finally, $\hat{\mathbf{H}}$
is defined by 
\[
{\displaystyle \hat{\mathbf{H}}\left(x,\mathbf{y}\right):=\hat{\mathbf{P}}\left(x\right)^{-1}\mathbf{\hat{G}}\left(x,\hat{\mathbf{P}}\left(x\right)\mathbf{y}\right)}\,\,,
\]
and it is 1-summable in every direction $\theta\neq\arg\left(\pm\lambda\right)$,
by Proposition \ref{prop: compositon summable}.
\end{proof}
The goal of the last following lemma is to transform $\hat{a}_{1}\left(x\right)$
and $\hat{a}_{2}\left(x\right)$ in (\ref{eq: Y semi-diag}) to constant
terms.
\begin{lem}
There exists a pair of formal power series $\left(\hat{q}_{1},\hat{q}_{2}\right)\in\left(\form x\right)^{2}$
with $\hat{q}_{1}\left(0\right)=\hat{q}_{2}\left(0\right)=1$, which
are 1-summable in every direction $\theta\neq\arg\left(\pm\lambda\right)$,
such that the formal fibered diffeomorphism of the form 
\[
\hat{\Phi}_{3}\left(x,y_{1},y_{2}\right)=\left(x,\hat{q}_{1}\left(x\right)y_{1},\hat{q}_{2}\left(x\right)y_{2}\right)\,\,,
\]
(which is tangent to the identity and 1-summable in every direction
$\theta\neq\arg\left(\pm\lambda\right)$) conjugates $\hat{Y}_{2}$
in (\ref{eq: Y semi-diag}), to
\begin{eqnarray*}
\hat{Y}_{3}\left(x,\mathbf{y}\right) & = & x^{2}\pp x+\left(\left(-\lambda+a_{1}x\right)y_{1}+\hat{F}_{1}\left(x,\mathbf{y}\right)\right)\pp{y_{1}}\\
 &  & +\left(\left(\lambda+a_{2}x\right)y_{2}+\hat{F}_{2}\left(x,\mathbf{y}\right)\right)\pp{y_{2}}\qquad,
\end{eqnarray*}
where $\hat{F}_{1},\hat{F}_{2}$ are formal power series of order
at least $2$ with respect to $\mathbf{y}$ which are 1-summable in
every direction $\theta\neq\arg\left(\pm\lambda\right)$ and $\left(a_{1},a_{2}\right)=\left(\hat{a}_{1}\left(0\right),\hat{a}_{2}\left(0\right)\right)$.\end{lem}
\begin{proof}
We can associate to $\hat{Y}_{2}$ the following irregular differential
system: 
\[
x^{2}\ddd{\mathbf{z}}x\left(x\right)=\mathbf{\hat{A}}\left(x\right)\mathbf{z}\left(x\right)+\mathbf{\hat{H}}\left(x,\mathbf{z}\left(x\right)\right)\qquad,
\]
and we are looking for a change of coordinates of the form ${\displaystyle \mathbf{z}\left(x\right)=\hat{\mathbf{Q}}\left(x\right)\mathbf{y}\left(x\right)}$,
where 
\[
{\displaystyle \mathbf{\hat{Q}}\left(x\right)=\left(\begin{array}{cc}
\hat{q}_{1}\left(x\right) & 0\\
0 & \hat{q}_{2}\left(x\right)
\end{array}\right)\in\mbox{GL}_{2}\left(\form x\right)}
\]
 with $\hat{q}_{1}\left(0\right)=\hat{q}_{2}\left(0\right)=1$, such
that the new system is 
\[
x^{2}\ddd{\mathbf{y}}x\left(x\right)=\mathbf{A}\left(x\right)\mathbf{y}\left(x\right)+\mathbf{\hat{F}}\left(x,\mathbf{y}\left(x\right)\right)\qquad,
\]
with 
\[
\mathbf{A}\left(x\right)=\left(\begin{array}{cc}
-\lambda+a_{1}x & 0\\
0 & \lambda+a_{2}x
\end{array}\right),\quad\hat{\mathbf{F}}\left(x,\mathbf{y}\left(x\right)\right)=\left(\begin{array}{c}
\hat{F}_{1}\left(x,\mathbf{y}\left(x\right)\right)\\
\hat{F}_{2}\left(x,\mathbf{y}\left(x\right)\right)
\end{array}\right)\quad,
\]
and $\left(a_{1},a_{2}\right)=\left(\hat{a}_{1}\left(0\right),\hat{a}_{2}\left(0\right)\right)$. 

We have
\begin{eqnarray*}
x^{2}\ddd{\mathbf{y}}x\left(x\right) & = & \underbrace{\hat{\mathbf{Q}}\left(x\right){}^{-1}\left(\hat{\mathbf{A}}\left(x\right)\hat{\mathbf{Q}}\left(x\right)-x^{2}\ddd{\hat{\mathbf{Q}}}x\left(x\right)\right)}\mathbf{y}(x)+\hat{\mathbf{Q}}(x)^{-1}\hat{\mathbf{H}}\left(x,\hat{\mathbf{Q}}(x)\mathbf{y}(x)\right)\\
 &  & \qquad\,\qquad\qquad\quad=\\
 &  & \qquad\quad\left(\begin{array}{cc}
-\lambda+a_{1}x & 0\\
0 & \lambda+a_{2}x
\end{array}\right)
\end{eqnarray*}
so that

\[
x^{2}\ddd{\hat{\mathbf{Q}}}x(x)=\hat{\mathbf{A}}(x)\hat{\mathbf{Q}}(x)-\hat{\mathbf{Q}}(x)\left(\begin{array}{cc}
-\lambda+a_{1}x & 0\\
0 & \lambda+a_{2}x
\end{array}\right)\qquad
\]
and we obtain:

\begin{eqnarray*}
 & \begin{cases}
{\displaystyle x^{2}\ddd{\hat{q}_{1}}x(x)=x\hat{q}_{1}(x)\left(\hat{a}_{1}\left(x\right)-a_{1}\right)}\\
{\displaystyle x^{2}\ddd{\hat{q}_{2}}x(x)=x\hat{q}_{2}(x)\left(\hat{a}_{2}\left(x\right)-a_{2}\right)}
\end{cases}\\
\Longleftrightarrow & \begin{cases}
{\displaystyle \ddd{\hat{q}_{1}}x(x)=\hat{q}_{1}(x)\left(\frac{\hat{a}_{1}\left(x\right)-a_{1}}{x}\right)}\\
{\displaystyle \ddd{\hat{q}_{2}}x(x)=\hat{q}_{2}(x)\left(\frac{\hat{a}_{2}\left(x\right)-a_{2}}{x}\right)}
\end{cases}\\
\Longleftrightarrow & \begin{cases}
{\displaystyle \hat{q}_{1}(x)=\mbox{exp}\left(\int_{0}^{x}\frac{\hat{a}_{1}\left(s\right)-a_{1}}{s}\mbox{d}s\right)}\\
{\displaystyle \hat{q}_{2}(x)=\mbox{exp}\left(\int_{0}^{x}\frac{\hat{a}_{2}\left(s\right)-a_{2}}{s}\mbox{d}s\right)}
\end{cases} & ,\,\mbox{if we set }\hat{q}_{1}\left(0\right)=\hat{q}_{2}\left(0\right)=1\,\,,
\end{eqnarray*}
and the expression ${\displaystyle \int_{0}^{x}\frac{\hat{a}_{j}\left(s\right)-a_{j}}{s}\mbox{d}s}$,
for $j=1,2$, means the only anti-derivative of ${\displaystyle \frac{\hat{a}_{j}\left(s\right)-a_{j}}{s}}$
without constant term. Since $\hat{a}_{1}$ and $\hat{a}_{2}$ are
1-summable in every direction $\theta\neq\arg\left(\pm\lambda\right)$,
the same goes for $\hat{q}_{1}$ and $\hat{q}_{2}$, and then for
$\hat{F}_{1}$ and $\hat{F}_{2}$ by Proposition \ref{prop: compositon summable}.
\end{proof}

\subsubsection{\label{sub:Proof-of-Proposition}Proof of Proposition \ref{prop: diag prep}.}

~

We are now able to prove Proposition \ref{prop: diag prep}.
\begin{proof}[Proof of Proposition \ref{prop: diag prep}\emph{.}]

We have to use successively Lemma \ref{lem: preparation sur hypersurface invariante}
(to $Y_{0}:=Y_{\mid\acc{x=0}}$), followed by Proposition \ref{prop: diag prep},
then Proposition \ref{prop: preparation} and finally Proposition
\ref{prop: ordre N avec eq homologique}, using at each step Corollary
\ref{cor: summability push-forward} to obtain the 1-summability.
\end{proof}

\subsection{1-summable straightening of two invariant hypersurfaces}

~

For any $\theta\in\ww R$, we recall that we denote by $F_{\theta}$
the 1-sum of a 1-summable series $\hat{F}$ in the direction $\theta$.

Let $\theta\in\ww R$ with $\theta\neq\arg\left(\pm\lambda\right)$
and consider a formal vector field $\hat{Y}$, 1-summable in the direction
$\theta$ of 1-sum $Y_{\theta}$, of the form 
\begin{equation}
\hat{Y}=x^{2}\pp x+\left(\lambda_{1}\left(x\right)y_{1}+\hat{F}_{1}\left(x,\mathbf{y}\right)\right)\pp{y_{1}}+\left(\lambda_{2}\left(x\right)y_{2}+\hat{F}_{2}\left(x,\mathbf{y}\right)\right)\pp{y_{2}}\qquad,\label{eq: Y semi prep}
\end{equation}
where:
\begin{itemize}
\item $\lambda_{1}\left(x\right)=-\lambda+a_{1}x$
\item $\lambda_{2}\left(x\right)=\lambda+a_{2}x$
\item $\lambda\neq0$
\item $a_{1},a_{2}\in\ww C$
\item for $j=1,2$, 
\[
{\displaystyle \hat{F}_{j}\left(x,\mathbf{y}\right)=\sum_{\substack{\mathbf{n}\in\ww N^{2}\\
\abs{\mathbf{n}}\geq2
}
}\hat{F}_{\mathbf{n}}^{\left(j\right)}\left(x\right)\mathbf{y^{n}}}\in\form{x,\mathbf{y}}
\]
 is 1-summable in the direction $\theta$ of 1-sum 
\[
{\displaystyle F_{j,\theta}\left(x,\mathbf{y}\right)=\sum_{\substack{\mathbf{n}\in\ww N^{2}\\
\abs{\mathbf{n}}\geq2
}
}F_{j,\mathbf{n},\theta}\left(x\right)\mathbf{y^{n}}}\,\,.
\]
In particular, there exists $A,B,\mu>0$ such that for all $\mathbf{n}\in\ww N^{2}$,
$\abs n\geq2$, for $j=1,2$: 
\[
\forall t\in\Delta_{\theta,\epsilon,\rho},\,\abs{\widetilde{\cal B}\left(\hat{F}_{j,\mathbf{n}}\right)\left(t\right)}\leq A.B^{\abs{\mathbf{n}}}\frac{\exp\left(\mu\abs t\right)}{1+\mu^{2}\abs t^{2}}\qquad,
\]
for some $\rho>0$ and $\epsilon>0$ such that $\left(\ww R.\lambda\right)\cap\cal A_{\theta,\epsilon}=\emptyset$
(see Definition \ref{def: 1_summability} and Remark \ref{rem: norme bis borel}
for the notations). Notice that $F_{j,\theta}$ is analytic and bounded
in some sectorial neighborhood $\cal S\in{\cal S}_{\theta,\pi}$ of
the origin. For technical reasons, we use in this subsection the alternative
definition of the Borel transform $\widetilde{\cal B}$, with its
associate norm $\norm{\cdot}_{\mu}^{\tx{bis}}$ (see Remarks \ref{rem: deifinition bis transformee de Borel}
and \ref{rem: norme bis borel} and Proposition \ref{prop: norme d'algebre}\end{itemize}
\begin{prop}
\label{prop: preparation}Under the assumptions above, there exists
a pair of formal power series $\left(\hat{\phi}_{1},\hat{\phi}_{2}\right)\in\left(\form{x,\mathbf{y}}\right)^{2}$
of order at least two with respect to $\mathbf{y}$ which are 1-summable
in every direction $\theta\neq\arg\left(\pm\lambda\right)$, such
that the formal fibered diffeomorphism 
\[
\hat{\Phi}\left(x,\mathbf{y}\right)=\left(x,y_{1}+\hat{\phi}_{1}\left(x,\mathbf{y}\right),y_{2}+\hat{\phi}_{2}\left(x,\mathbf{y}\right)\right)\qquad,
\]
(which is tangent to the identity and 1-summable in every direction
$\theta\neq\arg\left(\pm\lambda\right)$) conjugates $\hat{Y}$ in
(\ref{eq: Y semi prep}) to 
\[
\hat{Y}_{\tx{prep}}=x^{2}\pp x+\left(\left(-\lambda+a_{1}x\right)+y_{2}\hat{R}_{1}\left(x,\mathbf{y}\right)\right)y_{1}\pp{y_{1}}+\left(\left(\lambda+a_{2}x\right)+y_{1}\hat{R}_{2}\left(x,\mathbf{y}\right)\right)y_{2}\pp{y_{2}}\qquad,
\]
where ${\displaystyle \hat{R}_{1},\hat{R}_{2}\in\form{x,\mathbf{y}}}$
are 1-summable in every direction $\theta\neq\arg\left(\pm\lambda\right)$.\end{prop}
\begin{proof}
We follow and adapt the proof of analytic straightening of invariant
curves for resonant saddles in two dimensions in \cite{MatteiMoussu}.

We are looking for 
\[
\hat{\Psi}\left(x,\mathbf{y}\right)=\left(x,y_{1}+\hat{\psi}_{1}\left(x,\mathbf{y}\right),y_{2}+\hat{\psi}_{2}\left(x,\mathbf{y}\right)\right)\quad,
\]
with $\hat{\psi}_{1},\hat{\psi}_{2}$ of order at least 2, and $\hat{R}_{1},\hat{R}_{2}$
as above such that: 
\[
\hat{\Psi}_{*}\left(\hat{Y}_{\tx{prep}}\right)=\hat{Y}\qquad,
\]
\emph{i.e. 
\begin{eqnarray}
\mbox{D}\hat{\Psi}\cdot\hat{Y}_{\tx{prep}} & = & \hat{Y}\circ\hat{\Psi}\qquad.\label{eq: conjugaison forme prep}
\end{eqnarray}
}Then, we will set $\Phi:=\Psi^{-1}$. Let us write 
\begin{eqnarray*}
\hat{T}_{1} & := & y_{1}y_{2}\hat{R}_{1}=\sum_{\abs{\mathbf{n}}\geq2}\hat{T}_{1,\mathbf{n}}\left(x\right)\mathbf{y^{n}}\\
\hat{T}_{2} & := & y_{1}y_{2}\hat{R}_{2}=\sum_{\abs{\mathbf{n}}\geq2}\hat{T}_{2,\mathbf{n}}\left(x\right)\mathbf{y^{n}}\qquad\\
\hat{\psi}_{1} & = & \sum_{\abs{\mathbf{n}}\geq2}\hat{\psi}_{1,\mathbf{n}}\left(x\right)\mathbf{y^{n}}\qquad\\
\hat{\psi}_{2} & = & \sum_{\abs{\mathbf{n}}\geq2}\hat{\psi}_{2,\mathbf{n}}\left(x\right)\mathbf{y^{n}}\qquad,
\end{eqnarray*}
so that equation $\left(\mbox{\ref{eq: conjugaison forme prep}}\right)$
becomes:
\begin{eqnarray*}
 &  & x^{2}\ppp{\hat{\psi}_{1}}{x^{2}}+\left(1+\ppp{\hat{\psi}_{1}}{y_{1}}\right)\left(\lambda_{1}\left(x\right)y_{1}+\hat{T}_{1}\right)+\ppp{\hat{\psi}_{1}}{y_{2}}\left(\lambda_{2}\left(x\right)y_{2}+\hat{T}_{2}\right)\\
 &  & =\lambda_{1}\left(x\right)\left(y_{1}+\hat{\psi}_{1}\right)+\hat{F}_{1}\left(x,y_{1}+\hat{\psi}_{1},y_{2}+\hat{\psi}_{2}\right)
\end{eqnarray*}
and 
\begin{eqnarray*}
 &  & x^{2}\ppp{\hat{\psi}_{2}}{x^{2}}+\ppp{\hat{\psi}_{2}}{y_{1}}\left(\lambda_{1}\left(x\right)y_{1}+\hat{T}_{1}\right)+\left(1+\ppp{\hat{\psi}_{2}}{y_{2}}\right)\left(\lambda_{2}\left(x\right)y_{2}+\hat{T}_{2}\right)\\
 &  & =\lambda_{2}\left(x\right)\left(y_{2}+\hat{\psi}_{2}\right)+\hat{F}_{2}\left(x,y_{1}+\hat{\psi}_{1},y_{2}+\hat{\psi}_{2}\right)\,\,\,.
\end{eqnarray*}
These equations can be written as:
\begin{equation}
\begin{cases}
{\displaystyle \underset{\left|\mathbf{n}\right|\geq2}{\sum}\left(\delta_{1,\mathbf{n}}(x)\hat{\psi}_{1,\mathbf{n}}\left(x\right)+x^{2}\ddd{\hat{\psi}_{1,\mathbf{n}}}x(x)+\hat{T}_{1,\mathbf{n}}\left(x\right)\right)\mathbf{y}^{\mathbf{n}}}\\
{\displaystyle =\hat{F}_{1}\left(x,y_{1}+\hat{\psi}_{1}\left(x,\mathbf{y}\right),y_{2}+\hat{\psi}_{2}\left(x,\mathbf{y}\right)\right)-\hat{T}_{1}(x)\frac{\partial\hat{\psi}_{1}}{\partial y_{1}}(x,\mathbf{y})-\hat{T}_{2}(x)\frac{\partial\hat{\psi}_{1}}{\partial y_{2}}(x,\mathbf{y})}\\
{\displaystyle =:\underset{\left|\mathbf{n}\right|\geq2}{\sum}\zeta_{1,\mathbf{n}}(x)\mathbf{y}^{\mathbf{n}}}\\
{\displaystyle \underset{\left|\mathbf{n}\right|\geq2}{\sum}\left(\delta_{2,\mathbf{n}}(x)\hat{\psi}_{2,\mathbf{n}}\left(x\right)+x^{2}\ddd{\hat{\psi}_{2,\mathbf{n}}}x(x)+\hat{T}_{2,\mathbf{n}}\left(x\right)\right)\mathbf{y}^{\mathbf{n}}}\\
{\displaystyle =\hat{F}_{2}\left(x,y_{1}+\hat{\psi}_{1}\left(x,\mathbf{y}\right),y_{2}+\hat{\psi}_{2}\left(x,\mathbf{y}\right)\right)-\hat{T}_{1}(x)\frac{\partial\hat{\psi}_{2}}{\partial y_{1}}(x,\mathbf{y})-\hat{T}_{2}(x)\frac{\partial\hat{\psi}_{2}}{\partial y_{2}}(x,\mathbf{y})}\\
{\displaystyle =:\underset{\left|\mathbf{n}\right|\geq2}{\sum}\zeta_{2,\mathbf{n}}(x)\mathbf{y}^{\mathbf{n}}}
\end{cases}\label{eq: =0000E9quation conjugaison serie}
\end{equation}
where ${\displaystyle \delta_{j,\mathbf{n}}(x)=\lambda_{1}(x)n_{1}+\lambda_{2}(x)n_{2}-\lambda_{j}(x)}$,
$j=1,2$. We are looking for $\hat{T}_{1},\hat{T}_{2}$ such that
\[
\begin{cases}
\hat{T}_{1,\mathbf{n}}=0 & \mbox{, if }n_{1}=0\mbox{ or }n_{2}=0\\
\hat{T}_{2,\mathbf{n}}=0 & \mbox{, if }n_{1}=0\mbox{ or }n_{2}=0
\end{cases}\mbox{ }\qquad.
\]
Notice that $\zeta_{j,\mathbf{n}}$, for $j=1,2$ and $\abs{\mathbf{n}}\geq2$,
depends only on the $\hat{\psi}_{i,\mathbf{k}}\mbox{'s}$ and the
$\hat{F}_{i,\mathbf{k}}\mbox{'s}$, for $i=1,2$, $\abs{\mathbf{k}}<\mathbf{n}$.
We can then determine the coefficients $\hat{\psi}_{j,\mathbf{n}}$
and $\hat{T}_{j,\mathbf{n}}$, $j=1,2$, $\abs{\mathbf{n}}\geq2$,
by induction on $\abs{\mathbf{n}}$, setting 
\[
\begin{cases}
\hat{T}_{1,\mathbf{n}}=0 & \mbox{, if }n_{1}=0\mbox{ or }n_{2}=0\\
\hat{T}_{2,\mathbf{n}}=0 & \mbox{, if }n_{1}=0\mbox{ or }n_{2}=0\\
\hat{\psi}_{1,\mathbf{n}}=0 & ,\mbox{ if }n_{1}\geq1\mbox{ and }n_{2}\geq1\\
\hat{\psi}_{2,\mathbf{n}}=0 & ,\mbox{ if }n_{1}\geq1\mbox{ and }n_{2}\geq1
\end{cases}\qquad,
\]
and solving for each $\mathbf{n}=\left(n_{1},n_{2}\right)\in\ww N^{2}$
with $\abs{\mathbf{n}}\geq2$, the equations
\[
\begin{cases}
{\displaystyle \delta_{1,\mathbf{n}}(x)\hat{\psi}_{1,\mathbf{n}}\left(x\right)+x^{2}\ddd{\hat{\psi}_{1,\mathbf{n}}}x(x)=\zeta_{1,\mathbf{n}}\left(x\right)} & \mbox{, if }n_{1}=0\mbox{ or }n_{2}=0\\
{\displaystyle \delta_{2,\mathbf{n}}(x)\hat{\psi}_{2,\mathbf{n}}\left(x\right)+x^{2}\ddd{\hat{\psi}_{2,\mathbf{n}}}x(x)=\zeta_{2,\mathbf{n}}\left(x\right)} & \mbox{, if }n_{1}=0\mbox{ or }n_{2}=0
\end{cases}\,\,\,\,\,.
\]

\begin{lem}
There exists $\beta>4\pi,M>0$ such that for all $\mathbf{n}\in\ww N^{2}$
with $\abs{\mathbf{n}}\geq2$, and for $j=1,2$, $\norm{\zeta_{j,\mathbf{n}}}_{\beta}^{\tx{bis}}<+\infty$
and: 
\[
\norm{\hat{\psi}_{j,\mathbf{n}}}_{\beta}^{\tx{bis}}\leq M.\norm{\zeta_{j,\mathbf{n}}}_{\beta}^{\tx{bis}}\qquad,
\]
where the norm corresponds to the domain $\triangle_{\theta,\epsilon,\rho}$
(see Definition \ref{def: 1_summability}).\end{lem}
\begin{proof}
For $\mathbf{n}=\left(n_{1},n_{2}\right)\in\ww N^{2}$ with $n_{1}+n_{2}\geq2$
we want to solve: 
\[
\begin{cases}
{\displaystyle \delta_{1,\mathbf{n}}(x)=\lambda_{1}(x)\left(n_{1}-1\right)+\lambda_{2}(x)n_{2}=} & \begin{cases}
\lambda\left(n_{2}+1\right)+x\left(-a_{1}+a_{2}n_{2}\right) & \mbox{ , if }n_{1}=0\\
-\lambda\left(n_{1}-1\right)+a_{1}x\left(n_{1}-1\right) & \mbox{ , if }n_{2}=0
\end{cases}\\
{\displaystyle \delta_{2,\mathbf{n}}(x)=\lambda_{2}(x)\left(n_{2}-1\right)+\lambda_{1}(x)n_{1}=} & \begin{cases}
\lambda\left(n_{2}-1\right)+a_{2}x\left(n_{2}-1\right) & \mbox{ , if }n_{1}=0\\
-\lambda\left(n_{1}+1\right)+x\left(-a_{2}+a_{1}n_{1}\right) & \mbox{ , if }n_{2}=0\,\,.
\end{cases}
\end{cases}
\]
We will only deal with $\delta_{1,\mathbf{n}}(x)$ (the case of $\delta_{2,\mathbf{n}}(x)$
being similar). Notice that we are exactly in the situation of Proposition
\ref{prop: solution borel sommable precise}. In particular, using
notation in this definition, we respectively have:

\begin{eqnarray*}
 &  & \begin{cases}
{\displaystyle k=\lambda\left(n_{2}+1\right),\,\alpha=\frac{\left(-a_{1}+a_{2}n_{2}\right)}{\lambda\left(n_{2}+1\right)},}\\
{\displaystyle d_{k}=\min\acc{\abs{\lambda\left(n_{2}+1\right)}-\rho,\abs{\lambda\left(n_{2}+1\right)}\abs{\sin\left(\theta+\epsilon\right)},\abs{\lambda\left(n_{2}+1\right)}\abs{\sin\left(\theta-\epsilon\right)}}}
\end{cases}\\
 &  & (\mbox{when }n_{1}=0)
\end{eqnarray*}
and

\begin{eqnarray*}
 &  & \begin{cases}
{\displaystyle k=-\lambda\left(n_{1}+1\right),\,\alpha=\frac{\left(-a_{2}+a_{1}n_{1}\right)}{-\lambda\left(n_{1}+1\right)},}\\
{\displaystyle d_{k}=\min\acc{\abs{\lambda\left(n_{1}+1\right)}-\rho,\abs{\lambda\left(n_{1}+1\right)}\abs{\sin\left(\theta+\epsilon\right)},\abs{\lambda\left(n_{1}+1\right)}\abs{\sin\left(\theta-\epsilon\right)}}}
\end{cases}\\
 &  & (\mbox{when }n_{2}=0).
\end{eqnarray*}
We can chose the domain $\Delta_{\theta,\epsilon,\rho}$ corresponding
to the 1-summability of $\hat{F}_{1}$ and $\hat{F}_{2}$ with $0<\rho<\abs{\lambda}$,
so that $d_{k}>0$, since $\epsilon>0$ is such that $\left(\ww R.\lambda\right)\cap\cal A_{\theta,\epsilon}=\emptyset$.
Finally, we chose 
\[
\beta>\frac{C\left(\abs{a_{1}}+\abs{a_{2}}\right)}{\min\acc{\abs{\lambda}-\rho,\abs{\lambda\sin\left(\theta+\epsilon\right)},\abs{\lambda\sin\left(\theta-\epsilon\right)}}}>0\,\,,
\]
(with $C=\frac{2\exp\left(2\right)}{5}+5)$, so that $\norm{\hat{F}_{1}}_{\beta}^{\tx{bis}}<+\infty$.
This choice of $\beta$ implies $\beta d_{k}>C\abs{\alpha k}$ as
needed in Proposition \ref{prop: solution borel sommable precise},
in both considered situations, namely $n_{1}=0$ and $n_{2}=0$ respectively.
Since for $j=1,2$ and $\abs{\mathbf{n}}\geq2$, $\zeta_{j,\mathbf{n}}$
depends only on the $\hat{\psi}_{i,\mathbf{k}}$'s and the $\hat{F}_{i,\mathbf{k}}$'s,
for $i=1,2$, $\abs{\mathbf{k}}<\mathbf{n}$, we deduce by induction
that 
\[
\begin{cases}
{\displaystyle \norm{\zeta_{1,\mathbf{n}}}_{\beta}^{\tx{bis}}<+\infty} & \mbox{ , if }n_{1}=0\mbox{ or }n_{2}=0\\
{\displaystyle \norm{\zeta_{2,\mathbf{n}}}_{\beta}^{\tx{bis}}<+\infty} & \mbox{ , if }n_{1}=0\mbox{ or }n_{2}=0
\end{cases}
\]
and then, thanks to Proposition \ref{prop: solution borel sommable precise}:
\[
{\displaystyle \norm{\hat{\psi}_{j,\mathbf{n}}}_{\beta}^{\tx{bis}}\leq\left(\frac{\beta}{\beta\left(\abs{\lambda}-\rho\right)-C\left(\abs{a_{1}}+\abs{a_{2}}\right)}\right).\norm{\zeta_{j,\mathbf{n}}}_{\beta}^{\tx{bis}}}\,\,,\mbox{ for \ensuremath{j=1,2}.}
\]
The lemma is proved, with 
\[
M=\left(\frac{\beta}{\beta\min\acc{\abs{\lambda}-\rho,\abs{\lambda\sin\left(\theta+\epsilon\right)},\abs{\lambda\sin\left(\theta-\epsilon\right)}}-C\left(\abs{a_{1}}+\abs{a_{2}}\right)}\right)\,\,.
\]

\end{proof}
In order to finish the proof of Proposition \ref{prop: preparation},
we have to prove that for $j=1,2$, the series ${\displaystyle \overline{\hat{\psi}_{j}}:=\sum_{\mathbf{n}\in\ww N^{2}}\norm{\hat{\psi}_{j,\mathbf{n}}}_{\beta}^{\tx{bis}}\mathbf{y}^{\mathbf{n}}}$
is convergent in a poly-disc $\mathbf{D\left(0,r\right)}$, with ${\displaystyle \mathbf{r}=\left(r_{1},r_{2}\right)\in\left(\ww R_{>0}\right)^{2}}$
(then, Corollary \ref{cor: faible et forte sommabilite} gives 1-summability).
We will prove this by using a method of dominant series. Let us introduce
some useful notations. If $\left(\mathfrak{B},\norm{\cdot}\right)$
is a Banach algebra, for any formal power series ${\displaystyle f\left(\mathbf{y}\right)=\sum_{\mathbf{n}}f_{\mathbf{n}}\mathbf{y^{n}}}\in\mathfrak{B}\left\llbracket \mathbf{y}\right\rrbracket $,
we define ${\displaystyle \overline{f}:=\sum_{\mathbf{n}}\norm{f_{\mathbf{n}}}\mathbf{y^{n}}},$
and $\overline{\overline{f}}\left(y\right):=\overline{f}\left(y,y\right)$.
If ${\displaystyle g=\sum_{\mathbf{n}}g_{\mathbf{n}}\mathbf{y^{n}}}\in\mathfrak{B}\left\llbracket \mathbf{y}\right\rrbracket $
is another formal power series, we write $\overline{f}\prec\overline{g}$
if for all $\mathbf{n}\in\ww N^{2}$, we have $\norm{f_{\mathbf{n}}}\leq\norm{g_{\mathbf{n}}}$.
We remind the following classical result (the proof is performed in
\cite{rebelo_reis} when $\left(\mathfrak{B},\norm{\cdot}\right)=\left(\ww C,\abs{\cdot}\right)$
, but the same proof works for any Banach algebra).
\begin{lem}
\cite[Theorem 2.2 p.48]{rebelo_reis} For $j=1,2$, let ${\displaystyle f_{j}=\sum_{\abs{\mathbf{n}}\geq2}f_{j,\mathbf{n}}\mathbf{y^{n}}}\in\mathfrak{B}\left\llbracket \mathbf{y}\right\rrbracket $
be two formal power series with coefficients in a Banach algebra $\left(\mathfrak{B},\norm{\cdot}\right)$,
and of order at least two. Consider also two other series ${\displaystyle g_{j}=\sum_{\abs{\mathbf{n}}\geq2}g_{j,\mathbf{n}}\mathbf{y^{n}}}\in\mathfrak{B}\acc{\ww{\mathbf{y}}}$,
$j=1,2$, of order at least two, which have a non-zero radius of convergence
at the origin. Assume that there exists $\sigma>0$ such that for
$j=1,2$: 
\[
\sigma\overline{f_{j}}\prec\overline{g_{j}}\left(y_{1}+\overline{f_{1}},y_{2}+\overline{f_{2}}\right)\qquad.
\]
 Then, $f_{1}$ and $f_{2}$ have a non-zero radius of convergence.
\end{lem}
Taking $\beta>4\pi$, according to Proposition \ref{prop: norme d'algebre},
for all $\hat{f},\hat{g}\in\mathfrak{B}_{\beta}^{\tx{bis}}$, we have:
\[
\norm{\hat{f}\hat{g}}_{\beta}^{\tx{bis}}\leq\norm{\hat{f}}_{\beta}^{\tx{bis}}\norm{\hat{g}}_{\beta}^{\tx{bis}}\,\,.
\]
This implies that $\left(\mathfrak{B}_{\beta}^{\tx{bis}},\norm{\cdot}_{\beta}^{\tx{bis}}\right)$
is a Banach algebra as needed in the above lemma. It remains to prove
that there exists $\sigma>0$ such that for $j=1,2$:
\[
\sigma\overline{\hat{\psi}_{j}}\prec\overline{\hat{F}_{j}}\left(y_{1}+\overline{\hat{\psi}_{1}},y_{2}+\overline{\hat{\psi}_{2}}\right)\qquad.
\]
Remember that there exists $M>0$ such that for $j=1,2$:
\[
\norm{\hat{\psi}_{j,\mathbf{n}}}_{\beta}^{\tx{bis}}\leq M.\norm{\zeta_{j,\mathbf{n}}}_{\beta}^{\tx{bis}}\qquad
\]
where
\[
\begin{cases}
{\displaystyle \zeta_{1}:=\underset{\left|\mathbf{n}\right|\geq2}{\sum}\zeta_{1,\mathbf{n}}(x)\mathbf{y}^{\mathbf{n}}}\\
{\displaystyle =\hat{F}_{1}\left(x,y_{1}+\hat{\psi}_{1}\left(x,\mathbf{y}\right),y_{2}+\hat{\psi}_{2}\left(x,\mathbf{y}\right)\right)-\hat{T}_{1}(x)\frac{\partial\hat{\psi}_{1}}{\partial y_{1}}(x,\mathbf{y})-\hat{T}_{2}(x)\frac{\partial\hat{\psi}_{1}}{\partial y_{2}}(x,\mathbf{y})}\\
{\displaystyle \zeta_{2}:=\underset{\left|\mathbf{n}\right|\geq2}{\sum}\zeta_{2,\mathbf{n}}(x)\mathbf{y}^{\mathbf{n}}}\\
{\displaystyle =\hat{F}_{2}\left(x,y_{1}+\hat{\psi}_{1}\left(x,\mathbf{y}\right),y_{2}+\hat{\psi}_{2}\left(x,\mathbf{y}\right)\right)-\hat{T}_{1}(x)\frac{\partial\hat{\psi}_{2}}{\partial y_{1}}(x,\mathbf{y})-\hat{T}_{2}(x)\frac{\partial\hat{\psi}_{2}}{\partial y_{2}}(x,\mathbf{y})} & \,\,\,.
\end{cases}
\]
If we set $\sigma:=\frac{1}{M}$, then we have
\[
\begin{cases}
{\displaystyle \sigma\overline{\hat{\psi}_{1}}\prec\overline{\zeta_{1}}\prec\overline{\hat{F}}_{1}\left(x,y_{1}+\overline{\hat{\psi}_{1}}\left(x,\mathbf{y}\right),y_{2}+\overline{\hat{\psi}_{2}}\left(x,\mathbf{y}\right)\right)+\overline{\hat{T}_{1}}(x)\frac{\partial\overline{\hat{\psi}_{1}}}{\partial y_{1}}(x,\mathbf{y})+\overline{\hat{T}_{2}}(x)\frac{\partial\overline{\hat{\psi}_{1}}}{\partial y_{2}}(x,\mathbf{y})}\\
{\displaystyle \sigma\overline{\hat{\psi}_{2}}\prec\overline{\zeta_{2}}\prec\overline{\hat{F}}_{2}\left(x,y_{1}+\overline{\hat{\psi}_{1}}\left(x,\mathbf{y}\right),y_{2}+\overline{\hat{\psi}_{2}}\left(x,\mathbf{y}\right)\right)+\overline{\hat{T}_{1}}(x)\frac{\partial\overline{\hat{\psi}_{2}}}{\partial y_{1}}(x,\mathbf{y})+\overline{\hat{T}_{2}}(x)\frac{\partial\overline{\hat{\psi}_{2}}}{\partial y_{2}}(x,\mathbf{y})}
\end{cases}\qquad.
\]
Moreover, we recall that 
\[
\begin{cases}
\hat{T}_{1,\mathbf{n}}=0 & \mbox{, if }n_{1}=0\mbox{ or }n_{2}=0\\
\hat{T}_{2,\mathbf{n}}=0 & \mbox{, if }n_{1}=0\mbox{ or }n_{2}=0\\
\hat{\psi}_{1,\mathbf{n}}=0 & ,\mbox{ if }n_{1}\geq1\mbox{ and }n_{2}\geq1\\
\hat{\psi}_{2,\mathbf{n}}=0 & ,\mbox{ if }n_{1}\geq1\mbox{ and }n_{2}\geq1
\end{cases}\qquad,
\]
so that we have in fact more precise dominant relations: 
\[
\begin{cases}
{\displaystyle \sigma\overline{\hat{\psi}_{1}}\prec\overline{\zeta_{1}}\prec\overline{\hat{F}}_{1}\left(x,y_{1}+\overline{\hat{\psi}_{1}}\left(x,\mathbf{y}\right),y_{2}+\overline{\hat{\psi}_{2}}\left(x,\mathbf{y}\right)\right)}\\
{\displaystyle \sigma\overline{\hat{\psi}_{2}}\prec\overline{\zeta_{2}}\prec\overline{\hat{F}}_{2}\left(x,y_{1}+\overline{\hat{\psi}_{1}}\left(x,\mathbf{y}\right),y_{2}+\overline{\hat{\psi}_{2}}\left(x,\mathbf{y}\right)\right)}
\end{cases}\qquad.
\]
It remains the apply the lemma above to conclude.\end{proof}
\begin{rem}
\label{rem: redressement hypersurfaces cas ham}In the previous proposition,
assume that for $j=1,2$, 
\[
{\displaystyle \hat{F}_{j}\left(x,\mathbf{y}\right)=\sum_{\mathbf{n}\in\ww N^{2},\,\abs{\mathbf{n}}\geq2}\hat{F}_{\mathbf{n}}^{\left(j\right)}\left(x\right)\mathbf{y^{n}}}
\]
 in the expression of $\hat{Y}$ satisfies
\[
\begin{cases}
\hat{F}_{\mathbf{n}}^{\left(1\right)}\left(0\right)=0 & ,\,\forall\mathbf{n}=\left(n_{1},n_{2}\right)\mid n_{1}+n_{2}\geq2\mbox{ and }\big(n_{1}=0\mbox{ or }n_{2}=0\big)\\
\hat{F}_{\mathbf{n}}^{\left(2\right)}\left(0\right)=0 & ,\,\forall\mathbf{n}=\left(n_{1},n_{2}\right)\mid n_{1}+n_{2}\geq2\mbox{ and }\big(n_{1}=0\mbox{ or }n_{2}=0
\end{cases}\,\,\,.
\]
Then, the diffeomorphism $\hat{\Phi}$ in the proposition can be chosen
to be the identity on $\acc{x=0}$, so that
\[
\begin{cases}
y_{1}y_{2}\hat{R}_{1}\left(x,\mathbf{y}\right) & =\hat{F}_{1}\left(0,\mathbf{y}\right)+x\hat{S}_{1}\left(x,\mathbf{y}\right)\\
y_{1}y_{2}\hat{R}_{2}\left(x,\mathbf{y}\right) & =\hat{F}_{2}\left(0,\mathbf{y}\right)+x\hat{S}_{2}\left(x,\mathbf{y}\right)\qquad,
\end{cases}
\]
 where $\hat{S}_{1},\hat{S}_{2}$ are 1-summable in the direction
$\theta\neq\arg\left(\pm\lambda\right)$ and ${\displaystyle \hat{F}_{1}\left(0,\mathbf{y}\right),\hat{F}_{2}\left(0,\mathbf{y}\right)\in\germ{\mathbf{y}}}$
are convergent in neighborhood of the origin in $\ww C^{2}$. Indeed,
we easily see by induction on $\abs{\mathbf{n}}=n_{1}+n_{2}\geq2$
that $\hat{\psi}_{1}$ and $\hat{\psi}_{2}$ can be chosen ``divisible''
by $x$, and that $\zeta_{1},\zeta_{2}$ are such that $\zeta_{j,\mathbf{n}}\left(x\right)$
is also ``divisible'' by $x$ if $n_{1}=0$ or $n_{2}=0$.
\end{rem}

\subsection{\label{sub:Summable-normal-form}1-summable normal form up to arbitrary
order $N$}

~

We consider now a (formal) non-degenerate diagonal doubly-resonant
saddle node, which is supposed to be div-integrable and 1-summable
in every direction $\theta\neq\arg\left(\pm\lambda\right)$, of the
form 
\begin{eqnarray*}
\hat{Y}_{\tx{prep}} & = & x^{2}\pp x+\left(-\lambda+a_{1}x-d\left(y_{1}y_{2}\right)+x\hat{S}_{1}\left(x,\mathbf{y}\right)\right)y_{1}\pp{y_{1}}\\
 &  & +\left(\lambda+a_{2}x+d\left(y_{1}y_{2}\right)+x\hat{S}_{2}\left(x,\mathbf{y}\right)\right)y_{2}\pp{y_{2}}\qquad,
\end{eqnarray*}
where:
\begin{itemize}
\item $\lambda\in\ww C\backslash\acc 0$;
\item $\hat{S}_{1},\hat{S}_{2}\in\form{x,\mathbf{y}}$ are of order at least
one with respect to $\mathbf{y}$ and 1-summable in every direction
$\theta\in\ww R$ with $\theta\neq\arg\left(\pm\lambda\right)$;
\item $a:=\tx{res}\left(\hat{Y}_{\tx{prep}}\right)=a_{1}+a_{2}\notin\ww Q_{\leq0}$
;
\item $d\left(v\right)\in v\ww C\acc v$ is the germ of an analytic function
in $v:=y_{1}y_{2}$ vanishing at the origin.
\end{itemize}
As usual, we denote by $Y_{\tx{prep},\theta},S_{1,\theta},S_{2,\theta}$
the respective 1-sums of $\hat{Y},\hat{S}_{1},\hat{S}_{2}$ in the
direction $\theta$. Let us introduce some useful notations:

\begin{eqnarray*}
\hat{Y}_{\tx{prep}} & = & Y_{0}+D\overrightarrow{\mathcal{C}}+R\overrightarrow{\mathcal{R}}\qquad,
\end{eqnarray*}
where
\begin{itemize}
\item $\overrightarrow{\cal C}:=-y_{1}\pp{y_{1}}+y_{2}\pp{y_{2}}$
\item $\overrightarrow{\cal R}:=y_{1}\pp{y1}+y_{2}\pp{y_{2}}$
\item $Y_{0}:=\lambda\overrightarrow{\cal C}+x\left(x\pp x+a_{1}y_{1}\pp{y_{1}}+a_{2}y_{2}\pp{y2}\right)$ 
\item ${\displaystyle D\left(x,\mathbf{y}\right)=d\left(y_{1}y_{2}\right)+xD^{\left(1\right)}\left(x,\mathbf{y}\right)=d\left(y_{1}y_{2}\right)+x\left(\frac{\hat{S}_{2}-\hat{S}_{1}}{2}\right)}$
is 1-summable in the direction $\theta$ of 1-sum $D_{\theta}:$ it
is called the ``\emph{tangential}'' part. $D_{\theta}$ is also
dominated by $\norm{\mathbf{y}}=\max\left(\abs{y_{1}},\abs{y_{2}}\right)$
($D$ is of order at least one with respect to $\mathbf{y}$).
\item ${\displaystyle R\left(x,\mathbf{y}\right)=xR^{\left(1\right)}\left(x,\mathbf{y}\right)=x\left(\frac{\hat{S}_{2}+\hat{S}_{1}}{2}\right)}$
is 1-summable in the direction $\theta$ of 1-sum $R_{\theta}$: it
is called the ``\emph{radial}'' part. $R_{\theta}$ is also dominated
by $\norm{\mathbf{y}}_{\infty}=\max\left(\abs{y_{1}},\abs{y_{2}}\right)$
($R$ is of order at least one with respect to $\mathbf{y}$).
\end{itemize}
The following proposition gives the existence of a 1-summable normalizing
map, up to any order $N\in\ww N_{>0}$, with respect to $x$.
\begin{prop}
\label{prop: ordre N avec eq homologique}Let 
\begin{eqnarray*}
\hat{Y}_{\tx{prep}} & = & Y_{0}+D\overrightarrow{\mathcal{C}}+R\overrightarrow{\mathcal{R}}\qquad
\end{eqnarray*}
be as above. 

Then for all $N\in\ww N_{>0}$ there exist $d^{\left(N\right)}\left(v\right)\in\ww C\acc v$
of order at least one and $\Phi^{\left(N\right)}\in\fdiff[\ww C^{3},0,\tx{Id}]$
which conjugates $\hat{Y}_{\tx{prep}}$ (\emph{resp. }its 1-sums $Y_{\tx{prep},\theta}$
in the direction $\theta$) to 
\begin{eqnarray*}
Y^{\left(N\right)} & = & Y_{0}+\left(d^{\left(N\right)}\left(y_{1}y_{2}\right)+x^{N}D^{\left(N\right)}\left(x,\mathbf{y}\right)\right)\overrightarrow{\cal C}+x^{N}R^{\left(N\right)}\left(x,\mathbf{y}\right)\overrightarrow{\cal R}\\
\bigg(\emph{resp. }\, Y_{\theta}^{\left(N\right)} & = & Y_{0}+\left(d^{\left(N\right)}\left(y_{1}y_{2}\right)+x^{N}D_{\pm}^{\left(N\right)}\left(x,\mathbf{y}\right)\right)\overrightarrow{\cal C}+x^{N}R_{\theta}^{\left(N\right)}\left(x,\mathbf{y}\right)\overrightarrow{\cal R}\bigg)\qquad,
\end{eqnarray*}
where $D^{\left(N\right)},R^{\left(N\right)}$ are 1-summable in the
direction $\theta$, of order at least one with respect to $\mathbf{y}$,
of 1-sums $D_{\theta}^{\left(N\right)},R_{\theta}^{\left(N\right)}$
in the direction $\theta$. Moreover, one can choose $d^{\left(2\right)}=\dots=d^{\left(N\right)}$
for all $N\geq2$, and $d^{\left(1\right)}=d$.\end{prop}
\begin{proof}
The proof is performed by induction on $N$.
\begin{itemize}
\item The case $N=1$ is the initial situation here, and is already proved
with $\hat{Y}_{\tx{prep}}=Y^{\left(1\right)}$.
\item Assume that the result holds for $N\in\ww N_{>0}$. 

\begin{enumerate}
\item We start with the radial part. Let us write
\[
R^{\left(N\right)}\left(x,\mathbf{y}\right)=\sum_{n_{1}+n_{2}\geq1}R_{n_{1},n_{2}}^{\left(N\right)}\left(x\right)y_{1}^{n_{1}}y_{2}^{n_{2}}
\]
and
\[
R_{\mbox{res}}^{\left(N\right)}\left(0,v\right)=\sum_{k\geq1}R_{k,k}^{\left(N\right)}\left(0\right)v^{k}\quad.
\]
We are looking for an analytic solution $\tau$ to the equations:
\begin{eqnarray}
\cal L_{Y^{\left(N\right)}}\left(\tau\right) & = & -x^{N}R^{\left(N\right)}+\left(x^{N+1}\tilde{R}^{\left(N+1\right)}\right)\circ\Lambda_{\tau}\label{eq: eq homologic rec}\\
\cal L_{Y_{\theta}^{\left(N\right)}}\left(\tau\right) & = & -x^{N}R_{\theta}^{\left(N\right)}+\left(x^{N+1}\tilde{R}_{\theta}^{\left(N+1\right)}\right)\circ\Lambda_{\tau}\qquad,\nonumber 
\end{eqnarray}
for a convenient choice of $\tilde{R}^{\left(N+1\right)},\tilde{R}_{\theta}^{\left(N+1\right)}$,
with 
\[
\Lambda_{\tau}\left(x,\mathbf{y}\right):=\left(x,y_{1}\exp\left(\tau\left(x,\mathbf{y}\right)\right),y_{2}\exp\left(\tau\left(x,\mathbf{y}\right)\right)\right)\qquad,
\]
and 
\[
\tau\left(x,\mathbf{y}\right)=x^{N-1}\tau_{0}\left(y_{1}y_{2}\right)+x^{N}\tau_{1}\left(\mathbf{y}\right)\quad,
\]
where ${\displaystyle \tau_{1}\left(\mathbf{y}\right)=\sum_{j_{1}\neq j_{2}}\tau_{1,j_{1}j_{2}}y_{1}^{j_{1}}y_{2}^{j_{2}}}$.
More concretely, $\Lambda_{\tau}$ is the formal flow of $\overrightarrow{\cal R}$
at ``time'' $\tau\left(x,\mathbf{y}\right)$. \\
If we admit for a moment that such an analytic solution $\tau$ exists,
then $\Lambda_{\tau}\in\fdiffid$ and therefore $\Lambda_{\tau}^{-1}\in\fdiffid$.
If we consider $d^{\left(N\right)}$ and $\tilde{D}^{\left(N\right)}$
such that
\begin{eqnarray*}
 & d^{\left(N+1\right)}\left(z_{1}z_{2}\right)+x^{N}\tilde{D}^{\left(N\right)}\left(x,\mathbf{z}\right):=\left(d^{\left(N\right)}\left(y_{1}y_{2}\right)+x^{N}D^{\left(N\right)}\left(x,\mathbf{y}\right)\right)\circ\Lambda_{\tau}^{-1}\left(x,\mathbf{z}\right)\\
 & d^{\left(N+1\right)}\left(z_{1}z_{2}\right)+x^{N}\tilde{D_{\theta}}^{\left(N\right)}\left(x,\mathbf{z}\right):=\left(d^{\left(N\right)}\left(y_{1}y_{2}\right)+x^{N}D_{\theta}^{\left(N\right)}\left(x,\mathbf{y}\right)\right)\circ\Lambda_{\tau}^{-1}\left(x,\mathbf{z}\right),
\end{eqnarray*}
then the two equations given in $\left(\mbox{\ref{eq: eq homologic rec}}\right)$
imply that 
\begin{eqnarray*}
\left(\Lambda_{\tau}\right)_{*}\left(Y^{\left(N\right)}\right) & = & Y_{0}+\left(d^{\left(N+1\right)}\left(z_{1}z_{2}\right)+x^{N}\tilde{D}^{\left(N\right)}\left(x,\mathbf{z}\right)\right)\overrightarrow{\cal C}\\
 &  & +x^{N+1}\tilde{R}^{\left(N+1\right)}\left(x,\mathbf{z}\right)\overrightarrow{\cal R}\\
\left(\Lambda_{\tau}\right)_{*}\left(Y_{\theta}^{\left(N\right)}\right) & = & Y_{0}+\left(d^{\left(N+1\right)}\left(z_{1}z_{2}\right)+x^{N}\tilde{D}_{\theta}^{\left(N\right)}\left(x,\mathbf{z}\right)\right)\overrightarrow{\cal C}\\
 &  & +x^{N+1}\tilde{R}_{\theta}^{\left(N+1\right)}\left(x,\mathbf{z}\right)\overrightarrow{\cal R}\qquad.
\end{eqnarray*}
Indeed: 
\begin{eqnarray*}
 &  & \mbox{D}\Lambda_{\tau}\cdot Y^{\left(N\right)}=\left(\begin{array}{c}
\cal L_{Y^{\left(N\right)}}\left(x\right)\\
{\cal L}_{Y^{\left(N\right)}}\left(y_{1}\exp\left(\tau\left(x,\mathbf{y}\right)\right)\right)\\
\cal L_{Y^{\left(N\right)}}\left(y_{2}\exp\left(\tau\left(x,\mathbf{y}\right)\right)\right)
\end{array}\right)\\
 &  & =\left(\begin{array}{c}
x^{2}\\
\left(\cal L_{Y^{\left(N\right)}}\left(y_{1}\right)+y_{1}\left(\cal L_{Y^{\left(N\right)}}\left(\tau\right)\right)\right)\exp\left(\tau\left(x,\mathbf{y}\right)\right)\\
\left(\cal L_{Y^{\left(N\right)}}\left(y_{2}\right)+y_{2}\left(\cal L_{Y^{\left(N\right)}}\left(\tau\right)\right)\right)\exp\left(\tau\left(x,\mathbf{y}\right)\right)
\end{array}\right)\\
 &  & =\left(Y_{0}+\left(d^{\left(N+1\right)}+x^{N}\tilde{D}^{\left(N\right)}\right)\overrightarrow{\cal C}+x^{N+1}\tilde{R}^{\left(N+1\right)}\overrightarrow{\cal R}\right)\circ\Lambda_{\tau}\left(x,\mathbf{y}\right)\,\,\,.
\end{eqnarray*}
These computations are also true with the corresponding 1-sums of
formal objects considered here, \emph{i.e.} with $Y_{\theta}^{\left(N\right)},D_{\theta}^{\left(N\right)},\tilde{D}_{\theta}^{\left(N\right)},\tilde{R}_{\theta}^{\left(N+1\right)}$
instead of $Y^{\left(N\right)},D^{\left(N\right)},\tilde{D}^{\left(N\right)},\tilde{R}^{\left(N+1\right)}$
respectively. We use Proposition \ref{prop: compositon summable}
to obtain the 1-summability of the objects defined by compositions.\\
Let us prove that there exists a germ of analytic function of the
form 
\[
\tau\left(x,\mathbf{y}\right)=x^{N-1}\tau_{0}\left(y_{1}y_{2}\right)+x^{N}\tau_{1}\left(\mathbf{y}\right)\quad,
\]
of order ate least one with respect to $\mathbf{y}$ in the origin,
with 
\[
{\displaystyle \tau_{1}\left(\mathbf{y}\right)=\sum_{j_{1}\neq j_{2}}\tau_{1,j_{1}j_{2}}y_{1}^{j_{1}}y_{2}^{j_{2}}}
\]
satisfying equation $\left(\mbox{\ref{eq: eq homologic rec}}\right)$.
This equation can be written 
\begin{eqnarray*}
 &  & x^{2}\ppp{\tau}x+\left(-\lambda+a_{1}x-d^{\left(N\right)}\left(y_{1}y_{2}\right)-x^{N}D^{\left(N\right)}\left(x,\mathbf{y}\right)+x^{N}R^{\left(N\right)}\left(x,\mathbf{y}\right)\right)y_{1}\ppp{\tau}{y_{1}}\\
 &  & +\left(\lambda+a_{2}x+d^{\left(N\right)}\left(y_{1}y_{2}\right)+x^{N}D^{\left(N\right)}\left(x,\mathbf{y}\right)+x^{N}R^{\left(N\right)}\left(x,\mathbf{y}\right)\right)y_{2}\ppp{\tau}{y_{2}}\\
 &  & =-x^{N}R^{\left(N\right)}+\left(x^{N+1}\tilde{R}^{\left(N+1\right)}\right)\circ\Lambda_{\tau}\quad,
\end{eqnarray*}
or equivalently
\begin{eqnarray*}
 & x^{2}\ppp{\tau}x+a_{1}xy_{1}\ppp{\tau}{y_{1}}+a_{2}xy_{2}\ppp{\tau}{y_{2}}+\left(\lambda+d^{\left(N\right)}\left(y_{1}y_{2}\right)+x^{N}D^{\left(N\right)}\left(x,\mathbf{y}\right)\right)\cal L_{\overrightarrow{\cal C}}\left(\tau\right)\\
 & +\left(x^{N}R^{\left(N\right)}\left(x,\mathbf{y}\right)\right)\cal L_{\overrightarrow{\cal R}}\left(\tau\right)=-x^{N}R^{\left(N\right)}+\left(x^{N+1}\tilde{R}^{\left(N+1\right)}\right)\circ\Lambda_{\tau}\quad.
\end{eqnarray*}
Let us consider terms of degree $N$ with respect to $x$:
\begin{eqnarray}
 & \left(N-1\right)\tau_{0}\left(y_{1}y_{2}\right)+\left(a_{1}+a_{2}+2\delta_{N,1}R^{\left(N\right)}\left(0,\mathbf{y}\right)\right)y_{1}y_{2}\ppp{\tau_{0}}v\left(y_{1}y_{2}\right)\nonumber \\
 & +\left(\lambda+d^{\left(N\right)}\left(y_{1}y_{2}\right)\right)\cal L_{\overrightarrow{\cal C}}\left(\tau_{1}\right)=-R^{\left(N\right)}\left(0,\mathbf{y}\right)\label{eq: eq homo rec degre 1 wrt x}
\end{eqnarray}
(here $\delta_{N,1}$ is the Kronecker notation). We use now the fact
that $\tx{Im}\left(\lie{\overrightarrow{\cal C}}\right)\varoplus\tx{Ker}\left(\lie{\overrightarrow{\cal C}}\right)$
is a direct sum, and that $\tx{Ker}\left(\lie{\overrightarrow{\cal C}}\right)$
is the set of formal power series in the resonant monomial $v=y_{1}y_{2}$.
Isolating the term $\lie{\overrightarrow{\cal C}}\left(\tau_{1}\right)$
on the one hand, and the others on the other hand, the direct sum
above gives us: 
\[
\begin{cases}
{\displaystyle v\left(a_{1}+a_{2}+2\delta_{N,1}R_{\mbox{res}}^{\left(N\right)}\left(0,v\right)\right)\ddd{\tau_{0}}v\left(v\right)+\left(N-1\right)\tau_{0}\left(v\right)=-R_{\mbox{res}}^{\left(N\right)}\left(0,v\right)}\\
{\displaystyle \tau_{0}\left(0\right)=0}
\end{cases}
\]
 and
\[
\begin{cases}
{\displaystyle \cal L_{\overrightarrow{\cal C}}\left(\tau_{1}\right)=\frac{-1}{\lambda+d^{\left(N\right)}\left(y_{1}y_{2}\right)}\Bigg(\left(2\delta_{N,1}\left(R^{\left(N\right)}\left(0,\mathbf{y}\right)-R_{\mbox{res}}^{\left(N\right)}\left(0,v\right)\right)\right)y_{1}y_{2}\ddd{\tau_{0}}v\left(y_{1}y_{2}\right)}\\
{\displaystyle \qquad\,\,+R^{\left(N\right)}\left(0,\mathbf{y}\right)-R_{\mbox{res}}^{\left(N\right)}\left(0,v\right)\Bigg)}\\
{\displaystyle \tau_{1}\left(0\right)=0\,\,.}
\end{cases}
\]
Since $R^{\left(N\right)}$ is analytic with respect to $\mathbf{y}$,
$R_{\mbox{res}}^{\left(N\right)}\left(0,v\right)$ is analytic near
$v=0$. Furthermore, as $R_{\mbox{res}}^{\left(N\right)}\left(0,0\right)=0$
and $a_{1}+a_{2}\notin\ww Q_{\leq0}$, the first of the two equation
above has a unique formal solution $\tau_{0}$ with $\tau_{0}\left(0\right)$,
and this solution is convergent in a neighborhood of the origin. Once
$\tau_{0}$ is determined, there exists a unique formal solution $\tau_{1}$
to the second equation satisfying ${\displaystyle \tau_{1}\left(\mathbf{y}\right)=\sum_{j_{1}\neq j_{2}}\tau_{1,j_{1}j_{2}}y_{1}^{j_{1}}y_{2}^{j_{2}}}$,
which is moreover convergent in a neighborhood of the origin of $\ww C^{2}$
.\\
Therefore $\Lambda_{\tau}$ is a germ of analytic diffeomorphism fixing
the origin, fibered, tangent to the identity and conjugates $Y^{\left(N\right)}$
\emph{$\big($resp.} $Y_{\theta}^{\left(N\right)}$$\big)$ to ${\displaystyle \tilde{Y}^{\left(N\right)}:=\left(\Lambda_{\tau}\right)_{*}\left(Y^{\left(N\right)}\right)}$
$\big($\emph{resp.} ${\displaystyle \tilde{Y}_{\theta}^{\left(N\right)}:=\left(\Lambda_{\tau}\right)_{*}\left(Y_{\theta}^{\left(N\right)}\right)}$$\big)$.\\
Equation $\left(\mbox{\ref{eq: eq homo rec degre 1 wrt x}}\right)$
implies that ${\displaystyle \left(\cal L_{Y^{\left(N\right)}}\left(\tau\right)+x^{N}R^{\left(N\right)}\right)}$
and ${\displaystyle \left(\cal L_{Y_{\theta}^{\left(N\right)}}\left(\tau\right)+x^{N}R_{\theta}^{\left(N\right)}\right)}$
are divisible by $x^{N+1}$, so that we can define:
\begin{eqnarray*}
\tilde{R}^{\left(N+1\right)}\left(x,\mathbf{z}\right) & := & \left(\frac{\cal L_{Y^{\left(N\right)}}\left(\tau\right)+x^{N}R^{\left(N\right)}}{x^{N+1}}\right)\circ\Lambda_{\tau}^{-1}\left(x,\mathbf{z}\right)\\
\tilde{R}_{\theta}^{\left(N+1\right)}\left(x,\mathbf{z}\right) & := & \left(\frac{\cal L_{Y_{\theta}^{\left(N\right)}}\left(\tau\right)+x^{N}R_{\theta}^{\left(N\right)}}{x^{N+1}}\right)\circ\Lambda_{\tau}^{-1}\left(x,\mathbf{z}\right)\quad.
\end{eqnarray*}
By Proposition \ref{prop: compositon summable}, $\tilde{R}^{\left(N+1\right)}$
$\big($\emph{resp.} $\tilde{D}^{\left(N\right)}$$\big)$ is 1-summable
in the direction $\theta$, of 1-sum $\tilde{R}_{\theta}^{\left(N+1\right)}$
$\big($\emph{resp.} $\tilde{D}_{\theta}^{\left(N\right)}$ $\big)$.\\
Finally, notice that ${\displaystyle d^{\left(N+1\right)}\circ\Lambda_{\tau}\left(0,\mathbf{y}\right)=d^{\left(N\right)}\left(y_{1},y_{2}\right)}$,
$\tau\left(0,\mathbf{y}\right)=0$ and then $\Lambda_{\tau}\left(0,\mathbf{y}\right)=\left(0,y_{1},y_{2}\right)$
if $N>1$, so that $d^{\left(N+1\right)}=d^{\left(N\right)}$ when
$N>1$.
\item No we deal with the tangential part. Let us write
\[
\tilde{D}^{\left(N\right)}\left(x,\mathbf{z}\right)=\sum_{n_{1}+n_{2}\geq1}\tilde{D}_{n_{1},n_{2}}^{\left(N\right)}\left(x\right)z_{1}^{n_{1}}z_{2}^{n_{2}}
\]
and
\[
\tilde{D}_{\mbox{res}}^{\left(N\right)}\left(0,v\right)=\sum_{k\geq1}\tilde{D}_{k,k}^{\left(N\right)}\left(0\right)v^{k}\quad.
\]
Exactly as in the previous case which dealt with the ``radial part''
(in fact the computations are even easier here), we can prove the
existence of a germ of an analytic function $\sigma$, solution to
the equation: 
\begin{eqnarray}
\cal L_{\tilde{Y}^{\left(N\right)}}\left(\sigma\right) & = & -x^{N}\tilde{D}^{\left(N\right)}+\left(x^{N+1}D^{\left(N+1\right)}\right)\circ\Gamma_{\sigma}\nonumber \\
\cal L_{\tilde{Y}_{\theta}^{\left(N\right)}}\left(\sigma\right) & = & -x^{N}\tilde{D}_{\theta}^{\left(N\right)}+\left(x^{N+1}D_{\theta}^{\left(N+1\right)}\right)\circ\Gamma_{\sigma}\qquad,\label{eq: eq homologic rec-1}
\end{eqnarray}
for a good choice of $D^{\left(N+1\right)},D_{\theta}^{\left(N+1\right)}$,
with 
\[
\Gamma_{\sigma}\left(x,\mathbf{z}\right):=\left(x,y_{1}\exp\left(-\sigma\left(x,\mathbf{z}\right)\right),y_{2}\exp\left(\sigma\left(x,\mathbf{z}\right)\right)\right)\qquad
\]
and 
\[
\sigma\left(x,\mathbf{z}\right)=x^{N-1}\sigma_{0}\left(z_{1}z_{2}\right)+x^{N}\sigma_{1}\left(\mathbf{z}\right)\qquad,
\]
where ${\displaystyle \sigma_{1}\left(\mathbf{z}\right)=\sum_{j_{1}\neq j_{2}}\sigma_{1,j_{1},j_{2}}z_{1}^{j_{1}}z_{2}^{j_{2}}}$.
Notice that $\Gamma_{\sigma}$ is the formal flow of $\overrightarrow{\cal C}$
at ``time'' $\sigma\left(x,\mathbf{z}\right)$.\\
Again, as in the first case with the ``radial part'', we have on
a $\Gamma_{\sigma}\in\fdiffid$ and then also $\Gamma_{\sigma}^{-1}\in\fdiffid$.
If we consider $R^{\left(N+1\right)}$ and $R_{\theta}^{\left(N+1\right)}$
such that
\begin{eqnarray*}
R^{\left(N+1\right)}\left(x,\mathbf{y}\right) & := & \tilde{R}^{\left(N+1\right)}\circ\Gamma_{\sigma}^{-1}\left(x,\mathbf{y}\right)\\
R_{\theta}^{\left(N+1\right)}\left(x,\mathbf{z}\right) & := & \tilde{R}_{\theta}^{\left(N+1\right)}\circ\Gamma_{\sigma}^{-1}\left(x,\mathbf{y}\right)\qquad,
\end{eqnarray*}
then it follows from $\left(\mbox{\ref{eq: eq homologic rec-1}}\right)$
that 
\begin{eqnarray*}
\left(\Gamma_{\sigma}\right)_{*}\left(\tilde{Y}^{\left(N\right)}\right) & = & Y_{0}+\left(d^{\left(N+1\right)}\left(y_{1}y_{2}\right)+x^{N+1}D^{\left(N+1\right)}\left(x,\mathbf{y}\right)\right)\overrightarrow{\cal C}\\
 &  & +x^{N+1}R^{\left(N+1\right)}\left(x,\mathbf{y}\right)\overrightarrow{\cal R}\\
\left(\Gamma_{\sigma}\right)_{*}\left(\tilde{Y}_{\theta}^{\left(N\right)}\right) & = & Y_{0}+\left(d^{\left(N+1\right)}\left(y_{1}y_{2}\right)+x^{N+1}D_{\theta}^{\left(N+1\right)}\left(x,\mathbf{y}\right)\right)\overrightarrow{\cal C}\\
 &  & +x^{N+1}R_{\theta}^{\left(N+1\right)}\left(x,\mathbf{y}\right)\overrightarrow{\cal R}\qquad.
\end{eqnarray*}
In fact, we choose:
\begin{eqnarray*}
D^{\left(N+1\right)}\left(x,\mathbf{y}\right) & := & \left(\frac{\cal L_{\tilde{Y}^{\left(N\right)}}\left(\sigma\right)+x^{N}\tilde{D}^{\left(N\right)}}{x^{N+1}}\right)\circ\Gamma_{\sigma}^{-1}\left(x,\mathbf{y}\right)\\
D_{\theta}^{\left(N+1\right)}\left(x,\mathbf{y}\right) & := & \left(\frac{\cal L_{\tilde{Y}_{\theta}^{\left(N\right)}}\left(\sigma\right)+x^{N}\tilde{D}_{\theta}^{\left(N\right)}}{x^{N+1}}\right)\circ\Gamma_{\sigma}^{-1}\left(x,\mathbf{y}\right)\quad.
\end{eqnarray*}
By Proposition \ref{prop: compositon summable}, $D^{\left(N+1\right)}$
$\big($\emph{resp.} $R^{\left(N+1\right)}$$\big)$ is 1-summable
in the direction $\theta$, of 1-sum $D_{\theta}^{\left(N+1\right)}$
$\big($\emph{resp.} $R_{\theta}^{\left(N+1\right)}$ $\big)$.
\end{enumerate}
\end{itemize}
\end{proof}

\subsection{Proof of Proposition \ref{prop: forme pr=0000E9par=0000E9e ordre N}}

~

We now give a short proof of Proposition \ref{prop: forme pr=0000E9par=0000E9e ordre N},
using the different results proved in this section.
\begin{proof}[\emph{Proof of Proposition \ref{prop: forme pr=0000E9par=0000E9e ordre N}.}]

We just have to use consecutively Proposition \ref{prop: preparation sur hypersurface invariante}
(applied to $Y_{0}:=Y_{\mid\acc{x=0}}$), Proposition \ref{prop: diag prep},
Proposition \ref{prop: preparation} and finally Proposition \ref{prop: ordre N avec eq homologique},
using at each time Corollary \ref{cor: summability push-forward}
in order to obtain the directional 1-summability.
\end{proof}

\section{\label{sec:Sectorial-analytic-normalization}Sectorial analytic normalization}

The aim of this section is to prove that for any $Y\in\snodiag$ and
for any ${\displaystyle \eta\in\big[\pi,2\pi\big[}$, there exists
a pair 
\[
\left(\Phi_{+},\Phi_{-}\right)\in\diffsect[\arg\left(i\lambda\right)][\eta]\times\diffsect[\arg\left(-i\lambda\right)][\eta]
\]
whose elements analytically conjugate $Y$ to its normal form $\ynorm$
(given by Theorem \ref{thm: forme normalel formelle}) in sectorial
neighborhoods of the origin with wide opening. The uniqueness of $\Phi_{+}$
and $\Phi_{-}$ will be proved in the next section. The existence
of sectorial normalizing maps $\Phi_{+}$ and $\Phi_{-}$ in domains
of the form $\cal S_{+}\in\cal S_{\arg\left(i\lambda\right),\eta}$
and $\cal S_{-}\in\cal S_{\arg\left(-i\lambda\right),\eta}$ for all
${\displaystyle \eta\in\big[\pi,2\pi\big[}$, is equivalent to the
existence of a sectorial normalizing map $\Phi_{\theta}$ in domains
$\cal S\in{\cal S}_{\theta,\pi}$, for all $\theta\in\ww R$ such
that $\theta\neq\arg\left(\pm\lambda\right)$. In the next section
we will also prove that $\Phi_{+}$ and $\Phi_{-}$ both admit the
unique formal normalizing map $\hat{\Phi}$ (given by Theorem \ref{thm: forme normalel formelle})
as weak Gevrey-1 asymptotic expansion in domains $\cal S_{+}\in\cal S_{\arg\left(i\lambda\right),\eta}$
and $\cal S_{-}\in\cal S_{\arg\left(-i\lambda\right),\eta}$ respectively.
In particular, this will prove that $\hat{\Phi}$ is weakly 1-summable
in every direction $\theta\neq\arg\left(\pm\lambda\right)$.

We start with a vector field $Y^{\left(N\right)}$ normalized up to
order $N\geq2$ as in Proposition \ref{prop: forme pr=0000E9par=0000E9e ordre N}.
First of all, we prove the existence of germs of sectorial analytic
functions $\alpha_{+}\in\cal O\left(\cal S_{+}\right),\alpha_{-}\in\cal O\left(\cal S_{-}\right)$,
which are solutions to homological equations of the form: 
\[
\cal L_{Y^{\left(N\right)}}\left(\alpha_{\pm}\right)=x^{M+1}A_{\pm}\left(x,\mathbf{y}\right)\qquad,
\]
where $M\in\ww N_{>0}$ and $A_{\pm}\in{\cal O}\left(\cal S_{\pm}\right)$
is analytic in $\cal S_{\pm}$ (see Lemma \ref{lem: solution eq homo}).
In order to construct such solutions, we will integrate some appropriate
meromorphic 1-form on asymptotic paths (see subsection \ref{sub: proof lemma}).
Once we have these solutions $\alpha_{+},\alpha_{-}$, we will construct
the desired germs of sectorial diffeomorphisms as the flows of some
elementary linear vector fields at ``time'' $\alpha_{\pm}\left(x,\mathbf{y}\right)$.
After that, we will prove in subsection \ref{sub :Sectorial isotropies in big sectors}
that there exist unique germs of sectorial fibered diffeomorphisms
tangent to the identity which conjugate $Y\in\snofib$ to its normal
form, by studying the sectorial isotropies in sectorial domains with
wide opening.

We go on using the notations introduced in subsection \ref{sub:Summable-normal-form},
\emph{i.e.}
\begin{itemize}
\item $\lambda\in\ww C^{*}$
\item $a_{1}+a_{2}\notin\ww Q_{\leq0}$
\item $\overrightarrow{\cal C}:=-y_{1}\pp{y_{1}}+y_{2}\pp{y_{2}}$
\item $\overrightarrow{\cal R}:=y_{1}\pp{y1}+y_{2}\pp{y_{2}}$
\item $Y_{0}:=\lambda\overrightarrow{\cal C}+x\left(x\pp x+a_{1}y_{1}\pp{y_{1}}+a_{2}y_{2}\pp{y2}\right)$.
\end{itemize}
For ${\displaystyle \epsilon\in\left]0,\frac{\pi}{2}\right[}$ and
$r>0$, we will consider two sectors, namely 
\[
S_{+}\left(r,\epsilon\right):=S\left(r,\arg\left(i\lambda\right)-\frac{\pi}{2}-\epsilon,\arg\left(i\lambda\right)+\frac{\pi}{2}+\epsilon\right)
\]
 and 
\[
S_{-}\left(r,\epsilon\right)=S\left(r,\arg\left(-i\lambda\right)-\frac{\pi}{2}-\epsilon,\arg\left(-i\lambda\right)+\frac{\pi}{2}+\epsilon\right).
\]

Let us consider a (weakly) 1-summable non-degenerate div-integrable
doubly-resonant saddle-node normalized up to an order $N+2$, with
$N>0$: 
\begin{eqnarray*}
Y^{\left(N+2\right)} & = & Y_{0}+\left(c\left(y_{1}y_{2}\right)+x^{N+2}D^{\left(N+2\right)}\left(x,\mathbf{y}\right)\right)\overrightarrow{\cal C}+x^{N+2}R^{\left(N+2\right)}\left(x,\mathbf{y}\right)\overrightarrow{\cal R}\\
 &  & \mbox{(formal)}\\
Y_{\pm}^{\left(N+2\right)} & = & Y_{0}+\left(c\left(y_{1}y_{2}\right)+x^{N+2}D_{\pm}^{\left(N+2\right)}\left(x,\mathbf{y}\right)\right)\overrightarrow{\cal C}+x^{N+2}R_{\pm}^{\left(N+2\right)}\left(x,\mathbf{y}\right)\overrightarrow{\cal R}\\
 &  & \mbox{(analytic in }S_{\pm}\left(r,\epsilon\right)\times\mathbf{D\left(0,r\right)})
\end{eqnarray*}
where $D^{\left(N+2\right)},R^{\left(N+2\right)}$ are of order at
least one with respect to $\mathbf{y}$, and (weak) 1-summable in
every direction $\theta\in\ww R$ with $\theta\neq\arg\left(\pm\lambda\right)$:
their respective (weak) 1-sums in the direction $\arg\left(\pm i\lambda\right)$
are $D_{\pm}^{\left(N+2\right)},R_{\pm}^{\left(N+2\right)}$, which
can be analytically extended in $S_{\pm}\left(r,\epsilon\right)\times\left(\ww C^{2},0\right)$.
In order to have the complete sectorial normalizing map, we have to
assume now that our vector field is \textbf{\emph{strictly non-degenerate}},
\emph{i.e.}

\[
\fbox{\ensuremath{\Re\left(a_{1}+a_{2}\right)}>0}\,\,\,.
\]

\begin{prop}
\label{prop: norm sect}Under the assumptions above, for all $\eta\in\left]\pi,2\pi\right[$,
there exist two germs of sectorial fibered diffeomorphisms
\[
\begin{cases}
\Psi_{+}\in\diffsect[\arg\left(i\lambda\right)][\eta]\\
\Psi_{-}\in\diffsect[\arg\left(-i\lambda\right)][\eta]
\end{cases}
\]
 of the form 
\begin{eqnarray*}
\Psi_{\pm}:\left(x,\mathbf{y}\right) & \mapsto & \left(x,\mathbf{y}+\tx O\left(\norm{\mathbf{y}}^{2}\right)\right)\,\,,
\end{eqnarray*}
 which conjugate $Y_{\pm}^{\left(N+2\right)}$ to its formal normal
form 
\[
\ynorm=x^{2}\pp x+\left(-\lambda+a_{1}x-c\left(y_{1}y_{2}\right)\right)y_{1}\pp{y_{1}}+\left(\lambda+a_{2}x+c\left(y_{1}y_{2}\right)\right)y_{2}\pp{y_{2}}\qquad,
\]
where $c\left(v\right)\in v\ww C\acc v$ is the germ of an analytic
function in $v:=y_{1}y_{2}$ vanishing at the origin. Moreover, we
can choose $\Psi_{\pm}$ above such that 
\[
\Psi_{\pm}\left(x,\mathbf{y}\right)=\tx{Id}\left(x,\mathbf{y}\right)+x^{N}\mathbf{P}_{\pm}^{\left(N\right)}\left(x,\mathbf{y}\right)\,\,,
\]
where $\mathbf{P}_{\pm}^{\left(N\right)}=\left(0,P_{1,\pm},P_{2,\pm}\right)$
is analytic in $S_{\pm}\left(r,\epsilon\right)\times\left(\ww C^{2},0\right)$
(for some $r>0$ and $\epsilon>\frac{\eta}{2}$) and of order at least
two with respect to $\mathbf{y}$.
\end{prop}
By combining Propositions \ref{prop: forme pr=0000E9par=0000E9e ordre N}
and \ref{prop: norm sect} we immediately obtain the following result.
\begin{cor}
\label{cor: existence normalisations sectorielles}Let $Y\in\snofib$
be a strictly non-degenerate diagonal doubly-resonant saddle-node
which is div-integrable. Then, for all $\eta\in\left]\pi,2\pi\right[$,
there exist two germs of sectorial fibered diffeomorphisms
\[
\begin{cases}
\Phi_{+}\in\diffsect[\arg\left(i\lambda\right)][\eta]\\
\Phi_{-}\in\diffsect[\arg\left(-i\lambda\right)][\eta]
\end{cases}
\]
 tangent to the identity such that:
\begin{eqnarray*}
\left(\Phi_{\pm}\right)_{*}\left(Y\right) & = & x^{2}\pp x+\left(-\lambda+a_{1}x-c\left(y_{1}y_{2}\right)\right)y_{1}\pp{y_{1}}+\left(\lambda+a_{2}x+c\left(y_{1}y_{2}\right)\right)y_{2}\pp{y_{2}}\\
 & =: & \ynorm\,\,,
\end{eqnarray*}
where $\lambda\in\ww C^{*}$, $\Re\left(a_{1}+a_{2}\right)>0$, and
$c\left(v\right)\in v\ww C\acc v$ is the germ of an analytic function
in $v:=y_{1}y_{2}$ vanishing at the origin.
\end{cor}
As already mentioned, we prove in the next section that $\Phi_{+}$
and $\Phi_{-}$ are unique as germs (see Proposition \ref{prop: unique normalizations}),
and that they are the weak 1-sums of the unique formal normalizing
map $\hat{\Phi}$ given by Theorem \ref{thm: forme normalel formelle}
(see Proposition \ref{prop: Weak sectorial normalizations}).

\subsection{Proof of Proposition \ref{prop: norm sect}. }

~

We give here two consecutive propositions which allow to prove Proposition
\ref{prop: norm sect} as an immediate consequence. When we say that
a function $f:U\rightarrow\ww C$ is \emph{dominated} by another $g:U\rightarrow\ww R_{+}$
in $U$, it means that there exists $L>0$ such that for all $u\in U$,
we have $\abs{f\left(u\right)}\leq L.g\left(u\right)$.
\begin{prop}
\label{prop: radial part}

Let ${\displaystyle Y_{\pm}^{\left(N+2\right)}=Y_{0}+D_{\pm}\overrightarrow{\cal C}+R_{\pm}\overrightarrow{\cal R}},$
where 
\[
\begin{cases}
D_{\pm}\left(x,\mathbf{y}\right)=c\left(y_{1}y_{2}\right)+x^{N+2}D_{\pm}^{\left(N+2\right)}\left(x,\mathbf{y}\right)\\
R_{\pm}\left(x,\mathbf{y}\right)=x^{N+2}R_{\pm}^{\left(N+2\right)}\left(x,\mathbf{y}\right)
\end{cases}\qquad,
\]
with $N\in\ww N_{>0}$, $c\left(v\right)\in v\ww C\acc v$ of order
at least one, and $D_{\pm}^{\left(N+2\right)},R_{\pm}^{\left(N+2\right)}$
analytic in $S_{\pm}\left(r,\epsilon\right)\times\left(\ww C^{2},0\right)$
and dominated by $\norm{\mathbf{y}}_{\infty}$. Assume that $\Re\left(a_{1}+a_{2}\right)>0$.

Then, possibly by reducing $r>0$ and the neighborhood $\left(\ww C^{2},0\right)$,
there exist two germs of sectorial fibered diffeomorphisms $\varphi_{+}$
and $\varphi_{-}$ in $S_{+}\left(r,\epsilon\right)\times\left(\ww C^{2},0\right)$
and $S_{-}\left(r,\epsilon\right)\times\left(\ww C^{2},0\right)$
respectively, which conjugate $Y_{\pm}^{\left(N+2\right)}$ to 
\begin{eqnarray*}
Y_{\overrightarrow{\cal C},\pm} & := & Y_{0}+C_{\pm}\overrightarrow{\cal C}\qquad,
\end{eqnarray*}
where $C_{\pm}\left(x,\mathbf{y}\right)=D_{\pm}\circ\varphi_{\pm}^{-1}\left(x,\mathbf{z}\right)$.
Moreover we can chose $\varphi_{\pm}$ to be of the form 
\begin{eqnarray*}
\varphi_{\pm}\left(x,\mathbf{y}\right) & = & \left(x,y_{1}\exp\left(\rho_{\pm}\left(x,\mathbf{y}\right)\right),y_{2}\exp\left(\rho_{\pm}\left(x,\mathbf{y}\right)\right)\right)\qquad,
\end{eqnarray*}
where $\rho_{\pm}\left(x,\mathbf{y}\right)=x^{N+1}\tilde{\rho}_{\pm}\left(x,\mathbf{y}\right)$
and $\tilde{\rho}_{\pm}$ is analytic in $S_{\pm}\left(r,\epsilon\right)\times\left(\ww C^{2},0\right)$
and dominated by $\norm{\mathbf{y}}_{\infty}$.\end{prop}
\begin{rem}
Notice that $\varphi_{\pm}^{-1}$ is of the form 
\[
\varphi_{\pm}^{-1}\left(x,\mathbf{z}\right)=\left(x,z_{1}\left(1+x^{N+1}\vartheta\left(x,\mathbf{z}\right)\right),z_{2}\left(1+x^{N+1}\vartheta\left(x,\mathbf{z}\right)\right)\right)\qquad,
\]
where $\vartheta$ is analytic in $S_{\pm}\left(r,\epsilon\right)\times\left(\ww C^{2},0\right)$
and dominated by $\norm{\mathbf{z}}_{\infty}$. Consequently: 
\[
C_{\pm}\left(x,\mathbf{z}\right)=c\left(z_{1}z_{2}\right)+x^{N+1}C_{\pm}^{\left(N+1\right)}\left(x,\mathbf{z}\right)\qquad,
\]
where $c$ is the same as above and $C_{\pm}$ is analytic in $S_{\pm}\left(r,\epsilon\right)\times\left(\ww C^{2},0\right)$
and dominated by $\norm{\mathbf{z}}_{\infty}$.\end{rem}
\begin{prop}
\label{prop: tangential part}Let ${\displaystyle Y_{\cal C,\pm}:=Y_{0}+C_{\pm}\overrightarrow{\cal C}}$,
where 
\[
C_{\pm}\left(x,\mathbf{z}\right)=c\left(z_{1}z_{2}\right)+x^{N+1}C_{\pm}^{\left(N+1\right)}\left(x,\mathbf{z}\right)\qquad,
\]
with $N\in\ww N_{>0}$, $c\left(v\right)\in v\ww C\acc v$ of order
at least one, and $C_{\pm}^{\left(N+1\right)}$ analytic in $S_{\pm}\left(r,\epsilon\right)\times\left(\ww C^{2},0\right)$
and dominated by $\norm{\mathbf{z}}_{\infty}$ . Assume $\Re\left(a_{1}+a_{2}\right)>0$. 

Then, possibly by reducing $r>0$ and the neighborhood $\left(\ww C^{2},0\right)$,
there exist two germs of sectorial fibered diffeomorphisms $\psi_{+}$
and $\psi_{-}$ in $S_{+}\left(r,\epsilon\right)\times\left(\ww C^{2},0\right)$
and $S_{-}\left(r,\epsilon\right)\times\left(\ww C^{2},0\right)$
respectively, which conjugate $Y_{\cal C,\pm}$ to 
\begin{eqnarray*}
\ynorm & := & Y_{0}+c\left(v\right)\overrightarrow{\cal C}\qquad.
\end{eqnarray*}
Moreover, we can chose $\psi_{\pm}$ to be of the form 
\begin{eqnarray*}
\psi_{\pm}\left(x,\mathbf{z}\right) & = & \left(x,z_{1}\exp\left(-\chi_{\pm}\left(x,\mathbf{z}\right)\right),z_{2}\exp\left(\chi_{\pm}\left(x,\mathbf{z}\right)\right)\right)\qquad,
\end{eqnarray*}
where $\chi_{\pm}\left(x,\mathbf{z}\right)=x^{N}\tilde{\chi}_{\pm}\left(x,\mathbf{z}\right)$
and $\tilde{\chi}$ is analytic in $S_{\pm}\left(r,\epsilon\right)\times\left(\ww C^{2},0\right)$
and dominated by $\norm{\mathbf{z}}_{\infty}$.
\end{prop}
If we assume for a moment the two propositions above, the proof of
Proposition becomes obvious.
\begin{proof}[Proof of Proposition \ref{prop: norm sect}.]

It is an immediate consequence of the consecutive application of the
previous two propositions, just by taking $\Psi_{\pm}=\psi_{\pm}\circ\varphi_{\pm}$
with the notations above.
\end{proof}

\subsection{Proof of Propositions \ref{prop: radial part} and \ref{prop: tangential part}}

~

In order to prove Propositions \ref{prop: radial part} and \ref{prop: tangential part},
we will need the following lemmas. The first one gives the existence
of analytic solutions (in sectorial domains) to a homological equations
we need to solve.
\begin{lem}
\label{lem: solution eq homo}Let ${\displaystyle Z_{\pm}:=Y_{0}+C_{\pm}\left(x,\mathbf{y}\right)\overrightarrow{\cal C}+xR_{\pm}^{\left(1\right)}\left(x,\mathbf{y}\right)\overrightarrow{\cal R}}$,
with $C_{\pm},R_{\pm}^{\left(1\right)}$ analytic in $S_{\pm}\left(r,\epsilon\right)\times\left(\ww C^{2},0\right)$
and dominated by $\norm{\mathbf{y}}_{\infty}$ and let also $A_{\pm}\left(x,\mathbf{y}\right)$
be analytic in $S_{\pm}\left(r,\epsilon\right)\times\left(\ww C^{2},0\right)$
and dominated by $\norm{\mathbf{y}}_{\infty}$. Then for all $M\in\ww N_{>0}$,
possibly by reducing $r>0$ and the neighborhood $\left(\ww C^{2},0\right)$,
there exists a solution $\alpha_{\pm}$ to the homological equation
\begin{equation}
\cal L_{Z_{\pm}}\left(\alpha_{\pm}\right)=x^{M+1}A_{\pm}\left(x,\mathbf{y}\right)\qquad,\label{eq: eq homo lemme}
\end{equation}
such that $\alpha_{\pm}\left(x,\mathbf{y}\right)=x^{M}\tilde{\alpha}_{\pm}\left(x,\mathbf{y}\right)$,
where $\tilde{\alpha}_{\pm}$ is a germ of analytic function in $S_{\pm}\left(r,\epsilon\right)\times\left(\ww C^{2},0\right)$
and dominated by $\norm{\mathbf{y}}_{\infty}$.
\end{lem}
We will prove this lemma in subsection \ref{sub: proof lemma}. The
following lemma proves that $\varphi_{\pm}$ and $\psi_{\pm}$ constructed
in Propositions \ref{prop: radial part} and \ref{prop: tangential part}
are indeed germs of sectorial fibered diffeomorphisms in domains of
the form $S_{\pm}\left(r,\epsilon\right)\times\left(\ww C^{2},0\right)$.
\begin{lem}
\label{lem: diffeo sectoriel}Let $f_{\pm},g_{\pm}$ be two germs
of analytic and bounded functions in $S_{\pm}\left(r,\epsilon\right)\times\left(\ww C^{2},0\right)$,
which tend to $0$ as $\left(x,\mathbf{y}\right)\rightarrow\left(0,\mathbf{0}\right)$
in $S_{\pm}\left(r,\epsilon\right)\times\left(\ww C^{2},0\right)$.
Then 
\[
\phi_{\pm}:\left(x,\mathbf{y}\right)\mapsto\left(x,y_{1}\exp\left(f_{\pm}\left(x,\mathbf{y}\right)\right),y_{2}\exp\left(g_{\pm}\left(x,\mathbf{y}\right)\right)\right)
\]
 defines a germ of sectorial fibered diffeomorphism analytic in $S_{\pm}\left(r,\epsilon\right)\times\left(\ww C^{2},0\right)$
(possibly by reducing $r>0$ and the neighborhood $\left(\ww C^{2},0\right)$).
\end{lem}
Let us explain why these lemmas imply Propositions \ref{prop: radial part}
and \ref{prop: tangential part}.
\begin{proof}[Proof of both Propositions \ref{prop: radial part} and \ref{prop: tangential part}.]

It is sufficient to apply Lemma \ref{lem: solution eq homo} with
$M=N+1$ , $A_{\pm}=-R_{\pm}^{\left(N+2\right)}$, $\alpha_{\pm}=\rho_{\pm}$
and $Z_{\pm}=Y_{\pm}^{\left(N+2\right)}$ for Proposition \ref{prop: radial part},
and with $M=N$, $A_{\pm}=-C_{\pm}^{\left(N+1\right)}$, $\alpha_{\pm}=\chi_{\pm}$
and $Z_{\pm}=Y_{\overrightarrow{\cal C},\pm}$ for Proposition \ref{prop: tangential part}.
Then we use Lemma \ref{lem: diffeo sectoriel} to see that $\varphi_{\pm}$
and $\psi_{\pm}$ are germs of sectorial fibered diffeomorphisms on
the considered domains, and we finally check that they do the conjugacy
we want. With the notations above:
\begin{eqnarray*}
\mbox{D}\varphi_{\pm}\cdot Y_{\pm}^{\left(N+2\right)} & = & \left(\begin{array}{c}
\cal L_{Y_{\pm}^{\left(N+2\right)}}\left(x\right)\\
\cal L_{Y_{\pm}^{\left(N+2\right)}}\left(y_{1}\exp\left(\rho_{\pm}\left(x,\mathbf{y}\right)\right)\right)\\
\cal L_{Y_{\pm}^{\left(N+2\right)}}\left(y_{2}\exp\left(\rho_{\pm}\left(x,\mathbf{y}\right)\right)\right)
\end{array}\right)\\
 & = & \left(\begin{array}{c}
x^{2}\\
\left(\cal L_{Y_{\pm}^{\left(N+2\right)}}\left(y_{1}\right)+y_{1}\left(\cal L_{Y_{\pm}^{\left(N+2\right)}}\left(\rho_{\pm}\right)\right)\right)\exp\left(\rho_{\pm}\left(x,\mathbf{y}\right)\right)\\
\left(\cal L_{Y_{\pm}^{\left(N+2\right)}}\left(y_{2}\right)+y_{2}\left(\cal L_{Y_{\pm}^{\left(N+2\right)}}\left(\rho_{\pm}\right)\right)\right)\exp\left(\rho_{\pm}\left(x,\mathbf{y}\right)\right)
\end{array}\right)\\
 & = & \left(\begin{array}{c}
x^{2}\\
\left(-\lambda+a_{1}x-D_{\pm}\left(x,\mathbf{y}\right)\right)y_{1}\exp\left(\rho_{\pm}\left(x,\mathbf{y}\right)\right)\\
\left(\lambda+a_{2}x+D_{\pm}\left(x,\mathbf{y}\right)\right)y_{2}\exp\left(\rho_{\pm}\left(x,\mathbf{y}\right)\right)
\end{array}\right)\\
 &  & \big(\mbox{we have used }\cal L_{Y_{\pm}^{\left(N+2\right)}}\left(\rho_{\pm}\right)=-x^{N+2}R_{\pm}^{\left(N+2\right)}\big)\\
 & = & \left(Y_{0}+C_{\pm}\overrightarrow{\cal C}\right)\circ\varphi_{\pm}\left(x,\mathbf{y}\right)\\
 & = & Y_{\overrightarrow{\cal C},\pm}\circ\varphi_{\pm}\left(x,\mathbf{y}\right)\hfill,
\end{eqnarray*}
so that $\left(\varphi_{\pm}\right)_{*}\left(Y_{\pm}^{\left(N+2\right)}\right)=Y_{\overrightarrow{\cal C},\pm}$
and then
\begin{eqnarray*}
\mbox{D}\psi_{\pm}\cdot Y_{\overrightarrow{\cal C},\pm} & = & \left(\begin{array}{c}
\cal L_{Y_{\overrightarrow{\cal C},\pm}}\left(x\right)\\
\cal L_{Y_{\overrightarrow{\cal C},\pm}}\left(z_{1}\exp\left(-\chi\left(x,\mathbf{z}\right)\right)\right)\\
\cal L_{Y_{\overrightarrow{\cal C},\pm}}\left(z_{2}\exp\left(\chi\left(x,\mathbf{z}\right)\right)\right)
\end{array}\right)\\
 & = & \left(\begin{array}{c}
x^{2}\\
\left(\cal L_{Y_{\overrightarrow{\cal C},\pm}}\left(z_{1}\right)+z_{1}\left(\cal L_{Y_{\overrightarrow{\cal C},\pm}}\left(\chi\right)\right)\right)\exp\left(-\chi\left(x,\mathbf{z}\right)\right)\\
\left(\cal L_{Y_{\overrightarrow{\cal C},\pm}}\left(z_{2}\right)+z_{2}\left(\cal L_{Y_{\overrightarrow{\cal C},\pm}}\left(\chi\right)\right)\right)\exp\left(\chi\left(x,\mathbf{z}\right)\right)
\end{array}\right)\\
 & = & \left(\begin{array}{c}
x^{2}\\
\left(-\lambda+a_{1}x-c\left(z_{1}z_{2}\right)\right)z_{1}\exp\left(-\chi\left(x,\mathbf{z}\right)\right)\\
\left(\lambda+a_{2}x+c\left(z_{1}z_{2}\right)\right)z_{2}\exp\left(\chi\left(x,\mathbf{y}\right)\right)
\end{array}\right)\\
 &  & \big(\mbox{we have used }\cal L_{Y_{\overrightarrow{\cal C},\pm}}\left(\chi_{\pm}\right)=-x^{N+1}C_{\pm}^{\left(N+1\right)}\big)\\
 & = & \left(Y_{0}+c\left(u\right)\overrightarrow{\cal C}\right)\circ\psi_{\pm}\left(x,\mathbf{z}\right)\\
 & = & \ynorm\circ\psi_{\pm}\left(x,\mathbf{z}\right)\hfill,
\end{eqnarray*}
so that $\left(\psi_{\pm}\right)_{*}\left(Y_{\overrightarrow{\cal C},\pm}\right)=\ynorm$
.
\end{proof}

\subsection{Proof of Lemma \ref{lem: diffeo sectoriel}}

~
\begin{proof}[\emph{Proof of Lemma \ref{lem: diffeo sectoriel}.}]

We consider two germs of analytic functions $f_{\pm},g_{\pm}$ in
$S_{\pm}\left(r,\epsilon\right)\times\left(\ww C^{2},0\right)$ which
which tend to $0$ as $\left(x,\mathbf{y}\right)\rightarrow\left(0,\mathbf{0}\right)$
in $S_{\pm}\left(r,\epsilon\right)\times\left(\ww C^{2},0\right)$,
and we define 
\[
\phi_{\pm}:\left(x,\mathbf{y}\right)\mapsto\left(x,y_{1}\exp\left(f_{\pm}\left(x,\mathbf{y}\right)\right),y_{2}\exp\left(g_{\pm}\left(x,\mathbf{y}\right)\right)\right)\,\,.
\]
Let us first prove that $\phi_{\pm}$ is into. Let $\mathbf{x}=\left(x,y_{1},y_{2}\right)$
and $\mathbf{x}'=\left(x',y'_{1},y'_{2}\right)$ in $S_{\pm}\left(r,\epsilon\right)\times\left(\ww C^{2},0\right)$
such that $\phi_{\pm}\left(\mathbf{x}\right)=\phi_{\pm}\left(\mathbf{x}'\right)$.
Since $\phi_{\pm}$ is fibered, necessarily $x=x'$. Then assume that
$\left(y_{1},y_{2}\right)\neq\left(y'_{1},y'_{2}\right)$, such that
\[
\left\Vert \left(y_{1}-y'_{1},y_{2}-y'_{2}\right)\right\Vert _{\infty}>0
\]
 and for instance $\left\Vert \left(y_{1}-y'_{1},y_{2}-y'_{2}\right)\right\Vert _{\infty}=\abs{y_{1}-y'_{1}}>0$
(the other case can be done similarly). We denote by $D_{\mathbf{y}}$
the derivative with respect to variables $\left(y_{1},y_{2}\right)$.
According to the mean value theorem: 
\begin{eqnarray*}
\abs{\frac{e^{f_{\pm}\left(\mathbf{x}\right)}-e^{f_{\pm}\left(\mathbf{x'}\right)}}{y_{1}-y'_{1}}} & \leq & \underset{\left(z_{1},z_{2}\right)\in[\left(y_{1},y_{2}\right),\left(y'_{1},y'_{2}\right)]}{\mbox{sup}}\left\Vert D_{\mathbf{y}}\left(e^{f\pm}\right)\left(x,z_{1},z_{2}\right)\right\Vert _{\infty}.
\end{eqnarray*}

Consequently we have: 
\begin{eqnarray*}
0 & = & \abs{y_{1}e^{f_{\pm}\left(\mathbf{x}\right)}-y'_{1}e^{f_{\pm}\left(\mathbf{x'}\right)}}\\
 & = & \abs{e^{f_{\pm}\left(\mathbf{x}\right)}}.\abs{y_{1}-y'_{1}}.\abs{1+\frac{y'_{1}}{e^{f_{\pm}\left(\mathbf{x}\right)}}.\frac{e^{f_{\pm}\left(\mathbf{x}\right)}-e^{f_{\pm}\left(\mathbf{x'}\right)}}{y_{1}-y'_{1}}}\\
 & \geq & \abs{e^{f_{\pm}\left(\mathbf{x}\right)}}.\abs{y_{1}-y'_{1}}.\left(1-\abs{\frac{y'_{1}}{e^{f_{\pm}\left(\mathbf{x}\right)}}}.\abs{\frac{e^{f_{\pm}\left(\mathbf{x}\right)}-e^{f_{\pm}\left(\mathbf{x'}\right)}}{y_{1}-y'_{1}}}\right)\\
 & \geq & \abs{e^{f_{\pm}\left(\mathbf{x}\right)}}.\abs{y_{1}-y'_{1}}.\left(1-\abs{\frac{y'_{1}}{e^{f_{\pm}\left(\mathbf{x}\right)}}}.\underset{\left(z_{1},z_{2}\right)\in[\left(y_{1},y_{2}\right),\left(y'_{1},y'_{2}\right)]}{\mbox{sup}}\left\Vert D_{\mathbf{y}}\left(e^{f_{\pm}}\right)\left(x,z_{1},z_{2}\right)\right\Vert _{\infty}.\right)
\end{eqnarray*}

Assume that we chose $\left(\ww C^{2},0\right)=\mathbf{D\left(0,r\right)}$
small enough such that $f_{\pm}$ is analytic in 
\[
S_{\pm}\left(r,\epsilon\right)\times D\left(0,3r{}_{1}+\delta\right)\times D\left(0,3r{}_{2}+\delta\right)
\]
with $\delta>0$ small. Without lost of generality we can take $r_{1}=r_{2}$.
We apply Cauchy \textquoteright s integral formula to $z_{1}\mapsto e^{f_{\pm}\left(x,z_{1},z_{2}\right)}$,
for all fixed $z_{2}$ , integrating on the circle of center $0$
and radius $3r_{1}=3r_{2}$. Similarly we also apply Cauchy \textquoteright s
integral formula to $z_{2}\mapsto e^{f_{\pm}\left(x,z_{1},z_{2}\right)}$,
for all fixed $z_{1}$, integrating on the circle of center $0$ and
radius $3r_{2}=3r_{1}$. Then we obtain 
\[
\underset{\left(z_{1},z_{2}\right)\in[\left(y_{1},y_{2}\right),\left(y'_{1},y'_{2}\right)]}{\mbox{sup}}\left\Vert D_{\mathbf{y}}\left(e^{f_{\pm}}\right)\left(x,z_{1},z_{2}\right)\right\Vert _{\infty}\leq\frac{3}{4r_{1}}.\exp\left(\underset{\mathbf{x}\in S_{\pm}\left(r,\epsilon\right)\times\mathbf{D\left(0,r\right)}}{\mbox{sup}}\left(\abs{f_{\pm}\left(\mathbf{x}\right)}\right)\right)\quad,
\]

such that: 
\begin{eqnarray*}
0 & = & \abs{y_{1}e^{f_{\pm}\left(\mathbf{x}\right)}-y'_{1}e^{f_{\pm}\left(\mathbf{x'}\right)}}\\
 & \geq & \abs{e^{f_{\pm}\left(\mathbf{x}\right)}}.\abs{y_{1}-y'_{1}}.\left(1-\frac{3}{4}\exp\left(\underset{\mathbf{x}\in S_{\pm}\left(r,\epsilon\right)\times\mathbf{D\left(0,r\right)}}{\mbox{sup}}\left(2\abs{f_{\pm}\left(\mathbf{x}\right)}\right)\right)\right)\qquad.
\end{eqnarray*}

Since $f_{\pm}\left(\mathbf{x}\right)\underset{\mathbf{x}\rightarrow\mathbf{0}}{\rightarrow}0$,
we can choose $r,r_{1}$ and $r_{2}$ small enough such that: 
\[
\exp\left(\underset{\mathbf{x}\in S_{\pm}\left(r,\epsilon\right)\times\mathbf{D\left(0,r\right)}}{\mbox{sup}}\left(2\abs{f_{\pm}\left(\mathbf{x}\right)}\right)\right)\leq\frac{5}{4}<\frac{4}{3}\qquad.
\]

Finally we obtain: 
\begin{eqnarray*}
0 & = & \abs{y_{1}e^{f_{\pm}\left(\mathbf{x}\right)}-y'_{1}e^{f_{\pm}\left(\mathbf{x'}\right)}}\\
 & \geq & \abs{e^{f_{\pm}\left(\mathbf{x}\right)}}\frac{\abs{y_{1}-y'_{1}}}{16}>0\qquad,
\end{eqnarray*}

and so, if $y_{1}\neq y'_{1}$, $0=\abs{y_{1}e^{\rho\left(\mathbf{x}\right)}-y'_{1}e^{\rho\left(\mathbf{x'}\right)}}>0$,
which is a contradiction.

Conclusion: $\left(y_{1},y_{2}\right)=\left(y'_{1},y'_{2}\right)$
and then $\phi_{\pm}$ is into in $S_{\pm}\left(r,\epsilon\right)\times\left(\ww C^{2},0\right)$.

Since $\phi_{\pm}$ is into and analytic in $S_{\pm}\left(r,\epsilon\right)\times\left(\ww C^{2},0\right)$,
it is a biholomorphism between $S_{\pm}\left(r,\epsilon\right)\times\left(\ww C^{2},0\right)$
and its image which is necessarily open (an analytic function is open),
and of the same form.
\end{proof}

\subsection{\label{sub: proof lemma}Resolution of the homological equation:
proof of Lemma \ref{lem: solution eq homo}}

~

The goal of this subsection is to prove Lemma \ref{lem: solution eq homo}
by studying the existence of paths asymptotic to the singularity and
tangent to the foliation, and then to use them to construct the solution
to the homological equation $\left(\mbox{\ref{eq: eq homo lemme}}\right)$.

\emph{For convenience and without lost of generality we assume $\lambda=1$}
\emph{during this subsection} $\Big($otherwise we can divide our
vector field by $\lambda\neq0$, make $x\mapsto\lambda x$ and finally
consider $\exp\left(-i\arg\left(\lambda\right)\right).S_{\pm}\left(r,\epsilon\right)$
instead of $S_{\pm}\left(r,\epsilon\right)$: these modifications
do not change $a_{1}$ and $a_{2}$,$\Big)$.

\subsubsection{\label{sub:Domain-of-stability}Domain of stability and asymptotic
paths}

~

We consider 
\begin{eqnarray*}
Z_{\pm} & = & Y_{0}+C_{\pm}\left(x,\mathbf{y}\right)\overrightarrow{\cal C}+xR_{\pm}^{\left(1\right)}\left(x,\mathbf{y}\right)\overrightarrow{\cal R}\\
 & = & \left(\begin{array}{c}
x^{2}\\
y_{1}\left(-\left(1+C_{\pm}\left(x,\mathbf{y}\right)\right)+a_{1}x+xR_{\pm}^{\left(1\right)}\left(x,\mathbf{y}\right)\right)\\
y_{2}\left(1+C_{\pm}\left(x,\mathbf{y}\right)+a_{2}x+xR_{\pm}^{\left(1\right)}\left(x,\mathbf{y}\right)\right)
\end{array}\right)
\end{eqnarray*}
with $\Re\left(a_{1}+a_{2}\right)>0$, and $C_{\pm},R_{\pm}^{\left(1\right)}$
analytic in $S_{\pm}\left(r,\epsilon\right)\times\mathbf{D\left(0,r\right)}$
and dominated by $\norm{\mathbf{y}}_{\infty}$. More precisely, we
consider the Cauchy problem of unknown $\mathbf{x}\left(t\right):=\left(x\left(t\right),y_{1}\left(t\right),y_{2}\left(t\right)\right)$,
with real and increasing time $t\geq0$, associated to 
\[
X_{\pm}:=\frac{\pm i}{1+\left(\frac{a_{2}-a_{1}}{2}\right)x+C_{\pm}}Z_{\pm}\,\,,
\]
\emph{i.e.} 
\begin{equation}
\begin{cases}
\ddd xt=\frac{\pm ix^{2}}{1+\left(\frac{a_{2}-a_{1}}{2}\right)x+C_{\pm}}\\
\ddd{y_{1}}t=\frac{\pm iy_{1}}{1+\left(\frac{a_{2}-a_{1}}{2}\right)x+C_{\pm}}\left(-\left(1+C_{\pm}\left(x,\mathbf{y}\right)\right)+a_{1}x+xR_{\pm}^{\left(1\right)}\left(x,\mathbf{y}\right)\right)\\
\ddd{y_{2}}t=\frac{\pm iy_{2}}{1+\left(\frac{a_{2}-a_{1}}{2}\right)x+C_{\pm}}\left(1+C_{\pm}\left(x,\mathbf{y}\right)+a_{2}x+xR_{\pm}^{\left(1\right)}\left(x,\mathbf{y}\right)\right)\\
\mathbf{x}\left(t\right)=\mathbf{x}_{0}=\left(x_{0},y_{1,0},y_{2,0}\right)\in S_{\pm}\left(r,\epsilon\right)\times\times\mathbf{D\left(0,r\right)} & .
\end{cases}\label{eq: Probleme de cauchy}
\end{equation}
We denote by $\left(t,\mathbf{x}_{0}\right)\mapsto\Phi_{X_{\pm}}^{t}\left(\mathbf{x}_{0}\right)$
the flow of $X_{\pm}$ with increasing time $t\geq0$ and with initial
point $\mathbf{x}_{0}$: $\Phi_{X\pm}^{0}\left(\mathbf{x}_{0}\right)=\mathbf{x}_{0}$.

We will prove the following:
\begin{prop}
\label{prop: domaine stable}For all ${\displaystyle \epsilon\in\left]0,\frac{\pi}{2}\right[}$,
there exists finite sectors $S_{\pm}\left(r,\epsilon\right),S_{\pm}\left(r',\epsilon\right)$
with $r,r'>0$ and an open domain $\Omega_{\pm}$ stable by the flow
of $\left(\mbox{\ref{eq: Probleme de cauchy}}\right)$ with increasing
time $t\geq0$ such that 
\[
S_{\pm}\left(r',\epsilon\right)\times\mathbf{D\left(0,r'\right)}\subset\Omega_{\pm}\subset S_{\pm}\left(r,\epsilon\right)\times\mathbf{D\left(0,r\right)}\qquad,
\]
(\emph{cf. }figure \ref{fig: domaine stable}). Moreover, if $\mathbf{x}_{0}\in\Omega_{\pm}$
then the corresponding solution of $\left(\mbox{\ref{eq: Probleme de cauchy}}\right)$,
namely $\mathbf{x}\left(t\right):=\Phi_{X_{\pm}}^{t}\left(\mathbf{x}_{0}\right)$
exists for all $t\geq0$ and $\mathbf{x}\left(t\right)\rightarrow\mathbf{0}$
as $t\rightarrow+\infty$.
\end{prop}
\begin{figure}
\includegraphics[bb=0bp 0bp 750bp 440bp,scale=0.55]{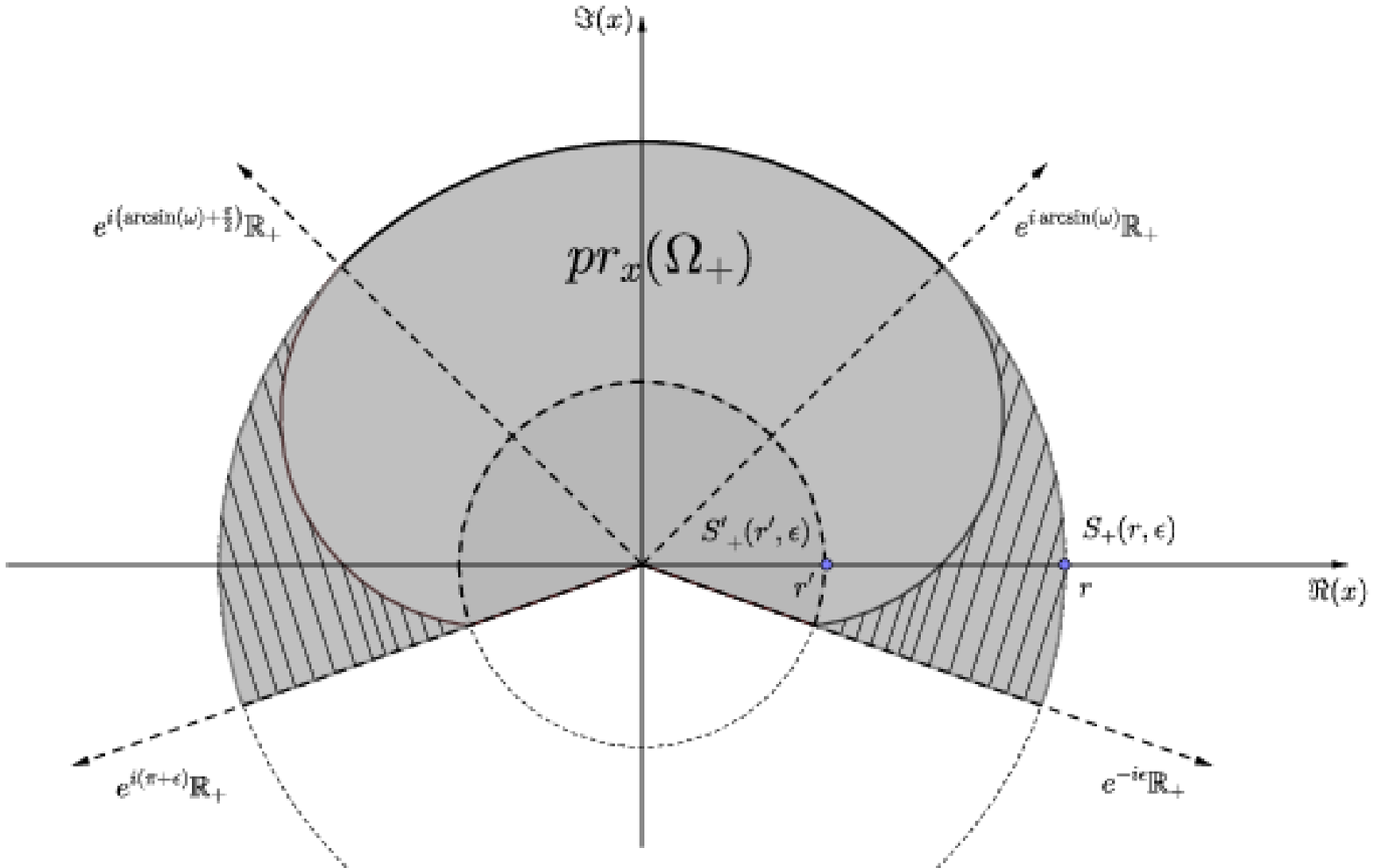}\protect\caption{\selectlanguage{french}%
\label{fig: domaine stable}\foreignlanguage{english}{Representation
of the projection $pr_{x}\left(\Omega_{+}\right)$ of the stable domain
$\Omega_{+}$ in the $x$-space. }\selectlanguage{english}%
}
\end{figure}

\begin{rem}
This will prove that the solution $\mathbf{x}\left(t\right)$ to $\left(\mbox{\ref{eq: Probleme de cauchy}}\right)$
exists for all $t\geq0$ and tends to the origin: it defines a path
tangent to the foliation and asymptotic to the origin. Moreover, notice
that the domain $\Omega_{\pm}$ depends on the choice of $r$ and
$r'>0$.\end{rem}
\begin{defn}
\label{def: chemin asympto}We define the \emph{asymptotic path} with
base point $\mathbf{x}_{0}\in\Omega_{\pm}$ associated to $X_{\pm}$
the path $\gamma_{\pm,\mathbf{x}_{0}}:=\acc{\Phi_{X_{\pm}}^{t}\left(\mathbf{x}_{0}\right),\, t\geq0}$. 
\end{defn}
\emph{For convenience and without lost of generality we }only detail
the case where ``$\pm=+$'' (the case where ``$\pm=-$'' is totally
similar).

If we write $a:=a_{1}+a_{2}$ and $b:=\frac{a_{2}-a_{1}}{2}$, in
the case ``$\pm=+$'' we have:
\[
\begin{cases}
\ddd xt=\frac{ix^{2}}{1+bx+C_{+}}\\
\ddd{y_{1}}t=iy_{1}\left(-1+\left(\frac{\frac{a}{2}+R_{+}^{\left(1\right)}\left(x,\mathbf{y}\right)}{1+bx+C_{+}\left(x,\mathbf{y}\right)}\right)x\right)\\
\ddd{y_{2}}t=iy_{2}\left(1+\left(\frac{\frac{a}{2}+R_{+}^{\left(1\right)}\left(x,\mathbf{y}\right)}{1+bx+C_{+}\left(x,\mathbf{y}\right)}\right)x\right)\\
\mathbf{x}\left(t\right)=\mathbf{x}_{0}=\left(x_{0},y_{1,0},y_{2,0}\right)\in S_{+}\left(r,\epsilon\right)\times\mathbf{D\left(0,r\right)}.
\end{cases}
\]
We also consider the differential equations satisfied by $\abs{x\left(t\right)}$,
$\abs{y_{1}\left(t\right)}$, $\abs{y_{2}\left(t\right)}$ and $\theta\left(t\right):=\arg\left(x\left(t\right)\right)$:
\[
\begin{cases}
\ddd{\abs{x\left(t\right)}}t=\abs{x\left(t\right)}\Re\left(\frac{ix\left(t\right)}{1+bx\left(t\right)+C_{+}\left(\mathbf{x}\left(t\right)\right)}\right)\\
\ddd{\abs{y_{1}\left(t\right)}}t=\abs{y_{1}\left(t\right)}\Re\left(ix\left(t\right)\left(\frac{\frac{a}{2}+R_{+}^{\left(1\right)}\left(\mathbf{x}\left(t\right)\right)}{1+bx\left(t\right)+C_{+}\left(\mathbf{x}\left(t\right)\right)}\right)\right)\\
\ddd{\abs{y_{2}\left(t\right)}}t=\abs{y_{2}\left(t\right)}\Re\left(ix\left(t\right)\left(\frac{\frac{a}{2}+R_{+}^{\left(1\right)}\left(\mathbf{x}\left(t\right)\right)}{1+bx\left(t\right)+C_{+}\left(\mathbf{x}\left(t\right)\right)}\right)\right)\\
\ddd{\theta\left(t\right)}t=\Im\left(\frac{ix\left(t\right)}{1+bx\left(t\right)+C_{+}\left(\mathbf{x}\left(t\right)\right)}\right).
\end{cases}
\]

\medskip{}
For any non-zero complex number $\zeta$ and positive numbers $R,B>0$,
we denote by $\Sigma_{+}\left(\zeta,R,B\right)$ the sector of radius
$R$ bisected by $i\bar{\zeta}\ww R_{+}$ and of opening $\pi-2\arcsin\left(B\right)=2\arccos\left(B\right)$:
\begin{eqnarray*}
\Sigma_{+}\left(\zeta,R,B\right) & := & \acc{x\in D\left(0,R\right)\mid\Im\left(\zeta x\right)>B\abs{\zeta x}}\\
 & = & \acc{x\in D\left(0,R\right)\mid-\arccos\left(B\right)<\arg\left(x\right)-\arg\left(i\bar{\zeta}\right)<\arccos\left(B\right)}\,\,.
\end{eqnarray*}
For $T,R>0$, we denote by $\Theta_{+}\left(R,T\right)$ $\big($\emph{resp.}
$\Theta_{-}\left(R,T\right)$$\big)$ the sector of radius $R$ bisected
by $\ww R_{+}$ $\big($\emph{resp.} $\ww R_{-}$$\big)$ and of opening
$2\arccos\left(T\right)$:
\begin{eqnarray*}
\Theta_{+}\left(R,T\right) & := & \acc{x\in D\left(0,R\right)\mid\Re\left(x\right)>T\abs x}\\
 & = & \acc{x\in D\left(0,R\right)\mid-\arccos\left(T\right)<\arg\left(x\right)<\arccos\left(T\right)}\\
\Theta_{-}\left(R,T\right) & := & \acc{x\in D\left(0,R\right)\mid\Re\left(x\right)<-T\abs x}\\
 & = & \acc{x\in D\left(0,R\right)\mid-\arccos\left(T\right)<\arg\left(x\right)-\pi<\arccos\left(T\right)}\,\,
\end{eqnarray*}
Since $\Re\left(a\right)>0$ by assumption, we can choose $\omega'\in\left]0,\frac{\Re\left(a\right)}{\abs a}\right[$,
such that $\Sigma_{+}\left(a,r,\omega'\right)$ contains $i\ww R_{>0}$.
Indeed, we have 
\[
\abs{\arg\left(i\right)-\arg\left(i\overline{a}\right)}=\abs{\arg\left(a\right)}<\arccos\left(\omega'\right)\,\,.
\]
In particular, we have: 
\[
0<\arccos\left(\omega'\right)-\left|\arg\left(a\right)\right|<\frac{\pi}{2}
\]
so that 
\[
0<\cos\left(\arccos\left(\omega'\right)-\left|\arg\left(a\right)\right|\right)<1\,\,.
\]
Hence we take $\omega>0$ such that 
\begin{equation}
\omega\in\left]\cos\left(\arccos\left(\omega'\right)-\left|\arg\left(a\right)\right|\right),1\right[\,\,,\label{eq: choix de w}
\end{equation}
and then ${\displaystyle \Sigma_{+}\left(1,r,\omega\right)\subset\Sigma_{+}\left(a,r,\omega'\right)}$.
Indeed, if $x\in\Sigma_{+}\left(1,r,\omega\right)$, then:
\begin{equation}
-\arccos\left(\omega\right)<\arg\left(x\right)-\frac{\pi}{2}<\arccos\left(\omega\right)\,\,,\label{eq: ineg arg et arccos1}
\end{equation}
and therefore
\begin{eqnarray*}
\abs{\arg\left(x\right)-\arg\left(i.\overline{a}\right)} & < & \arccos\left(\omega\right)+\abs{\arg\left(a\right)}\\
 &  & (\mbox{by \eqref{eq: ineg arg et arccos1}})\\
 & < & \arccos\left(\omega'\right)\\
 &  & (\mbox{by \eqref{eq: choix de w}}).
\end{eqnarray*}
Finally, we fix $\mu\in\left]0,\sqrt{1-\omega^{2}}\right[$ small
enough such that 
\begin{eqnarray*}
\Theta_{+}\left(r,\mu\right)\cap\Sigma_{+}\left(1,r,\omega\right) & \neq & \emptyset\\
\Theta_{-}\left(r,\mu\right)\cap\Sigma_{+}\left(1,r,\omega\right) & \neq & \emptyset
\end{eqnarray*}
and 
\[
{\displaystyle S_{+}\left(r,\epsilon\right)\subset\Sigma_{+}\left(1,r,\omega\right)\cup\Theta_{+}\left(r,\mu\right)\cup\Theta_{-}\left(r,\mu\right)\,\,.}
\]
More precisely, we must have $0<\epsilon<\arccos\left(\mu\right)$.
The idea is now to study the behavior of $t\mapsto x\left(t\right)$
(where $t\mapsto\mathbf{x}\left(t\right)=\left(x\left(t\right),y_{1}\left(t\right),y_{2}\left(t\right)\right)$
is the solution of (\ref{eq: Probleme de cauchy})) over each domains
$\Sigma_{+}\left(1,r,\omega\right),\Theta_{+}\left(r,\mu\right),\Theta_{-}\left(r,\mu\right)$
(\emph{cf. }figure \ref{fig:Repr=0000E9sentation-des-domaines})\emph{.}
\begin{figure}
\includegraphics[scale=0.55]{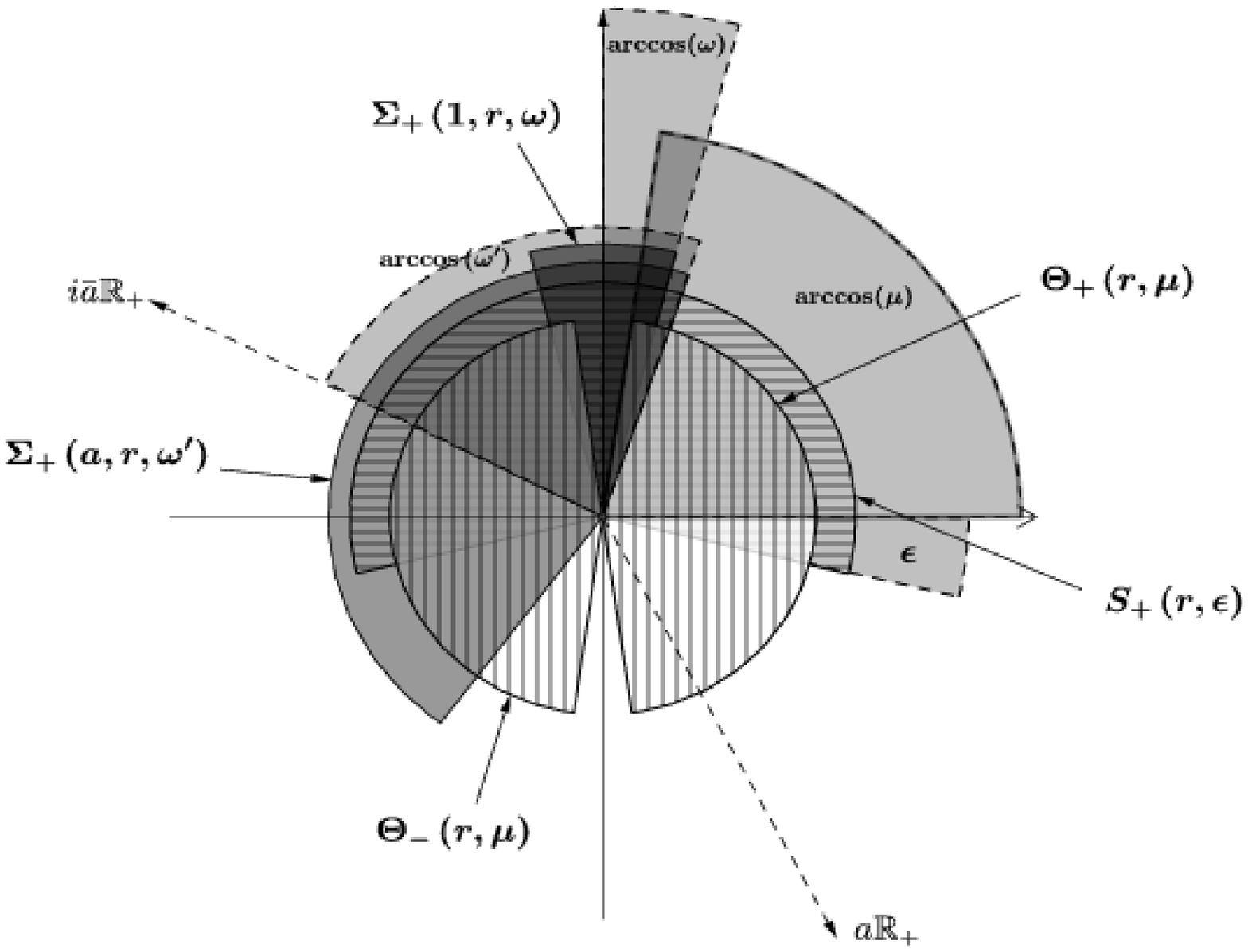}\protect\caption{\label{fig:Repr=0000E9sentation-des-domaines}Representation of domains
$\Sigma_{+}\left(1,r,\omega\right),\Sigma_{+}\left(a,r,\omega'\right),\Theta_{+}\left(r,\mu\right),\Theta_{-}\left(r,\mu\right),S_{+}\left(r,\epsilon\right)$
(with modified radii for more clarity).}
\end{figure}

We can now prove the following result, which is a precision of Proposition
\ref{prop: domaine stable}.

~
\begin{lem}
~
\begin{enumerate}
\item There exists $r,r_{1},r_{2}>0$ such that $\Sigma_{+}\left(1,r,\omega\right)\times\mathbf{D\left(0,r\right)}$
is stable by the flow of $\left(\mbox{\ref{eq: Probleme de cauchy}}\right)$
with increasing time $t\geq0$. Moreover in this region $\abs{x\left(t\right)}$,
$\abs{y_{1}\left(t\right)}$ and $\abs{y_{2}\left(t\right)}$ decrease
and go to $0$ as $t\rightarrow+\infty$.
\item There exists $0<r'<r$, $0<r'_{1}<r_{1}$, $0<r'_{2}<r_{2}$ and an
open domain $\Omega_{+}$ stable under the action flow of $\left(\mbox{\ref{eq: Probleme de cauchy}}\right)$
with increasing time $t\geq0$ such that 
\[
S_{+}\left(r',\epsilon\right)\times\mathbf{D\left(0,r'\right)}\subset\Omega_{+}\subset S_{+}\left(r,\epsilon\right)\times\mathbf{D\left(0,r\right)}\,\,.
\]
Moreover, if $x_{0}\in\Theta_{+}\left(r',\mu\right)$ $\Big($\emph{resp.}
$x_{0}\in\Theta_{-}\left(r',\mu\right)$$\Big)$ , then $\theta\left(t\right)=\arg\left(x\left(t\right)\right),t\geq0$
is increasing (\emph{resp.} decreasing) as long as $x\left(t\right)$
remains in $\Theta_{+}\left(r',\mu\right)$ $\Big($\emph{resp.} $\Theta_{-}\left(r',\mu\right)$$\Big)$.
Finally, there exists $t_{0}\geq0$ such that for all $t\geq t_{0}$,
$x\left(t\right)\in\Sigma_{+}\left(1,r,\omega\right)$.
\end{enumerate}
\end{lem}
\begin{proof}
We fix $\delta\in\left]0,\min\left(\omega,\mu\right)\right[$, $\delta'\in\left]0,\omega'\right[$
and we take $r>0$ small enough such that for all $\mathbf{x}=\left(x,\mathbf{y}\right)\in S_{+}\left(r,\epsilon\right)\times\mathbf{D\left(0,r\right)}$,
we have
\[
\begin{cases}
\abs{\frac{1}{1+bx+C_{+}\left(\mathbf{x}\right)}-1}<\delta\\
\abs{\frac{\frac{a}{2}+R_{+}^{\left(1\right)}\left(\mathbf{x}\right)}{1+bx+C_{+}\left(\mathbf{x}\right)}-\frac{a}{2}}<\delta' & \,\,.
\end{cases}
\]
Consequently for all $\mathbf{x}\in S_{+}\times\mathbf{D\left(0,r\right)}$
we have the following estimations:
\[
\begin{cases}
-\abs x\left(1+\delta\right)<\Re\left(\frac{ix}{1+bx+C_{+}\left(\mathbf{x}\right)}\right)<\abs x\left(1+\delta\right)\\
-\abs x\left(\abs{\frac{a}{2}}+\delta'\right)<\Re\left(ix\left(\frac{\frac{a}{2}+R_{+}^{\left(1\right)}\left(x,\mathbf{y}\right)}{1+bx+C_{+}\left(x,\mathbf{y}\right)}\right)\right)<\abs x\left(\abs{\frac{a}{2}}+\delta'\right) & \,\,.
\end{cases}
\]
Moreover:
\begin{itemize}
\item if $x\in\Sigma_{+}\left(1,r,\omega\right)$ then 
\begin{eqnarray*}
\Re\left(\frac{ix}{1+bx+C_{+}\left(\mathbf{x}\right)}\right) & < & -\abs x\left(\omega-\delta\right)\qquad;
\end{eqnarray*}

\item if $x\in\Sigma_{+}\left(a,r,\omega'\right)$ $\Big($in particular
if $x\in\Sigma_{+}\left(1,r,\omega\right)$$\Big)$ then
\begin{eqnarray*}
\Re\left(ix\left(\frac{\frac{a}{2}+R_{+}^{\left(1\right)}\left(x,\mathbf{y}\right)}{1+bx+C_{+}\left(x,\mathbf{y}\right)}\right)\right) & < & -\abs x\left(\omega'-\delta'\right)\qquad;
\end{eqnarray*}

\item if $x\in\Theta_{-}\left(r,\mu\right)$ $\Big($\emph{resp.} $\Theta_{+}\left(r,\mu\right)$$\Big)$
then 
\begin{eqnarray*}
\Im\left(\frac{ix}{1+bx+C_{+}\left(\mathbf{x}\right)}\right) & < & -\abs x\left(\mu-\delta\right)\\
\Bigg(\mbox{\emph{resp. }}\Im\left(\frac{ix}{1+bx+C_{+}\left(\mathbf{x}\right)}\right) & > & \abs x\left(\mu-\delta\right)\Bigg)\qquad.
\end{eqnarray*}

\end{itemize}
Hence:
\begin{itemize}
\item for all $t\geq0$
\begin{eqnarray*}
-\left(1+\delta\right)\abs{x\left(t\right)}^{2} & <\ddd{\abs{x\left(t\right)}}t< & -\left(1+\delta\right)\abs{x\left(t\right)}^{2}
\end{eqnarray*}
and then, as long as $\mathbf{x}\left(t\right)\in S_{+}\left(r,\epsilon\right)\times\mathbf{D\left(0,r\right)}$,
we have 
\[
\abs{x\left(t\right)}>\frac{\abs{x_{0}}}{1+\left(1+\delta\right)\abs{x_{0}}t}\qquad;
\]

\item for all $t\geq0$, if $x\left(t\right)\in\Sigma_{+}\left(1,r,\omega\right)$,
then 
\begin{eqnarray}
\ddd{\abs{x\left(t\right)}}t & < & -\left(\omega-\delta\right)\abs{x\left(t\right)}^{2}\label{eq: majoration derivee de x}
\end{eqnarray}
and 
\begin{equation}
\begin{cases}
\ddd{\abs{y_{1}\left(t\right)}}t<-\left(\omega'-\delta'\right)\abs{y_{1}\left(t\right)}\abs{x\left(t\right)}\\
\ddd{\abs{y_{2}\left(t\right)}}t<-\left(\omega'-\delta'\right)\abs{y_{2}\left(t\right)}\abs{x\left(t\right)}
\end{cases}\label{eq: majoration derivee de y}
\end{equation}
so that $\abs{x\left(t\right)},\abs{y_{1}\left(t\right)}\mbox{ and }\abs{y_{2}\left(t\right)}$
are decreasing as long as $x\left(t\right)\in\Sigma_{+}\left(1,r,\omega\right)$;
\item for all $t\geq0$, if $x\left(t\right)\in\Theta_{-}\left(r,\mu\right)$
$\Big($\emph{resp.} $\Theta_{+}\left(r,\mu\right)$$\Big)$ then
\begin{eqnarray*}
\ddd{\theta}t\left(t\right) & < & -\left(\mu-\delta\right)\abs{x\left(t\right)}<\frac{-\left(\mu-\delta\right)\abs{x_{0}}}{1+\left(1+\delta\right)\abs{x_{0}}t}\\
\Bigg(\mbox{\emph{resp. }}\ddd{\theta}t\left(t\right) & > & \left(\mu-\delta\right)\abs{x\left(t\right)}>\frac{\left(\mu-\delta\right)\abs{x_{0}}}{1+\left(1+\delta\right)\abs{x_{0}}t}\Bigg)
\end{eqnarray*}
so that $t\mapsto\theta\left(t\right)$ is strictly decreasing (\emph{resp.
}increasing) as long as $x\left(t\right)\in\Theta_{-}\left(r,\mu\right)$
$\Big($\emph{resp.} $\Theta_{+}\left(r,\mu\right)$$\Big)$. Moreover,
if $\theta_{0}=\theta\left(0\right)$ is such that $x_{0}=x\left(0\right)\in\Theta_{-}\left(t,\mu\right)\backslash\Sigma_{+}\left(1,r,\omega\right)$
$\Big($\emph{resp. }$\Theta_{+}\left(r,\mu\right)\backslash\Sigma_{+}\left(1,r,\omega\right)$$\Big)$,
then as long as $x\left(t\right)\in\Theta_{-}\left(r,\mu\right)$
$\Big($\emph{resp. }$\Theta_{+}\left(r,\mu\right)$$\Big)$ we have:
\begin{eqnarray*}
\theta\left(t\right) & < & \theta_{0}-\left(\frac{\mu-\delta}{1+\delta}\right)\ln\left(1+\left(1+\delta\right)\abs{x_{0}}t\right)\\
\Bigg(\mbox{\emph{resp. }}\theta\left(t\right) & > & \theta_{0}+\left(\frac{\mu-\delta}{1+\delta}\right)\ln\left(1+\left(1+\delta\right)\abs{x_{0}}t\right)\Bigg)\qquad.
\end{eqnarray*}
We see that $x\left(t\right)\in\Sigma_{+}\left(1,r,\omega\right)$
for all 
\[
t\geq t_{0}:=\frac{\left(\exp\left(\frac{1+\delta}{\mu-\delta}\left(\theta_{0}-\frac{\pi}{2}-\arccos\left(\omega\right)\right)\right)-1\right)}{\left(1+\delta\right)\abs{x_{0}}}
\]
 
\[
\Bigg(\mbox{\emph{resp. }}t_{0}:=\frac{\left(\exp\left(\frac{1+\delta}{\mu-\delta}\left(\frac{\pi}{2}-\arccos\left(\omega\right)-\theta_{0}\right)\right)-1\right)}{\left(1+\delta\right)\abs{x_{0}}}\Bigg)\qquad.
\]
Indeed, if $t\geq t_{0}$, with $t_{0}$ as above, and if $x\left(t\right)\in\Theta_{+}\left(r,\mu\right)$,
then we have: 
\begin{eqnarray*}
\theta\left(t\right) & > & \theta_{0}+\left(\frac{\mu-\delta}{1+\delta}\right)\ln\left(1+\left(1+\delta\right)\abs{x_{0}}t\right)\\
 & > & \theta_{0}+\left(\frac{\mu-\delta}{1+\delta}\right)\ln\left(\exp\left(\frac{1+\delta}{\mu-\delta}\left(\theta_{0}-\frac{\pi}{2}-\arccos\left(\omega\right)\right)\right)\right)\\
 & = & \theta_{0}+\frac{\pi}{2}-\arccos\left(\omega\right)-\theta_{0}=\frac{\pi}{2}-\arccos\left(\omega\right)
\end{eqnarray*}
and therefore 
\[
-\arccos\left(\omega\right)<\arg\left(x\left(t\right)\right)-\frac{\pi}{2}<0\,\,.
\]
Hence, we have $x\left(t\right)\in\Sigma_{+}\left(1,r,\omega\right)$.
Moreover, notice that 
\begin{equation}
t_{0}\leq\frac{\exp\left(\left(\frac{1+\delta}{\mu-\delta}\right)\left(\epsilon+\arcsin\left(\omega\right)\right)\right)}{\left(1+\delta\right)\abs{x_{0}}}\qquad.\label{eq: temps critique}
\end{equation}

\end{itemize}
\medskip{}
On the one hand $\Sigma_{+}\left(1,r,\omega\right)\times\mathbf{D\left(0,r\right)}$
is stable by the flow of $\left(\mbox{\ref{eq: Probleme de cauchy}}\right)$
with increasing time $t\geq0$. Indeed in this region $\abs{x\left(t\right)},\abs{y_{1}\left(t\right)}\mbox{ and }\abs{y_{2}\left(t\right)}$
are decreasing, and as soon as $x\left(t\right)$ goes in $\Sigma_{+}\left(1,r,\omega\right)\cap\Theta_{-}\left(r,\mu\right)$
$\Big($\emph{resp. }$\Sigma_{+}\left(1,r,\omega\right)\cap\Theta_{+}\left(r,\mu\right)$$\Big)$,
which is non-empty and contains a part of the boundary of $\Sigma_{+}\left(1,r,\omega\right)$
with constant argument, $\theta\left(t\right)$ is decreasing\emph{
}(\emph{resp.} increasing). Then, $x\left(t\right)$ remains in $\Sigma_{+}\left(1,r,\omega\right)$.

On the other hand, as long as we are $x\left(t\right)$ belongs to
$\Theta_{-}\left(r,\mu\right)$ $\Big($\emph{resp.} $\Theta_{+}\left(r,\mu\right)$$\Big)$
we can re-parametrized the solutions by $\left(-\theta\right)$ ($resp$
$\theta$) (we are now going to make an abuse of notation, writing
when needed $x\left(\theta\right)$ or $x\left(t\right)$):
\[
\begin{cases}
\ddd{\abs x}{\left(-\theta\right)}=-\abs x\frac{\Re\left(\frac{ix}{1+bx+C_{+}\left(\mathbf{x}\right)}\right)}{\Im\left(\frac{ix}{1+bx+C_{+}\left(\mathbf{x}\right)}\right)}\leq\abs x.\frac{1+\delta}{\mu-\delta}\\
\Bigg(\mbox{\emph{resp. }}\ddd{\abs x}{\theta}=\abs x\frac{\Re\left(\frac{ix}{1+bx+C_{+}\left(\mathbf{x}\right)}\right)}{\Im\left(\frac{ix}{1+bx+C_{+}\left(\mathbf{x}\right)}\right)}\leq\abs x.\frac{1+\delta}{\mu-\delta}\Bigg)\\
\ddd{\abs{y_{1}}}{\left(-\theta\right)}=-\abs{y_{1}}\frac{\Re\left(ix\left(\frac{\frac{a}{2}+R_{+}^{\left(1\right)}\left(x,\mathbf{y}\right)}{1+bx+C_{+}\left(x,\mathbf{y}\right)}\right)\right)}{\Im\left(\frac{ix}{1+bx+C_{+}\left(\mathbf{x}\right)}\right)}\leq\abs{y_{1}}.\frac{\abs{\frac{a}{2}}+\delta'}{\mu-\delta}\\
\Bigg(\mbox{\emph{resp. }}\ddd{\abs{y_{1}}}{\theta}=\abs{y_{1}}\frac{\Re\left(ix\left(\frac{\frac{a}{2}+R_{+}^{\left(1\right)}\left(x,\mathbf{y}\right)}{1+bx+C_{+}\left(x,\mathbf{y}\right)}\right)\right)}{\Im\left(\frac{ix}{1+bx+C_{+}\left(\mathbf{x}\right)}\right)}\leq\abs{y_{1}}.\frac{\abs{\frac{a}{2}}+\delta'}{\mu-\delta}\Bigg)\\
\ddd{\abs{y_{2}}}{\left(-\theta\right)}=-\abs{y_{2}}\frac{\Re\left(ix\left(\frac{\frac{a}{2}+R_{+}^{\left(1\right)}\left(x,\mathbf{y}\right)}{1+bx+C_{+}\left(x,\mathbf{y}\right)}\right)\right)}{\Im\left(\frac{ix}{1+bx+C_{+}\left(\mathbf{x}\right)}\right)}\leq\abs{y_{2}}.\frac{\abs{\frac{a}{2}}+\delta'}{\mu-\delta}\\
\Bigg(\mbox{\emph{resp. }}\ddd{\abs{y_{2}}}{\theta}=\abs{y_{2}}\frac{\Re\left(ix\left(\frac{\frac{a}{2}+R_{+}^{\left(1\right)}\left(x,\mathbf{y}\right)}{1+bx+C_{+}\left(x,\mathbf{y}\right)}\right)\right)}{\Im\left(\frac{ix}{1+bx+C_{+}\left(\mathbf{x}\right)}\right)}\leq\abs{y_{2}}.\frac{\abs{\frac{a}{2}}+\delta'}{\mu-\delta}\Bigg).
\end{cases}
\]
Hence, if $\theta_{0}:=\theta\left(0\right)$ is such that $x_{0}:=x\left(0\right)\in\Theta_{-}\left(r,\mu\right)$
$\Big($\emph{resp.} $\Theta_{+}\left(r,\mu\right)$$\Big)$, for
$t\leq t_{0}$ we have: 
\begin{equation}
\begin{cases}
\abs{x\left(t\right)}\leq\abs{x_{0}}\exp\left(\frac{1+\delta}{\mu-\delta}\left(\theta_{0}-\theta\left(t\right)\right)\right)\\
\Bigg(\mbox{\emph{resp. }}\abs{x\left(t\right)}\leq\abs{x_{0}}\exp\left(\frac{1+\delta}{\mu-\delta}\left(\theta\left(t\right)-\theta_{0}\right)\right)\Bigg)\\
\abs{y_{1}\left(t\right)}\leq\abs{y_{1,0}}\exp\left(\frac{\abs{\frac{a}{2}}+\delta'}{\mu-\delta}\left(\theta_{0}-\theta\left(t\right)\right)\right)\\
\Bigg(\mbox{\emph{resp. }}\abs{y_{1}\left(t\right)}\leq\abs{y_{1,0}}\exp\left(\frac{\abs{\frac{a}{2}}+\delta'}{\mu-\delta}\left(\theta\left(t\right)-\theta_{0}\right)\right)\Bigg)\\
\abs{y_{2}\left(t\right)}\leq\abs{y_{2,0}}\exp\left(\frac{\abs{\frac{a}{2}}+\delta'}{\mu-\delta}\left(\theta_{0}-\theta\left(t\right)\right)\right)\\
\Bigg(\mbox{\emph{resp. }}\abs{y_{1}\left(t\right)}\leq\abs{y_{1,0}}\exp\left(\frac{\abs{\frac{a}{2}}+\delta'}{\mu-\delta}\left(\theta\left(t\right)-\theta_{0}\right)\right)\Bigg).
\end{cases}\label{eq: estimation zone croissante}
\end{equation}

\begin{defn}
\label{def: stable domain}We define the domain $\Omega_{+}$ as the
set of all 
\[
\begin{array}{c}
\mathbf{x}=\left(x,y_{1},y_{2}\right)\in S_{+}\left(r,\epsilon\right)\times\mathbf{D\left(0,r\right)}\end{array}
\]
 such that:\end{defn}
\begin{itemize}
\item $\mbox{ if }\Im\left(x\right)\geq\omega\abs x\mbox{ then }\begin{cases}
\abs x\leq r\exp\left(\frac{1+\delta}{\mu-\delta}\left(\arg\left(x\right)-\arcsin\left(\omega\right)\right)\right)\\
\abs{y_{1}}\leq r_{1}\exp\left(\frac{\abs{\frac{a}{2}}+\delta'}{\mu-\delta}\left(\arg\left(x\right)-\arcsin\left(\omega\right)\right)\right)\\
\abs{y_{2}}\leq r_{2}\exp\left(\frac{\abs{\frac{a}{2}}+\delta'}{\mu-\delta}\left(\arg\left(x\right)-\arcsin\left(\omega\right)\right)\right)
\end{cases}$;
\item $\mbox{if }\Im\left(x\right)\leq-\omega\abs x\mbox{ then }\begin{cases}
\abs x\leq r\exp\left(\frac{1+\delta}{\mu-\delta}\left(\pi-\arcsin\left(\omega\right)-\arg\left(x\right)\right)\right)\\
\abs{y_{1}}\leq r_{1}\exp\left(\frac{\abs{\frac{a}{2}}+\delta'}{\mu-\delta}\left(\pi-\arcsin\left(\omega\right)-\arg\left(x\right)\right)\right)\\
\abs{y_{2}}\leq r_{2}\exp\left(\frac{\abs{\frac{a}{2}}+\delta'}{\mu-\delta}\left(\pi-\arcsin\left(\omega\right)-\arg\left(x\right)\right)\right)
\end{cases}$.
\end{itemize}
We see that $\Omega_{+}$ is stable by the flow of $\left(\mbox{\ref{eq: Probleme de cauchy}}\right)$
with increasing time $t\geq0$. We have seen that for any initial
condition in $\Omega_{+}$, the solution exists for any $t\geq0$,
stays in $\Omega_{+}$, and after a finite time $t_{0}\geq0$ enters
and remains in $\Sigma_{+}\left(1,r,\omega\right)$. Finally, we have:
\[
S_{+}\left(r',\epsilon\right)\times\mathbf{D}\left(\mathbf{0},\mathbf{r'}\right)\subset\Omega_{+}\subset S_{+}\left(r,\epsilon\right)\times\mathbf{D\left(0,r\right)}\,\,,
\]
where 
\[
\begin{cases}
r'=r\exp\left(-\left(\frac{1+\delta}{\mu-\delta}\right)\left(\epsilon+\arcsin\left(\omega\right)\right)\right) & <r\\
r'_{1}=r_{1}\exp\left(-\left(\frac{\abs{\frac{a}{2}}+\delta'}{\mu-\delta}\right)\left(\epsilon+\arcsin\left(\omega\right)\right)\right) & <r_{1}\\
r'_{2}=r_{2}\exp\left(-\left(\frac{\abs{\frac{a}{2}}+\delta'}{\mu-\delta}\right)\left(\epsilon+\arcsin\left(\omega\right)\right)\right) & <r_{2}\,\,.
\end{cases}
\]
\medskip{}
Let $\mathbf{x}_{0}=\left(x_{0},\mathbf{y}_{0}\right)\in\Sigma_{+}\left(1,r,\omega\right)\times\mathbf{D\left(0,r\right)}$.
From $\left(\mbox{\ref{eq: majoration derivee de x}}\right)$ and
$\left(\mbox{\ref{eq: majoration derivee de y}}\right)$we have for
all $t\geq0$:
\begin{equation}
\begin{cases}
\abs{x\left(t\right)}\leq\frac{\abs{x_{0}}}{1+\left(\omega-\delta\right)\abs{x_{0}}t}\\
\abs{y_{1}\left(t\right)}\leq\frac{\abs{y_{1,0}}}{\left(1+\left(1+\delta\right)\abs{x_{0}}t\right)^{\frac{\omega'-\delta'}{1+\delta}}}\\
\abs{y_{1}\left(t\right)}\leq\frac{\abs{y_{2,0}}}{\left(1+\left(1+\delta\right)\abs{x_{0}}t\right)^{\frac{\omega'-\delta'}{1+\delta}}} & \,\,,
\end{cases}\label{eq: estimation zone decroissante}
\end{equation}
which proves that the solutions goes to $\mathbf{0}$ as $t\rightarrow+\infty$. \end{proof}
\begin{rem}
Another stable domain $\Omega_{-}$ is defined similarly when dealing
with the case ``$\pm=-$''
\end{rem}

\subsubsection{Construction of a sectorial analytic solution to the homological
equation}

~

We consider the meromorphic 1-form $\tau:=\frac{\mbox{d}x}{x^{2}}$
, which satisfies $\tau\cdot\left(Z_{\pm}\right)=1$. Let also $A_{\pm}\left(x,\mathbf{y}\right)$
be analytic in $S_{\pm}\left(r,\epsilon\right)\times\left(\ww C^{2},0\right)$
and dominated by $\norm{\mathbf{y}}_{\infty}$, and $M\in\ww N_{>0}$.
The following proposition is a precision of Lemma \ref{lem: solution eq homo}.
\begin{prop}
For all $\mathbf{x}_{0}\in\Omega_{\pm}$ (see Definition \ref{def: stable domain}),
the integral defined by 
\[
{\displaystyle \alpha_{\pm}\left(\mathbf{x}_{0}\right):=-\int_{\gamma_{\pm,\mathbf{x}_{0}}}x^{M+1}A_{\pm}\left(\mathbf{x}\right)\,\tau}
\]
 is absolutely convergent (the integration path $\gamma_{\pm,\mathbf{x_{0}}}$
is the one of Definition \ref{def: chemin asympto}). Moreover, the
function $\mathbf{x}_{0}\mapsto\alpha_{\pm}\left(\mathbf{x}_{0}\right)$
is analytic in $\Omega_{\pm}$, satisfies 
\[
\cal L_{Z_{\pm}}\left(\alpha_{\pm}\right)=x^{M+1}A_{\pm}\left(\mathbf{x}\right)
\]
and $\alpha_{\pm}\left(x,\mathbf{y}\right)=x^{M}\tilde{\alpha}_{\pm}\left(x,\mathbf{y}\right)$,
where $\tilde{\alpha}_{\pm}$ is analytic on $\Omega_{\pm}$ and dominated
by $\norm{\mathbf{y}}_{\infty}$.\end{prop}
\begin{proof}
We are going to use the estimations obtained in the previous paragraph.
\begin{itemize}
\item Let us start by proving that the integral above is convergent. We
begin with: 
\begin{eqnarray*}
\alpha_{\pm}\left(\mathbf{x}_{0}\right) & = & -\int_{0}^{+\infty}\frac{x\left(t\right)^{M+1}A_{\pm}\left(\mathbf{x}\left(t\right)\right)}{x\left(t\right)^{2}}\frac{ix\left(t\right)^{2}}{1+bx\left(t\right)+C_{+}\left(\mathbf{x}\left(t\right)\right)}\mbox{d}t\\
 & = & -i\int_{0}^{+\infty}\frac{x\left(t\right)^{M+1}A_{\pm}\left(\mathbf{x}\left(t\right)\right)}{1+bx\left(t\right)+C_{+}\left(\mathbf{x}\left(t\right)\right)}\mbox{d}t\,\,.
\end{eqnarray*}
Since $\mathbf{x}\left(t\right)\in\Omega_{\pm}$ for all $t\geq0$
and $A_{\pm}\left(x,\mathbf{y}\right)$ is dominated by $\norm{\mathbf{y}}_{\infty}$,
we have then: 
\begin{eqnarray*}
\abs{\frac{x\left(t\right)^{M+1}A_{\pm}\left(\mathbf{x}\left(t\right)\right)}{1+bx\left(t\right)+C_{+}\left(\mathbf{x}\left(t\right)\right)}} & \leq & C\abs{x\left(t\right)}^{M+1}\norm{\mathbf{y}\left(t\right)}_{\infty}
\end{eqnarray*}
where $C>0$ is some constant, independent of $\mathbf{x}_{0}$ and
$t$. For $t\geq0$ big enough, we deduce from paragraph \ref{sub:Domain-of-stability}
that: 
\begin{eqnarray*}
\abs{\frac{x\left(t\right)^{M+1}A_{\pm}\left(\mathbf{x}\left(t\right)\right)}{1+bx\left(t\right)+C_{+}\left(\mathbf{x}\left(t\right)\right)}} & \leq & C\norm{\mathbf{y}_{0}}\left(\frac{\abs{x_{0}}}{1+\left(\omega-\delta\right)\abs{x_{0}}t}\right)^{M+1}\frac{1}{\left(1+\left(1+\delta\right)\abs{x_{0}}t\right)^{\frac{\omega'-\delta'}{1+\delta}}}\\
 & = & \underset{t\rightarrow+\infty}{\tx O}\left(\frac{1}{t^{M+1}}\right)
\end{eqnarray*}
and then the integral is absolutely convergent. 
\item Let us prove the analyticity of $\alpha_{\pm}$ in $\Omega_{\pm}$:
it is sufficient to prove that it is analytic in every compact $K\subset\Omega_{\pm}$.
Let $K$ be such a compact subset. Let $L>0$ such that for all $\mathbf{x}\in K$,
we have:
\[
\abs{\frac{A_{\pm}\left(\mathbf{x}\right)}{1+bx+C_{+}\left(\mathbf{x}\right)}}\leq L.
\]
Since $K$ in a compact subset of $\Omega_{\pm}\subset S_{\pm}\left(r,\epsilon\right)\times\wcc$
and $S_{\pm}\left(r,\epsilon\right)$ is open ($0\notin S_{\pm}\left(r,\epsilon\right)$),
there exists $\delta>0$ such that for all $\mathbf{x}=\left(x,y_{1},y_{2}\right)\in K$,
we have $\delta<\abs x<r$. Finally, according to the several estimates
in paragraph \ref{sub:Domain-of-stability}, there exists $B>0$ such
that for all $\mathbf{x}_{0}\in K$ and $t\geq0$, we have:
\[
\abs{x\left(t\right)}\leq B\frac{\abs{x_{0}}}{1+\left(\omega-\delta\right)\abs{x_{0}}t}\qquad.
\]
Hence: 
\begin{eqnarray*}
\abs{\frac{x\left(t\right)^{M+1}A_{\pm}\left(\mathbf{x}\left(t\right)\right)}{1+bx\left(t\right)+C_{+}\left(\mathbf{x}\left(t\right)\right)}} & \leq & LB^{M+1}\frac{\abs{x_{0}}^{M+1}}{\left(1+\left(\omega-\delta\right)\abs{x_{0}}t\right)^{M+1}}\\
 & \leq & \frac{LB^{M+1}r^{M+1}}{\left(1+\left(\omega-\delta\right)\delta t\right)^{M+1}},
\end{eqnarray*}
and the classical theorem concerning the analyticity of integral with
parameters proves that $\alpha_{\pm}$ is analytic in any compact
$K\subset\Omega_{\pm}$, and consequently in $\Omega_{\pm}$.
\item Let us write $F\left(\mathbf{x}\right):=\frac{\pm ix^{M+1}A_{\pm}\left(\mathbf{x}\right)}{1+bx+C_{+}\left(\mathbf{x}\right)}$,
so that 
\[
\alpha_{\pm}\left(\mathbf{x}_{0}\right)=-\int_{0}^{+\infty}F\left(\Phi_{X_{\pm}}^{t}\left(\mathbf{x}_{0}\right)\right)\mbox{d}t\,\,.
\]
 For all $\mathbf{x}_{0}\in\Omega_{\pm}$, the function $t\mapsto\mathbf{x}\left(t\right)=\Phi_{X_{\pm}}^{t}\left(\mathbf{x}_{0}\right)$
satisfies: 
\[
\pp t\left(\Phi_{X_{\pm}}^{t}\left(\mathbf{x}_{0}\right)\right)=\frac{\pm i}{1+bx\left(\Phi_{X_{\pm}}^{t}\left(\mathbf{x}_{0}\right)\right)+C_{+}\left(\Phi_{X_{\pm}}^{t}\left(\mathbf{x}_{0}\right)\right)}Z_{\pm}\left(\Phi_{\pm}^{t}\left(\mathbf{x}\right)\right)\,\,.
\]
The classical theorem about the analyticity of integral with parameters
tells us that we can compute the derivatives inside the integral symbol:
\begin{eqnarray*}
\left(\cal L_{Z_{\pm}}\alpha_{\pm}\right)\left(\mathbf{x}_{0}\right) & = & -\int_{0}^{+\infty}\cal L_{Z_{\pm}}\left(F\circ\Phi^{s}\right)\left(\mathbf{x}_{0}\right)\mbox{d}s\\
 & = & -\int_{0}^{+\infty}\mbox{D}F\left(\Phi_{X_{\pm}}^{s}\left(\mathbf{x}_{0}\right)\right).\mbox{D}\Phi_{X_{\pm}}^{s}\left(\mathbf{x}_{0}\right).Z_{\pm}\left(\mathbf{x}_{0}\right)\mbox{d}s\\
 & = & -\int_{0}^{+\infty}\mbox{D}F\left(\Phi_{X_{\pm}}^{s}\left(\mathbf{x}\right)\right).\pp t\left(\Phi_{X_{\pm}}^{s+t}\left(\mathbf{x}_{0}\right)\right)_{\mid t=0}\left(\pm\frac{1+bx_{0}+C_{\pm}\left(\mathbf{x}_{0}\right)}{i}\right)\mbox{d}s\\
 & = & -\left(\pm\frac{1+bx_{0}+C_{\pm}\left(\mathbf{x}_{0}\right)}{i}\right).\int_{0}^{+\infty}\mbox{D}F\left(\Phi_{X_{\pm}}^{s}\left(\mathbf{x}_{0}\right)\right).\pp t\left(\Phi_{X_{\pm}}^{t}\left(\mathbf{x}_{0}\right)\right)_{\mid t=s}\mbox{d}s\\
 & = & -\left(\pm\frac{1+bx_{0}+C_{\pm}\left(\mathbf{x}_{0}\right)}{i}\right).\int_{0}^{+\infty}\pp s\left(F\circ\Phi_{X_{\pm}}^{s}\left(\mathbf{x}_{0}\right)\right)ds\\
 & = & -\left(\pm\frac{1+bx_{0}+C_{\pm}\left(\mathbf{x}_{0}\right)}{i}\right).\cro{F\circ\Phi_{X_{\pm}}^{s}\left(\mathbf{x}_{0}\right)}_{s=0}^{s=+\infty}\\
 & = & -\left(\pm\frac{1+bx_{0}+C_{\pm}\left(\mathbf{x}_{0}\right)}{i}\right).\left(-F\left(\mathbf{x}_{0}\right)\right)\\
 & = & x_{0}^{M+1}A_{\pm}\left(\mathbf{x}_{0}\right).
\end{eqnarray*}

\item Let us prove that $\tilde{\alpha}_{\text{\ensuremath{\pm}}}\left(x,\mathbf{y}\right):=\frac{\alpha_{\pm}\left(x,\mathbf{y}\right)}{x^{M}}$
is bounded and dominated by $\norm{\mathbf{y}}_{\infty}$ in $\Omega_{\pm}$.
The fact that it is analytic in $\Omega_{\pm}$ is clear because $\alpha_{\pm}$
is analytic there and $0\notin\Omega_{\pm}$. As above, there exists
there exists $C>0$ such that for all $\mathbf{x}_{0}:=\left(x_{0},\mathbf{y}_{0}\right)\in\Omega_{\pm}$
and for all $t\geq0$: 
\begin{eqnarray*}
\abs{\frac{x\left(\Phi_{X_{\pm}}^{t}\left(\mathbf{x}_{0}\right)\right)^{M+1}A_{\pm}\left(\Phi_{X_{\pm}}^{t}\left(\mathbf{x}_{0}\right)\right)}{\left(1+bx\left(\Phi_{X_{\pm}}^{t}\left(\mathbf{x}_{0}\right)\right)+C_{+}\left(\Phi_{X_{\pm}}^{t}\left(\mathbf{x}_{0}\right)\right)\right)}} & \leq & C\abs{x\left(\Phi_{X_{\pm}}^{t}\left(\mathbf{x}_{0}\right)\right)}^{M+1}\norm{\mathbf{y}\left(\Phi_{X_{\pm}}^{t}\left(\mathbf{x}_{0}\right)\right)}_{\infty}\qquad.
\end{eqnarray*}
We will only deal with the case where $x_{0}\in\Theta_{\pm}\left(r,\mu\right)$
(the case where $\Sigma_{\pm}\left(1,r,\omega\right)$ is easier and
can be deduced from that case). On the one hand from $\left(\mbox{\ref{eq: estimation zone croissante}}\right)$
we have for all $t\leq t_{0}$: 
\[
\begin{cases}
\abs{x\left(\Phi_{X_{\pm}}^{t}\left(\mathbf{x}_{0}\right)\right)}\leq D\abs{x_{0}} & ,\,\mbox{where }D:=\exp\left(\frac{1+\delta}{\mu-\delta}\left(\arccos\left(\mu\right)+\epsilon\right)\right)\\
\norm{\mathbf{y}\left(\Phi_{X_{\pm}}^{t}\left(\mathbf{x}_{0}\right)\right)}_{\infty}\leq D'\norm{\mathbf{y}_{0}}_{\infty} & ,\,\mbox{where }D':=\exp\left(\frac{\abs{\frac{a}{2}}+\delta'}{\mu-\delta}\left(\arccos\left(\mu\right)+\epsilon\right)\right)\,\,.
\end{cases}
\]
On the other hand we have seen in $\left(\mbox{\ref{eq: estimation zone decroissante}}\right)$
that for all $t\geq t_{0}$:
\[
\begin{cases}
\abs{x\left(\Phi_{X_{\pm}}^{t}\left(\mathbf{x}_{0}\right)\right)}\leq\frac{\abs{x\left(\Phi_{X_{\pm}}^{t_{0}}\left(\mathbf{x}_{0}\right)\right)}}{1+\left(\omega-\delta\right)\abs{x\left(\Phi_{X_{\pm}}^{t}\left(\mathbf{x}_{0}\right)\right)}\left(t-t_{0}\right)}\\
\norm{\mathbf{y}\left(\Phi_{X_{\pm}}^{t}\left(\mathbf{x}_{0}\right)\right)}_{\infty}\leq\norm{\mathbf{y}_{0}}_{\infty} & \,\,.
\end{cases}
\]
Hence, we use the Chasles relation and the estimations above to obtain:
\begin{eqnarray*}
\abs{\tilde{\alpha}_{\pm}\left(x_{0},\mathbf{y}_{0}\right)} & \leq & \frac{\abs{\alpha_{\pm}\left(x_{0},\mathbf{y}_{0}\right)}}{\abs{x_{0}}^{M}}\\
 & \leq & \frac{CD^{M+1}D'\norm{\mathbf{y}_{0}}_{\infty}\abs{x_{0}}^{M+1}\abs{t_{0}}}{\abs{x_{0}}^{M}}\\
 &  & +\frac{C\norm{\mathbf{y}_{0}}_{\infty}}{\abs{x_{0}}^{M}}\int_{t_{0}}^{+\infty}\frac{\mbox{d}t}{\left(1+\left(\omega-\delta\right)\abs{x\left(\Phi_{X_{\pm}}^{t}\left(\mathbf{x}_{0}\right)\right)}\left(t-t_{0}\right)\right)}\\
 & \leq & CD^{M+1}D'\norm{\mathbf{y}_{0}}_{\infty}\abs{x_{0}}\abs{t_{0}}+\frac{C\norm{\mathbf{y}_{0}}_{\infty}\abs{x\left(\Phi_{X_{\pm}}^{t_{0}}\left(\mathbf{x}_{0}\right)\right)}^{M+1}}{M\left(\omega-\delta\right)\abs{x_{0}}^{M}\abs{x\left(\Phi_{X_{\pm}}^{t_{0}}\left(\mathbf{x}_{0}\right)\right)}}\,\,;
\end{eqnarray*}
and according to $\left(\mbox{\ref{eq: temps critique}}\right)$ we
have
\begin{eqnarray*}
\abs{\tilde{\alpha}_{\pm}\left(x_{0},\mathbf{y}_{0}\right)} & \leq & \left(\frac{D^{2}D'}{\left(1+\delta\right)}+\frac{1}{M\left(\omega-\delta\right)}\right)CD^{M}\norm{\mathbf{y}_{0}}_{\infty}\qquad.
\end{eqnarray*}

\end{itemize}
\end{proof}

\section{Uniqueness and weak 1-summability }

In this section, we prove the uniqueness of the normalizing maps obtained
in Corollary \ref{cor: existence normalisations sectorielles} (see
Proposition \ref{prop: unique normalizations}) and also the weak
1-summability of the formal normalizing map of Theorem \ref{thm: forme normalel formelle}
(see Proposition \ref{prop: Weak sectorial normalizations}). In particular,
we end this section by proving Theorem \ref{Th: Th drsn}.

\subsection{\label{sub :Sectorial isotropies in big sectors}Sectorial isotropies
in ``wide'' sectors and uniqueness of the normalizing maps: proof
of Proposition \ref{prop: unique normalizations}.}

~

We consider a normal form $\ynorm$ as given by Corollary \ref{cor: existence normalisations sectorielles}.
We study here the germs of sectorial isotropies of the normal form
$\ynorm$ in $S_{\pm}\times\left(\ww C^{2},0\right)$, where $S_{\pm}\in\germsect{\arg\left(\pm i\lambda\right)}{\eta}$
is a sectorial neighborhood of the origin with opening $\eta\in\left]\pi,2\pi\right[$
in the direction $\arg\left(\pm i\lambda\right)$. Proposition \ref{prop: unique normalizations}
states that the normalizing maps $\left(\Phi_{+},\Phi_{-}\right)$
are unique as sectorial germs. It is a straightforward consequence
of Proposition \ref{prop: isot sect} below, which show that the only
sectorial fibered isotropy (tangent to the identity) of the normal
form in over ``wide'' sector (\emph{i.e. }of opening $>\pi$) is
the identity itself.
\begin{defn}
A germ of sectorial fibered diffeomorphism $\Phi_{\theta,\eta}$ in
the direction $\theta\in\ww R$ with opening $\eta\geq0$ and tangent
to the identity, is a germ of fibered sectorial \emph{isotropy }of
$\ynorm$ (in the direction $\theta\in\ww R$ with opening $\eta\geq0$
and tangent to the identity ) if $\left(\Phi_{\theta,\eta}\right)_{*}\left(\ynorm\right)=\ynorm$
in $\cal S\in\cal S_{\theta,\eta}$. We denote by $\isotsect Y{\theta}{\eta}\subset\diffsect$
the subset formed composed of these elements.
\end{defn}
Proposition \ref{prop: unique normalizations} is an immediate consequence
of the following one.
\begin{prop}
\label{prop: isot sect}~For all $\eta\in\left]\pi,2\pi\right[$:
\[
\isotsect{\ynorm}{\arg\left(\pm i\lambda\right)}{\eta}=\acc{\tx{Id}}\,\,.
\]
\end{prop}
\begin{proof}
Let 
\[
\phi:\left(x,\mathbf{y}\right)\mapsto\left(x,\phi_{1}\left(x,\mathbf{y}\right),\phi_{2}\left(x,\mathbf{y}\right)\right)\in\isotsect{\ynorm}{\arg\left(\pm i\lambda\right)}{\eta}
\]
 be a germ of a sectorial fibered isotropy (tangent to the identity)
of $\ynorm$ in $\cal S_{\pm}\in\germsect{\arg\left(\pm i\lambda\right)}{\eta}$
with $\eta\in\left]\pi,2\pi\right[$. Possibly by reducing our domain,
we can assume that $\cal S_{\pm}$ is bounded and of the form $S_{\pm}\times\mathbf{D\left(0,r\right)}$
(where, as usual, $S_{\pm}$ is an adapted sector and $\mathbf{D\left(0,r\right)}$
a polydisc), and that $\phi$ is bounded in this domain. We have
\[
\phi_{*}\left(\ynorm\right)=\ynorm
\]
\emph{i.e.} 
\[
\mbox{D}\phi\cdot\ynorm=\ynorm\circ\phi
\]
 which is also equivalent to: 
\begin{equation}
\begin{cases}
x^{2}\ppp{\phi_{1}}x+\left(-1-c\left(y_{1}y_{2}\right)+a_{1}x\right)y_{1}\ppp{\phi_{1}}{y_{1}}+\left(1+c\left(y_{1}y_{2}\right)+a_{2}x\right)y_{2}\ppp{\phi_{1}}{y_{2}}\\
=\phi_{1}\left(-1-c\left(\phi_{1}\phi_{2}\right)+a_{1}x\right)\\
x^{2}\ppp{\phi_{2}}x+\left(-1-c\left(y_{1}y_{2}\right)+a_{1}x\right)y_{1}\ppp{\phi_{2}}{y_{1}}+\left(1+c\left(y_{1}y_{2}\right)+a_{2}x\right)y_{2}\ppp{\phi_{2}}{y_{2}}\\
=\phi_{2}\left(1+c\left(\phi_{1}\phi_{2}\right)+a_{2}x\right) & \,\,.
\end{cases}\label{eq: isot sect}
\end{equation}
Let us consider $\psi:=\phi_{1}\phi_{2}$. Then
\[
x^{2}\ppp{\psi}x+\left(-1-c\left(y_{1}y_{2}\right)+a_{1}x\right)y_{1}\ppp{\psi}{y_{1}}+\left(1+c\left(y_{1}y_{2}\right)+a_{2}x\right)y_{2}\ppp{\psi}{y_{2}}=\left(a_{1}+a_{2}\right)x\psi\qquad.
\]
By assumption we can write
\[
{\displaystyle \psi\left(x,\mathbf{y}\right)=\sum_{j_{1}+j_{2}\geq2}}\psi_{j_{1},j_{2}}\left(x\right)y_{1}^{j_{1}}y_{2}^{j_{2}}\,\,,
\]
where $\psi_{j_{1},j_{2}}\left(x\right)$ is analytic and bounded
in $S_{\pm}$ for all $j_{1},j_{2}\geq0$ and such that 
\[
{\displaystyle \sum_{j_{1}+j_{2}\geq1}\left(\underset{x\in S_{\pm}}{\sup}\left(\abs{\psi_{j_{1},j_{2}}\left(x\right)}\right)\right)y_{1}^{j_{1}}y_{2}^{j_{2}}}
\]
 is convergent near the origin of $\ww C^{2}$ (\emph{e.g. }in $\mathbf{D\left(0,r\right)})$.
Consequently, with an argument of uniform convergence in every compact
subset, we have for all $j_{1},j_{2}\geq0$:
\begin{eqnarray*}
 &  & x^{2}\ddd{\psi_{j_{1};j_{2}}}x\left(x\right)+\left(j_{2}-j_{1}+\left(a_{1}\left(j_{1}-1\right)+a_{2}\left(j_{2}-1\right)\right)x\right)\psi_{j_{1},j_{2}}\left(x\right)\\
 &  & =\left(j_{1}-j_{2}\right)\sum_{l=1}^{\min\left(j_{1},j_{2}\right)}\psi_{j_{1}-l,j_{2}-l}\left(x\right)c_{l}\,\,\,.
\end{eqnarray*}
For $j_{1}=j_{2}=j\geq1$, we have 
\[
\psi_{j,j}\left(x\right)=b_{j,j}x^{-\left(j-1\right)\left(a_{1}+a_{2}\right)},\qquad b_{j,j}\in\ww C.
\]
Since $\Re\left(a_{1}+a_{2}\right)>0$, the function $x\mapsto\psi_{j,j}\left(x\right)$
is bounded near the origin if and only if $b_{j,j}=0$ or $j=1$.
For $j_{1}>j_{2}$, we see recursively that $\psi_{j_{1},j_{2}}\left(x\right)=0$.
Indeed, we obtain by induction that 
\[
\psi_{j_{1},j_{2}}\left(x\right)=b_{j_{1},j_{2}}\exp\left(\frac{j_{2}-j_{1}}{x}\right)x^{-\left(a_{1}\left(j_{1}-1\right)+a_{2}\left(j_{2}-1\right)\right)}\qquad,
\]
and since it has to be bounded on $S_{\pm}$, we necessarily have
$b_{j_{1},j_{2}}=0$. Similarly, for $j_{1}<j_{2}$, we see recursively
that $\psi_{j_{1},j_{2}}\left(x\right)=0$. As a conclusion, $\psi\left(x,\mathbf{y}\right)=b_{1,1}y_{1}y_{2}=y_{1}y_{2}$
(we must have $b_{1,1}=1$ since $\phi$ is tangent to the identity). 

We can now solve separately each equation in $\left(\mbox{\ref{eq: isot sect}}\right)$:
\[
\begin{cases}
x^{2}\ppp{\phi_{1}}x+\left(-1-c\left(y_{1}y_{2}\right)+a_{1}x\right)y_{1}\ppp{\phi_{1}}{y_{1}}+\left(1+c\left(y_{1}y_{2}\right)+a_{2}x\right)y_{2}\ppp{\phi_{1}}{y_{2}}\\
=\phi_{1}\left(-1-c\left(y_{1}y_{2}\right)+a_{1}x\right)\\
x^{2}\ppp{\phi_{2}}x+\left(-1-c\left(y_{1}y_{2}\right)+a_{1}x\right)y_{1}\ppp{\phi_{2}}{y_{1}}+\left(1+c\left(y_{1}y_{2}\right)+a_{2}x\right)y_{2}\ppp{\phi_{2}}{y_{2}}\\
=\phi_{2}\left(1+c\left(y_{1}y_{2}\right)+a_{2}x\right) & \,\,\,.
\end{cases}
\]
As above for $i=1,2$ we can write 
\[
{\displaystyle \phi_{i}\left(x,\mathbf{y}\right)=\sum_{j_{1}+j_{2}\geq1}}\phi_{i,j_{1},j_{2}}\left(x\right)y_{1}^{j_{1}}y_{2}^{j_{2}}\,\,,
\]
 where $\phi_{i,j_{1},j_{2}}\left(x\right)$ is analytic and bounded
in $S_{\pm}$ for all $j_{1},j_{2}\geq0$ and such that 
\[
{\displaystyle \sum_{j_{1}+j_{2}\geq1}\left(\underset{x\in S_{\pm}}{\sup}\left(\abs{\phi_{i,j_{1},j_{2}}\left(x\right)}\right)\right)y_{1}^{j_{1}}y_{2}^{j_{2}}}
\]
 is a convergent entire series near the origin of $\ww C^{2}$ (\emph{e.g.
}in $\mathbf{D\left(0,r\right)})$. As above, using the uniform convergence
in every compact subset and identifying terms of same homogeneous
degree $\left(j_{1},j_{2}\right)$, we obtain:
\[
\begin{cases}
{\displaystyle x^{2}\ddd{\phi_{1,j_{1};j_{2}}}x\left(x\right)+\left(j_{2}-j_{1}+1+\left(a_{1}\left(j_{1}-1\right)+a_{2}j_{2}\right)x\right)\phi_{1,j_{1},j_{2}}\left(x\right)}\\
{\displaystyle =\sum_{l=1}^{\min\left(j_{1},j_{2}\right)}\phi_{1,j_{1}-l,j_{2}-l}\left(x\right)\left(j_{1}-j_{2}-1\right)c_{l}}\\
{\displaystyle x^{2}\ddd{\phi_{2,j_{1};j_{2}}}x\left(x\right)+\left(j_{2}-j_{1}-1+\left(a_{1}j_{1}+a_{2}\left(j_{2}-1\right)\right)x\right)\phi_{2,j_{1},j_{2}}\left(x\right)}\\
{\displaystyle =\sum_{l=1}^{\min\left(j_{1},j_{2}\right)}\phi_{2,j_{1}-l,j_{2}-l}\left(x\right)\left(j_{1}-j_{2}+1\right)c_{l}\qquad.}
\end{cases}
\]
From this we deduce: 
\[
\begin{cases}
\phi_{1,1,0}\left(x\right)=p_{1,0}\in\ww C\backslash\acc 0\\
\phi_{2,0,1}\left(x\right)=q_{0,1}\in\ww C\backslash\acc 0
\end{cases}
\]
with $p_{1,0}q_{0,1}=1$. Then, using the assumption that $\phi_{i,j_{1},j_{2}}\left(x\right)$
is analytic and bounded in $S_{\pm}$ for all $j_{1},j_{2}\geq0$,
we see (by induction on $j\geq1$) that
\[
\forall j\geq1\,\,\begin{cases}
\phi_{1,j+1,j}=0\\
\phi_{2,j,j+1}=0
\end{cases}\qquad.
\]
Indeed, we show recursively that for all $j\geq1$, we have:
\[
x^{2}\ddd{\phi_{1,j+2,j+1}}x\left(x\right)+\left(j+1\right)\left(a_{1}+a_{2}\right)x\phi_{1,j+2,j+1}\left(x\right)=0\,\,,
\]
 and the general solution to this equation is:
\[
\phi_{1,j+2,j+1}\left(x\right)={\displaystyle p_{j+2,j+1}x^{-\left(j+1\right)\left(a_{1}+a_{2}\right)}}\,\,,\,\,\mbox{with }p_{j+2j+1}\in\ww C\,\,.
\]
The quantity $\phi_{1,j+2,j+1}\left(x\right)$ is bounded near the
origin if and only if $p_{j+2,j+1}=0$, since $\Re\left(a_{1}+a_{2}\right)>0$.
The same arguments work for $\phi_{2,j,j+1}$, $j\geq1$. Consequently:
\[
{\displaystyle \begin{cases}
{\displaystyle x^{2}\ddd{\phi_{1,j_{1};j_{2}}}x\left(x\right)+\left(j_{2}-j_{1}+1+\left(a_{1}\left(j_{1}-1\right)+a_{2}j_{2}\right)x\right)\phi_{1,j_{1},j_{2}}\left(x\right)}\\
{\displaystyle =\left(j_{1}-j_{2}-1\right)\sum_{l=1}^{\min\left(j_{1},j_{2}\right)}\phi_{1,j_{1}-l,j_{2}-l}\left(x\right)c_{l}}\\
{\displaystyle x^{2}\ddd{\phi_{2,j_{1};j_{2}}}x\left(x\right)+\left(j_{2}-j_{1}-1+\left(a_{1}j_{1}+a_{2}\left(j_{2}-1\right)\right)x\right)\phi_{2,j_{1},j_{2}}\left(x\right)}\\
{\displaystyle =\left(j_{1}-j_{2}+1\right)\sum_{l=1}^{\min\left(j_{1},j_{2}\right)}\phi_{2,j_{1}-l,j_{2}-l}\left(x\right)c_{l}\qquad.}
\end{cases}}
\]
 Once again, we see recursively that for $j_{1}>j_{2}+1$, $\phi_{1,j_{1},j_{2}}\left(x\right)=0$.
Indeed, we obtain by induction that 
\[
\phi_{1,j_{1},j_{2}}\left(x\right)=p_{j_{1},j_{2}}\exp\left(\frac{j_{2}-j_{1}+1}{x}\right)x^{-\left(a_{1}\left(j_{1}-1\right)+a_{2}j_{2}\right)}\qquad,
\]
and since this has to be bounded on $S_{\pm}$, we necessarily have
$p_{j_{1},j_{2}}=0$, and therefore $\phi_{1,j_{1},j_{2}}\left(x\right)=0$.
Similarly, for $j_{1}<j_{2}+1$, we prove that $\phi_{j_{1},j_{2}}\left(x\right)=0$.
As a conclusion, $\phi_{1}\left(x,\mathbf{y}\right)=y_{1}$. By exactly
the same kind of arguments we have $\phi_{2}\left(x,\mathbf{y}\right)=y_{2}$.
\end{proof}

\subsection{\label{sub: Weak 1-summability of normalization}Weak 1-summability
of the normalizing map}

~

Let us consider the same data as in Lemma \ref{lem: solution eq homo}.
The following lemma states that an analytic solution to the considered
homological equation in $\cal S_{\pm}\in\cal S_{\arg\left(\pm i\lambda\right),\eta}$
with $\eta\in\big[\pi,2\pi\big[$, admits a weak Gevrey-1 asymptotic
expansion in this sector. In other words, it is the weak 1-sum of
a formal solution the homological equation. Let us re-use the notations
introduced at the beginning of the latter section \ref{sec:Sectorial-analytic-normalization}.
\begin{lem}
\label{lem: weak summability homo equation}Let 

\[
Z:=Y_{0}+C\left(x,\mathbf{y}\right)\overrightarrow{\cal C}+xR^{\left(1\right)}\left(x,\mathbf{y}\right)\overrightarrow{\cal R}\,\,
\]
be a formal vector field weakly 1-summable in $\cal S_{\pm}\in\cal S_{\arg\left(\pm i\lambda\right),\eta}$,
with $\eta\in\big[\pi,2\pi\big[$ and $C,R^{\left(1\right)}$ of order
at least one with respect to $\mathbf{y}$. We denote by 
\[
Z_{\pm}:=Y_{0}+C_{\pm}\left(x,\mathbf{y}\right)\overrightarrow{\cal C}+xR_{\pm}^{\left(1\right)}\left(x,\mathbf{y}\right)\overrightarrow{\cal R}
\]
the associate weak 1-sum in $\cal S_{\pm}$. Let also $A\in\form{x,\mathbf{y}}$
be weakly 1-summable in $\cal S_{\pm}$, of 1-sum $A_{\pm}$ and of
order at least one with respect to $\mathbf{y}$. Then, any sectorial
germ of an analytic function of the form $\alpha_{\pm}\left(x,\mathbf{y}\right)=x^{M}\tilde{\alpha}_{\pm}\left(x,\mathbf{y}\right)$
, with $M\in\ww N_{>0}$ and $\tilde{\alpha}_{\pm}$ analytic in $\cal S_{\pm}$,
which is dominated by $\norm{\mathbf{y}}_{\infty}$ and satisfies
\[
\cal L_{Z_{\pm}}\left(\alpha_{\pm}\right)=x^{M+1}A_{\pm}\left(x,\mathbf{y}\right)\qquad,
\]
has a Gevrey-1 asymptotic expansion in $\cal S_{\pm}$, denoted by
$\alpha$. Moreover, $\alpha$ is a formal solution to
\[
\cal L_{Z}\left(\alpha\right)=x^{M+1}A\left(x,\mathbf{y}\right)\qquad.
\]
\end{lem}
\begin{proof}
Let us write $Z$ as follow: 
\begin{eqnarray*}
Z & = & x^{2}\pp x+\left(-\left(\lambda+d\left(y_{1}y_{2}\right)\right)+a_{1}x+F_{1}\left(x,\mathbf{y}\right)\right)y_{1}\pp{y_{1}}\\
 &  & +\left(\lambda+d\left(y_{1}y_{2}\right)+a_{2}x+F_{2}\left(x,\mathbf{y}\right)\right)y_{2}\pp{y_{2}}\,\,,
\end{eqnarray*}
with $F_{1},F_{2}$ weakly 1-summable in $\cal S_{\pm}\in\cal S_{\arg\left(\pm i\lambda\right),\eta}$,
with $\eta\in\big[\pi,2\pi\big[$, of weak 1-sums $F_{1,\pm},F_{2,\pm}$
respectively, which are dominated by $\norm{\mathbf{y}}$, and with
$d\left(v\right)\in v\germ v$ without constant term. Consider the
Taylor expansion with respect to $\mathbf{y}$ of $d$,$F_{1}$,$F_{2}$,$A$
and $\alpha$:
\[
\begin{cases}
{\displaystyle d\left(y_{1}y_{2}\right)=\sum_{k\geq1}d_{k}y_{1}^{k}y_{2}^{k}}\\
{\displaystyle F_{1}\left(x,\mathbf{y}\right)=\sum_{j_{1}+j_{2}\geq1}F_{1,\mathbf{j}}\left(x\right)\mathbf{y^{j}}}\\
{\displaystyle F_{2}\left(x,\mathbf{y}\right)=\sum_{j_{1}+j_{2}\geq1}F_{2,\mathbf{j}}\left(x\right)\mathbf{y^{j}}}\\
{\displaystyle A\left(x,\mathbf{y}\right)=\sum_{j_{1}+j_{2}\geq1}A_{\mathbf{j}}\left(x\right)\mathbf{y^{j}}}\\
{\displaystyle \alpha\left(x,\mathbf{y}\right)=\sum_{j_{1}+j_{2}\geq1}\alpha_{\mathbf{j}}\left(x\right)\mathbf{y^{j}}}
\end{cases}
\]
(same expansions are valid in $\cal S_{\pm}$ for the corresponding
weak 1-sums). As usual, possibly by reducing $\cal S_{\pm}$, we can
assume that $\cal S_{\pm}=S_{\pm}\times\mathbf{D\left(0,r\right)}$
(where $S_{\pm}$ is an adapted sector and $\mathbf{D\left(0,r\right)}$
a polydisc). The homological equation 
\[
\cal L_{Z}\left(\alpha\right)=x^{M+1}A_{\pm}\left(x,\mathbf{y}\right)
\]
can be re-written:
\begin{eqnarray*}
x^{2}\ppp{\alpha}x+\left(-\left(\lambda+d\left(y_{1}y_{2}\right)\right)+a_{1}x+F_{1,\pm}\left(x,\mathbf{y}\right)\right)y_{1}\ppp{\alpha}{y_{1}}\\
+\left(\lambda+d\left(y_{1}y_{2}\right)+a_{2}x+F_{2,\pm}\left(x,\mathbf{y}\right)\right)y_{2}\ppp{\alpha}{y_{2}} & = & x^{M+1}A_{\pm}\left(x,\mathbf{y}\right)\,\,.
\end{eqnarray*}
Using normal convergence in any compact subset of $\cal S_{\pm}$,
we can compute the partial derivatives of 
\[
{\displaystyle \alpha\left(x,\mathbf{y}\right)=\sum_{j_{1}+j_{2}\geq1}\alpha_{\mathbf{j}}\left(x\right)\mathbf{y^{j}}}
\]
 with respect to $x$, $y_{1}$ or $y_{2}$ term by term, in order
to obtain after identification: $\forall\mathbf{j}=\left(j_{1},j_{2}\right)\in\ww N^{2}$,
\begin{eqnarray*}
x^{2}\ddd{\alpha_{\mathbf{j},\pm}}x\left(x\right)+\left(\lambda\left(j_{2}-j_{1}\right)+\left(a_{1}j_{1}+a_{2}j_{2}\right)x\right)\alpha_{\mathbf{j},\pm}\left(x\right) & = & G_{\mathbf{j},\pm}\left(x\right)\,\,\,,
\end{eqnarray*}
where $G_{\mathbf{j},\pm}\left(x\right)$ depends only on $d_{k},F_{1,\mathbf{k},\pm},F_{2,\mathbf{k},\pm},\alpha_{\mathbf{k},\pm}$
and $A_{\mathbf{l},\pm}$, for $k\leq\min\left(j_{1},j_{2}\right)$,
$\abs{\mathbf{k}}\leq\mathbf{\abs j}-1$ and $\abs{\mathbf{l}}\leq\abs{\mathbf{j}}$.
We obtain a similar differential equation for the associated formal
power series. Let us prove by induction on $\abs{\mathbf{j}}\geq0$
that:
\begin{enumerate}
\item $G_{\mathbf{j},\pm}$ is the 1-sum of $G_{\mathbf{j}}$ in $S_{\pm}$,
\item $G_{j,j}\left(0\right)=0$ if $\mathbf{j}=\left(j,j\right)$
\item $\alpha_{\mathbf{j},\pm}$ is the 1-sum $\alpha_{\mathbf{j}}$ in
$S_{\pm}$.
\end{enumerate}
It is paramount to use the fact that for all $\mathbf{j}\in\ww N^{2}$,
$\alpha_{\mathbf{j},\pm}$ is bounded in $S_{\pm}$. 
\begin{itemize}
\item For $\mathbf{j}=\left(0,0\right)$, we have $G_{\left(0,0\right)}=0$
and then $\alpha_{\left(0,0\right)}=0$.
\item Let $\mathbf{j}=\left(j_{1},j_{2}\right)\in\ww N^{2}$ with $\abs{\mathbf{j}}=j_{1}+j_{2}\geq1$.
Assume the property holds for all $\mathbf{k}\in\ww N^{2}$ with $\abs{\mathbf{k}}\leq\abs{\mathbf{j}}-1$. 

\begin{enumerate}
\item Since $G_{\mathbf{j}}\left(x\right)$ depends only on $d_{k},F_{1,\mathbf{k}},F_{2,\mathbf{k}},\alpha_{\mathbf{k}}$
and $A_{\mathbf{l}}$, for $k\leq\min\left(j_{1},j_{2}\right)$, $\abs{\mathbf{k}}\leq\mathbf{\abs j}-1$
and $\abs{\mathbf{l}}\leq\abs{\mathbf{j}}$, then $G_{\mathbf{j}}$
is 1-summable in $S_{\pm}$, of 1-sum $G_{\mathbf{j},\pm}$. 
\item We also see that $G_{j,j}\left(0\right)=0$, if $\mathbf{j}=\left(j,j\right)$.
\item If $j_{1}\neq j_{2}$, then point 1. in Proposition \ref{prop: solution borel sommable precise}
tells us that there exists a unique formal solution $\alpha_{\mathbf{j}}\left(x\right)$
to the irregular differential equation we are looking at, and such
that $\alpha_{\mathbf{j}}\left(0\right)=\frac{1}{\lambda\left(j_{2}-j_{1}\right)}G_{\mathbf{j}}\left(0\right)$
. Moreover, this solution is 1-summable in $S_{\pm}$ since the same
goes for $G_{\mathbf{j}}$.\\

\item If however $j_{1}=j_{2}=j\geq1$, since $G_{\left(j,j\right)}\left(0\right)=0$
we can write $G_{\left(j,j\right)}\left(x\right)=x\tilde{G}_{\left(j,j\right)}\left(x\right)$
with $\tilde{G}_{\left(j,j\right)}\left(x\right)$ 1-summable in $S_{\pm}$,
and then the differential equation becomes regular: 
\[
x\ddd{\alpha_{\left(j,j\right),\pm}}x\left(x\right)+\left(a_{1}+a_{2}\right)j\alpha_{\left(j,j\right),\pm}\left(x\right)=\tilde{G}_{\left(j,j\right),\pm}\left(x\right)\,\,.
\]
Since $\Re\left(a_{1}+a_{2}\right)>0$, according to point 2. in Proposition
\ref{prop: solution borel sommable precise}, the latter equation
has a unique formal solution $\alpha_{\left(j,j\right)}\left(x\right)$
such that $\alpha_{\left(j,j\right)}\left(0\right)=\frac{\tilde{G}_{\left(j,j\right)}\left(0\right)}{\left(a_{1}+a_{2}\right)j}$,
and this solution is moreover 1-summable in $S_{\pm}$, and its 1-sum
is the only solution to this equation bounded in $S_{\pm}$. Thus,
it is necessarily $\alpha_{\left(j,j\right),\pm}$.
\end{enumerate}
\end{itemize}
\end{proof}
We are now able to prove the weak 1-summability of the formal normalizing
map.
\begin{prop}
\label{prop: Weak sectorial normalizations}The sectorial normalizing
maps $\left(\Phi_{+},\Phi_{-}\right)$ in Corollary \ref{cor: existence normalisations sectorielles}
are the weak 1-sums in $\cal S_{\pm}\in\cal S_{\arg\left(\pm\lambda\right),\eta}$
of the formal normalizing map $\hat{\Phi}$ given by Theorem \ref{thm: forme normalel formelle},
for all $\eta\in\big[\pi,2\pi\big[$. In particular, $\hat{\Phi}$
is weakly 1-summable, except for $\arg\left(\pm\lambda\right)$.\end{prop}
\begin{proof}
The normalizing map $\Phi_{\pm}$ from Corollary \ref{cor: existence normalisations sectorielles}
is constructed as the composition of two germs of sectorial diffeomorphisms,
using successively Propositions \ref{prop: forme pr=0000E9par=0000E9e ordre N}
and \ref{prop: norm sect}. The sectorial map obtained in Proposition
\ref{prop: forme pr=0000E9par=0000E9e ordre N} is 1-summable except
in directions ${\displaystyle \arg\left(\pm\lambda\right)}$. The
sectorial transformation in Proposition \ref{prop: norm sect} is
constructed as the composition of two germs of sectorial diffeomorphisms,
using successively Proposition \ref{prop: radial part} and \ref{prop: tangential part}.
Both of these two sectorial maps are built thanks to Lemma \ref{lem: solution eq homo}.
Lemma \ref{lem: weak summability homo equation} above justifies that
each of these maps admits in fact a weak Gevrey-1 asymptotic expansion
in a domain of the form $\cal S_{\pm}\in\cal S_{\arg\left(\pm\lambda\right),\eta}$,
for all $\eta\in\big[\pi,2\pi\big[$. Consequently, the same goes
for the sectorial diffeomorphisms of Proposition \ref{prop: norm sect},
and then for those of Corollary \ref{cor: existence normalisations sectorielles}
(we used here Proposition \ref{prop: composition faible} for the
composition). 

Using item \ref{enu: derivation faible} in Lemma \ref{lem: proprietes faibles},
we deduce that the weak Gevrey-1 asymptotic expansion of the sectorial
normalizing maps of Corollary \ref{cor: existence normalisations sectorielles}
is therefore a formal normalizing map, such as the one given by Theorem
\ref{thm: forme normalel formelle}. By uniqueness of such a normalizing
map, it is $\hat{\Phi}$.
\end{proof}

\subsection{Proof of Theorem \ref{Th: Th drsn}.}

~

We can now prove Theorem \ref{Th: Th drsn}.
\begin{proof}[Proof of Theorem \ref{Th: Th drsn}.]

The existence of $\Phi_{+}$ and $\Phi_{-}$ is obtained in Corollary
\ref{cor: existence normalisations sectorielles}. The uniqueness
is given by \ref{prop: unique normalizations}. The weak 1-summability
is proved in thanks to Proposition \ref{prop: Weak sectorial normalizations}.
\end{proof}
\bibliographystyle{alpha}
\bibliography{references_preprint_AIF}

\end{document}